\pdfoutput=1
\documentclass{amsart}

\usepackage[utf8]{inputenc}
\usepackage[T1]{fontenc}
\usepackage{lmodern}
\usepackage[english]{babel}
\usepackage{bbm, dsfont}
\usepackage{mathrsfs}
\usepackage{mathtools}
\usepackage{amssymb} 
\usepackage{amsthm}
\usepackage{amsfonts}
\usepackage[all,cmtip]{xy}
\usepackage[dvipsnames]{xcolor}
\usepackage{tikz}
\usepackage{graphicx}
\usepackage[font = footnotesize, skip = -2pt]{caption}
\usepackage[inline]{enumitem}

\usepackage{booktabs}
\usepackage{units}

\usepackage{setspace}
\usepackage{esint}

\usepackage{nicefrac}
\usepackage{comment}

\usepackage{framed}
\usepackage[unicode=true, colorlinks = true, linkcolor=WildStrawberry, citecolor=NavyBlue, linktocpage=true, bookmarks=false]{hyperref}

\usepackage{mleftright} 

\usepackage[left=3cm, right=3cm, bottom = 3cm, top = 2.5cm]{geometry}

\usepackage{accents} 

\usepackage[section]{placeins} 

\newtheorem{theorem}{Theorem}
\numberwithin{theorem}{section}
\newtheorem{corollary}[theorem]{Corollary}
\newtheorem{lemma}[theorem]{Lemma}
\newtheorem{proposition}[theorem]{Proposition}

\theoremstyle{remark}
\newtheorem{remark}{Remark}

\theoremstyle{definition}
\newtheorem{definition}{Definition}
\newtheorem{claim}{Claim}

\newcommand{\R}{\mathbb{R}}
\newcommand{\N}{\mathbb{N}}
\newcommand{\Z}{\mathbb{Z}}

\newcommand{\cF}{\mathcal{F}}
\newcommand{\cE}{\mathcal{E}}

\newcommand{\tiling}{\mathcal{T}}
\newcommand{\shapes}{\mathcal{S}}
\newcommand{\centers}{\mathcal{C}}

\newcommand{\finSet}{\mathcal{F}(G)}

\newcommand{\gibbs}{\mathcal{G}(\phi)}

\newcommand{\sigmaAlg}{\mathcal{B}}  
\newcommand{\cB}{\mathcal{B}}  
\def\exp{\mathop{\textrm{\rm exp}}\nolimits} 

\allowdisplaybreaks

\begin{document}

\title[Thermodynamic formalism for amenable groups and countable state spaces]{Thermodynamic formalism for amenable groups and countable state spaces}
\author{Elmer R. Beltrán}
\address{Universidad Nacional Jorge Basadre Grohmamn, Tacna, Per\'u; Universidad Católica del Norte, Departamento de Matemáticas, Antofagasta, Chile}
\email{rusbert.unt@gmail.com}
\author{Rodrigo Bissacot}
\address{Faculty of Mathematics and Computer Science, Nicolaus Copernicus University, Poland; Institute of Mathematics and Statistics, University of S\~{a}o Paulo, Brazil}
\email{rodrigo.bissacot@gmail.com}
\author{Luísa Borsato}
\address{Institute of Mathematics and Statistics, University of S\~{a}o Paulo, Brazil}
\email{luisabborsato@gmail.com}
\author{Raimundo Briceño}
\address{Facultad de Matemáticas, Pontificia Universidad Católica de Chile, Santiago, Chile}
\email{raimundo.briceno@mat.uc.cl}


\maketitle

\begin{abstract}
Given the full shift over a countable state space on a countable amenable group, we develop its thermodynamic formalism. First, we introduce the concept of pressure and, using tiling techniques, prove its existence and further properties such as an infimum rule. Next, we extend the definitions of different notions of Gibbs measures and prove their existence and equivalence, given some regularity and normalization criteria on the potential. 
Finally, we provide a family of potentials that non-trivially satisfy the conditions for having this equivalence and a non-empty range of inverse temperatures where uniqueness holds. 
\end{abstract}

\setcounter{tocdepth}{2}
\tableofcontents

\section{Introduction}

There are two general ways to describe a system composed of many particles: microscopically and macroscopically. The first one makes use of the exact positions of the particles, as well as their local interactions. The second one, in turn, is usually outlined by thermodynamic quantities such as energy and entropy. One could say that statistical mechanics --- originated from the works of Boltzmann \cite{boltzmann1902leccons} and Gibbs \cite{gibbs1902elementary} --- is the bridge between the microscopic and the macroscopic descriptions of this kind of systems. In this connection, \emph{Gibbs measures} are a central object.

It is fair to say that Gibbs measures are at the core of the ``conceptual basis of equilibrium statistical mechanics'' \cite{ruelle-2004-thermo}. Relevant examples are the Ising model, which tries to capture the magnetic properties of certain materials; the hard-core model, that describes the distribution of gas particles in a given environment; among many others \cite{friedli2017statistical,georgii2011gibbs,georgii2001random}. In these cases it is customary to consider that the many particles interacting are infinite, take a value from a state space $A$ (also called \emph{alphabet} when $A$ is countable), and they are disposed in a crystalline structure. This structure and its symmetries are usually represented by a countable group $G$, possibly with some Cayley graph associated with it. A particular case is the \emph{hypercubic $d$-dimensional lattice}, which can be understood as the Cayley graph of the finitely generated abelian group $G = \Z^d$ according to its canonical generators. Then, it is natural to represent an arrangement of particles as an element of the space of configurations $X = A^G$, the \emph{$G$-full shift}. Considering this, one is interested in certain measures $\mu$ in the space $\mathcal{M}(X)$ of Borel probability measures supported on $X$. More specifically, the measures of interest are the ones that describe these kind of systems when they are in thermal equilibrium, where the energy of configurations is given by some potential $\phi: X \to \R$. However, there are many mathematically consistent ways to represent that situation by choosing an appropriate measure $\mu \in \mathcal{M}(X)$ and, as the theory evolved, it drew the attention from different areas of expertise such as probability \cite{rassoul2015course,duminil2017lectures} and ergodic theory \cite{sinai1972gibbs,bowen2008equilibrium}. Consequently, the very concept of Gibbs measure started to develop in more abstract and not always equivalent directions.

We focus mainly on four conceptualizations of the idea of thermal equilibrium, namely, DLR, conformal, Bowen-Gibbs, and equilibrium measures. We now proceed to briefly describe each of them.

Dating back to the 60's,  Dobrushin \cite{dobrushin1968problem,dobrushin1970conditional} and, independently, Lanford and Ruelle \cite{lanford1969observables} proposed a concept of Gibbs measure that extended the usual Boltzmann-Gibbs formalism to the infinite particles setting. Roughly, the idea involved looking for probability distributions compatible with a family of maps --- sometimes called \emph{specification} --- that prescribe conditional distributions inside finite subsets of $G$ given some fixed configuration outside. More specifically, given a collection $\gamma = (\gamma_K)_{K \in \finSet}$ of probability kernels $\gamma_{K}\colon \sigmaAlg \times X  \to [0,1]$, with $\finSet$ the set of finite subsets of $G$ and $\sigmaAlg$ the Borel $\sigma$-algebra, one is interested in finding measures $\mu \in \mathcal{M}(X)$ such that $\mu\gamma_K = \mu$ for every $K \in \finSet$, where $\mu\gamma_K$ is a new measure (a priori, different from $\mu$) obtained from $\mu$ via $\gamma_K$. Those distributions are called \emph{DLR measures} after the above cited authors and they have received considerable attention from both mathematical physicists and probabilists (see, for example, \cite{georgii2011gibbs,georgii2001random,keller1998equilibrium,ruelle-2004-thermo}).

Another rather classical way to define a Gibbs measure, which does not involve conditional distributions, was introduced by Capocaccia in \cite{capocaccia1976definition}. Given a class $\mathcal{E}$ of local homeomorphisms $\tau: X \to X$ and a potential $\phi: X \to \R$, one is interested in measures $\mu$ such that $\frac{d(\mu \circ \tau^{-1})}{d\mu} = \exp(\phi_*^\tau)$ for every $\tau \in \mathcal{E}$, where $\phi_*^\tau: X \to \R$ is a function representing the energy difference between a configuration $x$ and $\tau(x)$ (e.g., see \cite[Definition 5.2.1]{keller1998equilibrium}). This kind of measures fits in the more general context of $(\Psi, \mathcal{R})$-conformal measures explored in \cite{aaronson2007exchangeable}, where $\mathcal{R}$ is a Borel equivalence relation and $\Psi\colon \mathcal{R} \to \R_{+}$ is a measurable function. Then, Capocaccia's measures, that we simply call \emph{conformal measures}, can be recovered by taking a function $\Psi$ related to the given potential and $\mathcal{R}$  
the tail relation in the space of configurations. By considering other particular Borel relations $\mathcal{R}$ and measurable functions $\Psi$, one can recover other relevant notions of conformal measures, such as the ones presented in \cite{denker-urbanski-1991-conformal,petersen1997symmetric,sarig1999thermodynamic}, that are mainly adapted to the one-dimensional setting, i.e., when $G = \Z$ or, considering also semigroups, when $G = \N$.

A third possibility, introduced by Rufus Bowen in a one-dimensional and ergodic theoretical context \cite{bowen2008equilibrium}, is to define Gibbs measures by specifying bounds for the probability of cylindrical events. More concretely, one is interested in the measures $\mu \in \mathcal{M}(X)$ for which there exists constants $C > 0$ and $p \in \R$ such that
$$
C^{-1} \leq \frac{\mu([a_0a_1 \cdots a_{n-1}])}{\exp(\sum_{i=0}^{n-1}\phi(T^i x) - pn)} \leq C \qquad \text{ for } x \in X.
$$
As in \cite{berghout2019relation}, we call those measures \textit{Bowen-Gibbs measures} to avoid confusion. This definition has been considered in the literature \cite{chazottes2004entropy,haydn1992equivalence,keller1998equilibrium,ruelle-2004-thermo} and also relaxed versions of it, such as the so-called \textit{weak Gibbs measures} \cite{varandas2017weak,yuri1998zeta}, where the constant $C$ is replaced by a function that grows sublinearly in $n$. This and further relaxations have also played a relevant role in the multi-dimensional case, this is to say, when $G = \Z^d$ and $d > 1$, for finite state spaces (e.g., see \cite[Theorem 5.2.4]{keller1998equilibrium}).

The last important definition considered in this work is the one of \emph{equilibrium measure}. When $X$ is a finite configuration space, equilibrium measures are simply probability vectors that maximize the sum (or difference) of an entropy- and an energy-like quantity, that is, a quantity like $$H(p) + p \cdot (\phi(x_1),\dots,\phi(x_k)) = -\sum_{i=1}^{k}p_i \log p_i + \sum_{i=1}^{k} p_i\phi(x_i),$$ where $k = |X|$, $x_i \in X$, $\phi: X \to \R$ is a potential, $p = (p_1,\dots,p_k)$ is a probability vector with $p_i$ the probability associated with $x_i$, and $H(p)$ is the \emph{Shannon entropy} of $p$. These measures were considered, for example, in \cite{georgii2001random,keller1998equilibrium, ruelle-2004-thermo}. 
On the other hand, when $X$ is an infinite configuration space and there is a robust notion of specific entropy, let's say $h(\mu)$, we are interested in studying measures $\mu \in \mathcal{M}(X)$ that maximize the quantity $h(\mu) + \int \phi \, d\mu$ for a continuous potential $\phi: X \to \R$. This notion tries to capture the macroscopic behaviour of the system without making any assumption of the microscopic structure.

The problem of proving equivalences among these and other related notions has already been studied in different settings. We mention some relevant results that can be found in the literature.

In the one-dimensional case, for finite state spaces, Meyerovitch \cite{meyerovitch-2013-gibbs-eqm} proved the equivalence between conformal measures and DLR measures for some families of proper subshifts. Also, Sarig \cite[Theorem 3.6]{sarig2009notes} proved that any DLR measure on a mixing subshift of finite type is a conformal measure, for a different but related notion of conformal, restricted to the one-dimensional setting. In the same work, for one-sided and countably infinite state spaces, Sarig \cite[Proposition 2.2]{sarig2009notes} proved that conformal measures --- according to his definition --- are DLR measures for topological Markov shifts. In this same setting, Mauldin and Urba\'nski \cite{mauldin2001gibbs} proved the existence of equilibrium measures and that any equilibrium measure satisfies a Bowen-Gibbs equation. Moreover, if the topological Markov shift satisfies the BIP property and the potential has summable variation, Beltr\'an, Bissacot, and Endo \cite{Beltran_2021} proved that DLR measures and conformal measures --- in the same sense as Sarig --- are equivalent. Finally, for potentials with summable variation on sofic subshifts, Borsato and MacDonald \cite{borsato2020dobrushin} proved the equivalence between DLR and equilibrium measures.  There are also other classes of measures in the one-dimensional case which we do not treat here, such as \emph{$g$-measures} \cite{keane1972strongly,walters1975} and \emph{eigenmeasures} associated with the Ruelle operator \cite{bowen2008equilibrium, ruelle-2004-thermo}. When the state space is finite, it is known that the set of DLR measures and $g$-measures do not contain each other \cite{10.1214/ECP.v16-1681,Bissavanenter}, but there is a characterization for when a \emph{$g$-measure} is a DLR measure \cite{berghout2019relation}. In addition, eigenmeasures coincide with DLR measures for continuous potentials in the one-sided setting, as proven by Cioletti, Lopes, and Stadlbauer in \cite{cioletti2020ruelle}. Pioneering works in the one-dimensional countably infinite state space setting can be found in \cite{gurevich1970shift,gurevich1998thermodynamic}.

In the multi-dimensional case, some results regarding the equivalences among the four notions of Gibbs measures have been proved for finite state spaces. A first important reference is Keller \cite[Theorem 5.2.4 and Theorem 5.3.1]{keller1998equilibrium}, where it is proven that when $\phi: X \to \R$ is \emph{regular} (which includes the case of local and H\"older potentials, and well-behaved interactions), then the four definitions are equivalent. Here, by regular, we mean that
$$
\sum_{n=1}^\infty n^{d-1} \delta_n(\phi) < \infty,
$$
where $\delta_n(\phi)$ is the oscillation of $\phi$ when considering configurations that coincide in a specific finite box, namely, $[-n,n]^d \cap \Z^d$. Other classical references in this setting are due to Dobrushin \cite{dobruschin1968description} and Lanford and Ruelle \cite{lanford1969observables}, which, combined, establish the equivalence between DLR measures and equilibrium measures for a general class of subshifts of finite type. Kimura \cite{kimura-2015-thesis} generalized the equivalence between DLR and conformal measures for subshifts of finite type, and some of the implications are true for more general proper subshifts. In the countably infinite state space setting, Muir \cite{muir2011gibbs,muir2011paper} obtained all equivalences for the $G$-full shift when $G=\Z^d$. In order to do this, it was required that the potential $\phi: X \to \R$ is regular and satisfies a normalization criterion, namely, \emph{exp-summability}: 
$$
\sum_{a \in \N} \exp \left(\sup \phi([a])\right) < \infty.
$$
This last condition is automatically satisfied when $A$ is finite.

Results proving equivalences between different kinds of Gibbs measures go beyond the amenable \cite{shriver2020free,barbieri2023lanford,alpeev2016entropy,briceno2021kieffer} and even the symbolic setting to general topological dynamical systems \cite{baladi1991gibbs,haydn1992equivalence}.

One of our main contributions is to exhibit conditions to guarantee that the four notions of Gibbs measures presented above are equivalent, when considering the state space $A=\mathbb{N}$ and an arbitrary countable amenable group $G$, thus extending Muir's methods to the more general amenable case. Countable amenable groups play a fundamental role in ergodic theory \cite{ornstein1987} and include many relevant classes of groups, such as abelian (so, in particular, $G = \Z^d$), nilpotent, and solvable groups and are closed under many natural operations, namely, products, extensions, etc. (e.g., see \cite{loh2017geometric}). In the more general group and finite state space setting, the equivalence between DLR and conformal measures was extended to general subshifts over a countable discrete group $G$ with a special growth property by Borsato-MacDonald \cite[Theorems 5 and Theorem 6]{borsato2021conformal}. Recently, a different proof for the equivalence between DLR and conformal measures for any proper subshift was given by Pfister in \cite{pfister}. Also, in \cite{barbieri2020equivalence}, a Dobrushin-Lanford-Ruelle type theorem is proven in the case that the group is amenable and a topological Markov property holds, which is satisfied, in particular, by subshifts of finite type. Here, as Muir, we focus on the $G$-full shift case. We consider the configuration space $X = \N^G$, for $G$ an arbitrary countable amenable group, and an exp-summable potential $\phi: X \to \R$ with \emph{summable variation (according to some exhausting sequence)}. The concept of summable variation extends the one of regular potential presented before. More precisely, a potential $\phi$ has summable variation if 
$$\sum_{m = 1}^{\infty} \left|E_{m+1}^{-1} \setminus E_{m}^{-1} \right|\delta_{E_m} (\phi) < \infty,$$
where $\{E_m\}_m$ is an exhausting sequence for $G$ and $\delta_{E_m}(\phi)$ is a standard generalization of $\delta_m(\phi)$.

The paper is organized as follows. First, in Section \ref{section:preliminaries}, we present some preliminary notions about amenable groups $G$, the corresponding symbolic space $\N^G$, and potentials. Later, in Section \ref{section:pressure}, we introduce the concept of pressure in our framework and we prove its existence. Also, we prove that it satisfies an infimum rule and that it can be obtained as the supremum of the pressures associated with finite alphabet subsystems. In order to achieve this, we use relatively new techniques for tilings of amenable groups \cite{downarowicz2019tilings} and, inspired by ideas for entropy from \cite{downarowicz2016shearer}, we develop a generalization of Shearer's inequality for pressure. In Section \ref{section:permutations}, we introduce spaces of permutations and Gibbsian specifications in order to pave the way for the definitions of conformal and DLR measures, respectively. Next, in Section \ref{section:main-theorem}, we prove the equivalence between the four notions of Gibbs measures mentioned above given some conditions on the potential, such as exp-summability and summable variation. We also prove related results involving equilibrium measures. In order to prove the equivalence between DLR and conformal measures we rely on the strategies presented on \cite{muir2011gibbs} for the $G = \mathbb{Z}^d$ case, which already considers an infinite state space. Moreover, using Prokhorov's Theorem and relying on the existence of conformal measures in the compact setting \cite{denker-urbanski-1991-conformal}, we prove the existence of a conformal (and DLR) measure in our context. We also prove that DLR measures are Bowen-Gibbs. If it is also the case that the measure is invariant under shift actions of the group, we prove that any Bowen-Gibbs measure is an equilibrium measure and that any equilibrium measure is a DLR measure. At last, in Section \ref{section:examples}, we show how to recover previous results from ours and, inspired by the Potts model and considering a version of it with countably many states, we exhibit a family of examples for which all our results apply non-trivially and, in addition, a version of Dobrushin's Uniqueness Theorem adapted to our setting holds, thus providing a regime where the uniqueness of a Gibbs measure is satisfied.

\section{Preliminaries}
\label{section:preliminaries}

\subsection{Amenable groups and the space $\N^G$}

Let $G$ be a countable discrete group with identity element $1_G$ and $\N$ be the set of non-negative integers. Consider the {\bf $G$-full shift} over $\N$, that is, the set $\N^G = \{x: G \to \N\}$ of $\N$-colorings of $G$, endowed with the product topology. We abbreviate the set $\N^G$ simply by $X$. Given a set $A$, denote by $\mathcal{F}(A)$ the set of nonempty finite subsets of $A$.

Consider a sequence $\{E_m\}_m$ of finite sets of $G$ such that $E_0 = \emptyset$, $1_G \in E_1$, $E_m \subseteq E_{m+1}$ for all $m \in \N$, and $\bigcup_{m \in \N} E_m = G$. We will call such a sequence an {\bf exhaustion of $G$} or an {\bf exhausting sequence for $G$}.  Throughout this paper, we will consider a particular type of exhausting sequences: we will assume further that $E_1 = \{1_G\}$ and $\{E_m\}_m$ strictly increasing.

Given a fixed exhaustion $\{E_m\}_m$, the topology of $X$ is metrizable by the metric $d\colon X \times X \to \R$ given by $d(x, y) = 2^{-\inf \left\{m \in \N \,:\, x_{E_m} \neq y_{E_m} \right\}}$, where $x_F$ denotes the restriction of a configuration $x$ to a set $F \subseteq G$. Denote by $X_F = \{x_F: x \in X\}$ the set of restrictions of $x \in X$ to $F$. The sets of the form $[w] = \{ x \in X \colon x_{F} = w \}$, for $w \in X_F$, $F \in \finSet$, are called {\bf cylinder sets}. The family of such sets is the standard basis for the product topology of $\N^G$. 

Let $\sigmaAlg$ be the $\sigma$-algebra generated by the cylinder sets and let $\mathcal{M}(X)$ be the space of probability measures on $X$. Consider also $\mathcal{M}_G(X)$ the subspace of $G$-invariant probability on X.

The group $G$ acts by translations on $X$ as follows: for every $x \in X$ and every $g,h \in G$,
\begin{equation*}
    (g \cdot x)(h) = x(hg).
\end{equation*}
This action is also referred, in the literature, as the shift action. Moreover, it can be verified that $g \cdot [x_F] = [(g \cdot x)_{Fg^{-1}}]$, for every $x \in X$, $g \in G$, and $F \subseteq G$.

Given $K,F \in \finSet$ and $\delta> 0$, we say that $F$ is {\bf $(K, \delta)$-invariant} if $|K F \Delta F|<\delta|F|$, where $KF = \{kf: k \in K, f\in F\}$. A group $G$ is called {\bf amenable} if for every $K \in \finSet$ and $\delta > 0$, there exists a $(K, \delta)$-invariant set $F$. 

For $K,F \in \finSet$, define: \begin{itemize}
\item[$i)$] the {\bf $K$-interior of $F$} as $\mathrm{Int}_K(F) = \{g \in G\colon Kg \subseteq F\}$,
\item[$ii)$] the {\bf $K$-exterior of $F$} as $\mathrm{Ext}_K(F) = \{g \in G\colon Kg \subseteq G \setminus F\}$, and
\item[$iii)$] the {\bf $K$-boundary of $F$} as $\partial_K(F) = \{g \in G\colon Kg \cap F \neq \emptyset, Kg \cap F^c \neq \emptyset\}$.
\end{itemize}

\subsection{Potentials and variations}

A function $\phi\colon X \to \R$ is called a {\bf potential}. Given $E \subseteq G$, the {\bf variation of $\phi$ on $E$} is given by
\begin{equation*}
 \delta_E(\phi):= \sup \{|\phi(x) - \phi(y)|\colon x_{E} = y_{E}\}.
\end{equation*}

Notice that if $E \subseteq E'$, then $ \delta_{E'}(\phi) \leq \delta_E(\phi)$. If $E = \{1_G\}$, we denote $\delta_E(\phi)$ simply by $\delta(\phi)$.
We say that $\phi$ has {\bf finite oscillation} if $\delta(\phi) < \infty$.

Let $\{E_m\}_m$ be an exhausting sequence for $G$. Given a potential $\phi\colon X \to \R$, it is not difficult to see that $\phi$ is uniformly continuous if and only if $\lim_{m \to \infty} \delta_{E_m}(\phi) = 0$. In this context, given $F \in \finSet$, we define the {\bf $F$-sum of variations of $\phi$ (according to $\{E_m\}_m$)} as
\begin{equation*}
V_F(\phi) := \sum_{m \geq 1} \left|E_{m+1}^{-1}F \setminus E_{m}^{-1}F\right|\cdot\delta_{E_m} (\phi).
\end{equation*}
If $F = \{1_G\}$, we denote $V_F(\phi)$ simply by $V(\phi)$. We say that $\phi\colon X \to \R$ has {\bf summable variation (according to $\{E_m\}_m$)} if $V(\phi) < \infty$.

\begin{remark}\label{rmk:partition-of-G}
For any exhausting sequence $\{E_m\}_m$ and any $F \in \finSet$, the sequence $\{E_{m+1}^{-1}F \setminus E_{m}^{-1}F\}_m$ is a partition of $G$. Moreover, $E_{m+1}^{-1}F \setminus E_{m}^{-1}F \subseteq (E_{m+1}^{-1} \setminus E_{m}^{-1}) F$, so
\begin{equation*}
    \left|E_{m+1}^{-1}F \setminus E_{m}^{-1}F\right| \leq \left|E_{m+1}^{-1} \setminus E_{m}^{-1}\right| \left|F\right|,
\end{equation*}
and $V_F(\phi) \leq V(\phi)  |F|$. In particular, if $\phi$ has summable variation, $V_{F}(\phi) < \infty$ for all $F \in \finSet$.
\end{remark}

\begin{proposition}\label{prop:finite-F-variation-implies-uniformly-continuous} Let $\phi\colon X \to \R$ be a potential such that the $F$-sum of variation of $\phi$ is finite for some $F \in \finSet$. Then $\phi$ is a uniformly continuous potential.
\end{proposition}

\begin{proof} Let $\{E_m\}_m$ be an exhausting sequence for $G$. Since, in particular, $E_m \subseteq E_{m+1}$ for every $m \geq 1$, we have that $0 \leq \delta_{E_{m+1}}(\phi) \leq \delta_{E_m}(\phi)$ for every $m \geq 1$. Then, for every $M \geq 1$,
\begin{align*}
    V_F(\phi) &\geq \sum_{m=1}^{M} |E_{m+1}^{-1}F \setminus E_{m}^{-1}F|   \cdot\delta_{E_m} (\phi)\\
    &\geq \delta_{E_M} (\phi)  \cdot \sum_{m=1}^{M} |E_{m+1}^{-1}F \setminus E_{m}^{-1}F| \\
    &= \delta_{E_M}(\phi)  |E^{-1}_{M+1}F \setminus E_1^{-1}F|,
\end{align*}
where the last line follows from Remark \ref{rmk:partition-of-G}. Therefore, 
$$0 \leq \lim_{M \to \infty} \delta_{E_M}(\phi) \leq \lim_{M \to \infty} \frac{V_F(\phi)}{|E^{-1}_{M+1}F \setminus E_1^{-1}F|} = 0,$$ 
and the result follows.
\end{proof}

\begin{definition}\label{defn:limit-as-more-invariant}
Let $\varphi\colon \finSet \to \R$ be a function. Given $L \in \R$, we say that {\bf $\varphi(F)$ converges to $L$ as $F$ becomes more and more invariant} if for every $\epsilon > 0$ there exist $K \in \finSet$ and $\delta > 0$ such that $|\varphi(F) - L| < \epsilon$ for every $(K, \delta)$-invariant set $F \in \finSet$. We will abbreviate such a fact as $\displaystyle\lim_{F \to G} \varphi(F) = L$.
\end{definition}

A sequence $\{F_n\}_n$ in $\finSet$ is {\bf (left) F{\o}lner} for $G$ if
$$
\lim_{n \to \infty}  \frac{|gF_n \setminus F_n|}{|F_n|} = 0, \text{ for any $g \in G$}.
$$

For example, if $G = \Z^d$ and $F_n = [-n,n]^d \cap \Z^d$, then $\{F_n\}_n$ is a F{\o}lner sequence for $\Z^d$. It is not difficult to see that if $\lim_{F \to G} \varphi(F) = L$, then $\lim_{n \to \infty} \varphi(F_n) = L$ for every F{\o}lner sequence $\{F_n\}_n$. In particular, when $G = \Z^d$, convergence as $F$ becomes more and more invariant implies convergence along $d$-dimensional boxes, which is a common condition in the multi-dimensional framework. It is not difficult to see that a group is amenable if and only if it has F{\o}lner sequence. Moreover, for every amenable group $G$ there exists a F{\o}lner sequence that is also an exhaustion.

\begin{proposition}\label{prop:limit-of-VF-over-the-size-of-F}  Let $\phi\colon X \to \R$ be a potential with summable variation according to an exhausting sequence $\{E_m\}_m$. Then,
\begin{equation*}
\lim_{F \to G} \frac{V_F(\phi)}{|F|} = 0.
\end{equation*}
\end{proposition}

\begin{proof} Let $\epsilon > 0 $. Since $\phi\colon X \to \R$ has summable variation, there exists $m_0 \geq 1$ such that
$$
\sum_{m > m_0} |E_{m+1}^{-1} \setminus E_{m}^{-1}| \cdot  \delta_{E_m}(\phi) < \epsilon.
$$

Then, for every $F \in \finSet$,
\begin{align*}
V_F(\phi)   &\leq \sum_{m=1}^{m_0} |E_{m+1}^{-1}F \setminus E_{m}^{-1}F| \cdot  \delta_{E_m}(\phi) + \sum_{m > m_0} |E_{m+1}^{-1} \setminus E_{m}^{-1}||F|\cdot  \delta_{E_m}(\phi)  \\
            &   \leq \sum_{m=1}^{m_0} |E_{m+1}^{-1}F \setminus E_{m}^{-1}F| \cdot  \delta_{E_m}(\phi) + |F|\cdot\epsilon.
\end{align*}

Due to the amenability of $G$, for any given $m_0 \geq 1$, we have that, for all $m \leq m_0$,
$$
|F| \leq |E_{m+1}^{-1}F| \leq |E_{m_0+1}^{-1}F| \leq (1+\epsilon)|F|
$$
for every $(E_{m_0+1},\epsilon)$-invariant set $F$. 
Therefore, for every $\epsilon > 0$, there exists $m_0 \geq 1$ and $K \in \finSet$ such that for every $(K,\epsilon)$-invariant set $F$,
$$
V_F(\phi) \leq \sum_{m=1}^{m_0} ((1+\epsilon)|F| - |F|) \cdot  \delta_{E_m}(\phi) + \epsilon\cdot|F| = \epsilon\cdot|F|\sum_{m=1}^{m_0}\delta_{E_m}(\phi) + \epsilon\cdot|F|,
$$
so
$$
\frac{V_F(\phi)}{|F|} \leq \epsilon\cdot C,
$$
where $C = 1 + V(\phi)$. Since $\epsilon$ was arbitrary, we conclude.
\end{proof}

Given a potential $\phi\colon X \to \R$, for each $F \in \finSet$, define $\phi_{F}\colon X \to \R$ as $\phi_{F}(x) = \sum_{g \in F} \phi(g \cdot x)$ and $\Delta_{F}(\phi) = \delta_{F}(\phi_{F})$. Notice that $\Delta_{Fg}(\phi) = \Delta_{F}(\phi)$ for every $g \in G$.

\begin{lemma}
\label{regularityDeltak}  Let $\{E_m\}_m$ be an exhausting sequence for $G$, $\phi\colon X \to \R$ be a potential that has finite oscillation and such that $\liminf_{m \to \infty} \delta_{E_m}(\phi) = 0$. Then,
\begin{equation}\label{eq:limit-of-Delta-F-divided-by-the=size-of-F}
\lim_{F \to G} \frac{\Delta_{F}(\phi)}{|F|} = 0.    
\end{equation}
In particular, if $\phi$ has summable variation according to an exhausting sequence $\{E_m\}_m$, then equation \eqref{eq:limit-of-Delta-F-divided-by-the=size-of-F} holds.
\end{lemma}

\begin{proof} 
Let $\epsilon > 0$. Since $\liminf_{m \to \infty} \delta_{E_m}(\phi) = 0$, there exists $m_{0} \geq 1$ such that $\delta_{E_{m_0}}(\phi) \leq \epsilon$. Denote $E_{m_0}$ by $K$.
Due to amenability, we can find $K' \supseteq K$ and $0 < \epsilon' \leq \epsilon$ such that if $F$ is $(K',\epsilon')$-invariant, we have that 
\begin{equation*}
    \left|\mathrm{Int}_K(F) \triangle F\right| < \epsilon\cdot|F|.
\end{equation*}

Considering this, if $x, y \in X$ are such that $x_F = y_F$, we have that
\begin{align*}
|\phi_F(x) - \phi_F(y)|		&	\leq	\sum_{g \in F} |\phi(g \cdot x)  - \phi(g \cdot y) |\\
  					&	=	\sum_{g \in \mathrm{Int}_{K}(F) \cap F} |\phi(g \cdot x)  - \phi(g \cdot y) | + \sum_{g \in F \setminus \mathrm{Int}_{K}(F)} |\phi(g \cdot x)  - \phi(g \cdot y) |\\
					&	\leq	\sum_{g \in \mathrm{Int}_{K}(F) \cap F} \delta_{K}(\phi) + \sum_{g \in F \setminus \mathrm{Int}_{K}(F)} \delta(\phi)\\
					&	\leq	|\mathrm{Int}_{K}(F)|\cdot\epsilon + |F \setminus \mathrm{Int}_{K}(F)| \cdot\delta(\phi)\\
					&	\leq |F| \cdot  (1+\epsilon) \cdot \epsilon + |F|\cdot  \epsilon \cdot \delta(\phi)	\\
					&	=	|F| \cdot \epsilon \cdot (1+\epsilon+\delta(\phi)),
\end{align*}
and the result follows.    
\end{proof}

\section{Pressure}
\label{section:pressure}

We dedicate this section to introduce the \emph{pressure} of a potential. We define and work on the setting of exp-summable potentials with summable variation on a countable alphabet. The pressure --- basically equivalent to the \emph{specific Gibbs free energy} --- is a very relevant thermodynamic quantity that helps to capture the concept of Gibbs measure in a quantitative way.
    
First, we prove that the pressure, which we define through a limit over sets that are becoming more and more invariant, exists in the finite alphabet case. The definition of the pressure is often done in terms of a particular F\o lner sequence, which is an, \textit{a priori}, less robust and less overarching
approach. Existence of the limit for a particular F\o lner sequence $\{F_n\}_n$ and the fact that it is independent on the choice of such sequence is well-known (see, for example, \cite{tempelman2013ergodic,gurevich2007breiman,bufetov2011}, in the context of absolutely summable interactions). Here, we prove something stronger: that our definition of pressure obeys the infimum rule --- which is a refinement of the Ornstein-Weiss Lemma (see, for example, \cite[\S 4.5]{kerr2016ergodic}) ---, this is to say, it can be expressed as an infimum over all finite sets of $G$. In order to conclude this, we extend the results about Shearer's inequality in \cite{downarowicz2016shearer} for topological entropy to pressure.

Now, in the countable alphabet context, we take a similar approach. First, we consider again a definition of pressure in terms of sets that are becoming more and more invariant. Next, we prove that the infimum rule still holds and, finally, we prove that the pressure can be obtained as the supremum of the pressures associated with finite alphabet subsystems. A related result was obtained by Muir in \cite{muir2011gibbs} for the $\Z^d$ group case, where the pressure was defined as a limit over a particular type of F{\o}lner sequence, namely, open boxes centered at the origin of radius $n$. The existence of this limit was proven through a sub-additivity argument that exploits the property that large boxes can be partitioned into many equally sized ones, which might not be valid in more general groups. In order to generalize this idea of partitioning sets, we make use of tiling techniques introduced in \cite{downarowicz2019tilings}, which, together to what is done in the finite alphabet case, allow us to prove the infimum rule for infinite alphabets over a countable amenable group. This type of result was not considered in \cite{muir2011gibbs}.

We begin by introducing some definitions.  Given a potential $\phi\colon X \to \R$ and $F \in \finSet$, define the {\bf partition function for $\phi$ on $F$ } as
\begin{equation*}
    Z_F(\phi):= \sum_{w \in X_F} \exp \left(\sup \phi_F([w])\right),
\end{equation*}
where  $\sup\phi_F([w]) = \sup\{\phi_F(x): x \in [w]\}$. We define the {\bf pressure of $\phi$}, which we denote by $p(\phi)$, as 
$$
p(\phi) := \lim_{F \to G} \frac{1}{|F|} \log Z_F(\phi),
$$
whenever such limit as $F$ becomes more and more invariant exists. In addition, given a finite subset $A \in \cF(\N)$, we define $Z_F(A,\phi)$ as the partition function associated with the restriction of $\phi$ to $A^G$. More precisely, 
\begin{equation*}
    Z_F(A,\phi):= \sum_{w \in X_F \cap A^F} \exp \sup\left( \phi_F\left([w] \cap A^G\right)\right).
\end{equation*}

Similarly, we define $p(A,\phi)$ as
\begin{equation*}
    p(A,\phi):= \lim_{F \to G} \frac{1}{|F|} \log Z_F(A,\phi),
\end{equation*}
whenever such limit exists.

\subsection{Infimum rule for finite alphabet pressure}

The main goal of this subsection is to prove the following theorem.

\begin{theorem}
\label{thm:continuous-potential-finite-alphabet-then-pressure-exists}
Let $\phi\colon X \to \R$ be a continuous potential. Then, for any finite alphabet $A \subseteq \N$, $p(A,\phi)$ exists and
$$
p(A,\phi)  = \inf_{E \in \finSet} \frac{1}{|E|}\log Z_E(A,\phi).
$$
\end{theorem}

In order to prove this result, we require some definitions. A function $\varphi\colon \finSet \rightarrow \R$ is 
\begin{itemize}
\item {\bf $G$-invariant} if $\varphi(Fg) = \varphi(F)$ for every $F \in \finSet$ and $g \in G$;
\item {\bf monotone} if $\varphi(E) \leq \varphi(F)$ for every $E,F \in \finSet$ such that $E \subseteq F$; and
\item {\bf sub-additive} if $\varphi(E \cup F) \leq \varphi(E) + \varphi(F)$ for any $E,F \in \finSet$.
\end{itemize}

A {\bf $k$-cover} $\mathcal{K}$ of a set $F \in \finSet$ is a family $\left\{K_{1}, K_{2}, \dots, K_{r}\right\} \subseteq \finSet$ (with possible repetitions) such that each element of $F$ belongs to $K_{i}$ for at least $k$ indices $i \in \{1,\dots, r\}$. We say that $\varphi$ satisfies {\bf Shearer's inequality} if for any $F \in \finSet$ and any $k$-cover $\mathcal{K}$ of ${F}$, it holds that
$$
\varphi({F}) \leq \frac{1}{k} \sum_{K \in \mathcal{K}} \varphi(K).
$$

Notice that Shearer's inequality implies sub-additivity. Considering this, we have the following key lemma.

\begin{lemma}[{\cite[\S 4]{kerr2016ergodic}}] Let $\varphi\colon \finSet \rightarrow \R$ be a non-negative monotone $G$-invariant sub-additive function. Then there exists $\alpha \in [0,\infty)$ such that
\begin{equation*}
\lim_{F \to G} \frac{\varphi(F)}{|F|} = \alpha.    
\end{equation*}

Moreover, if $\varphi$ satisfies Shearer's inequality, then
\begin{equation*}
\alpha = \inf_{E \in \finSet} \frac{\varphi(E)}{|E|}.    
\end{equation*}

In this last case, we say that $\varphi$ satisfies the {\bf infimum rule}.
\end{lemma}

Now, fix a finite alphabet $A \in \mathcal{F}(\N)$. For a continuous potential $\phi\colon X \to \R$, we denote by $\|\phi\|_A$ the supremum norm of $\phi$ over the compact set $X \cap A^G$, i.e., $\|\phi\|_A = \sup_{x \in X \cap A^G} |\phi(x)|$. Next, given a set $E \subseteq G$, $F \in \finSet$, and $u_E \in X_E \cap A^E$, we define
\begin{equation*}
    Z^{u_E}_F := \sum_{w_{F \setminus E} \in A^{F\setminus E}} \exp\left(\sup \phi_F([w_{F \setminus E}u_E])\right),
\end{equation*}
where the supremum is over $x \in X \cap A^G$ and, if $[v] = \emptyset$, then $\sup \phi([v]) = -\infty$ and $\exp(-\infty) = 0$. Notice that $Z_F = Z^{u_E}_F$ for $E = \emptyset$.

Now, suppose that $\left.\phi\right\vert_{X \cap A^G}$ is non-negative. Then, it is easy to check that for any $E \subseteq G$ and $u_E \in A^E$, the function $\tilde{\varphi}\colon \finSet \to \R$ given by $\tilde{\varphi}(F) =  Z^{u_E}_F$ satisfies that
\begin{enumerate}
    \item[$i)$] $\tilde{\varphi}(F) \geq 1$ for every $F \in \finSet$ and
    \item[$ii)$] $\tilde{\varphi}$ is monotone, that is, if $F_1 \subseteq F_2$, then $\tilde{\varphi}(F_1) \leq \tilde{\varphi}(F_2)$.
\end{enumerate}

Next, consider the function $\varphi\colon \finSet \to \R$ defined as $\varphi(F) = \log Z_F$. From the properties above and properties of the $\log(\cdot)$ function, it follows that $\varphi$ is non-negative and monotone. Moreover, $\varphi$ is $G$-invariant. The following lemma is a generalization of \cite[Lemma 6.1]{downarowicz2016shearer} designed to address the pressure case instead of just the topological entropy and, in particular, it claims that $\varphi$ satisfies Shearer's inequality.

\begin{lemma}\label{lemma:shearers-inequality-is-satisfied} Let $\phi\colon X \to \R$ be a potential and $A \in \mathcal{F}(\N)$ such that $\left.\phi\right\vert_{X \cap A^G}$ is non-negative. Then, for every $E \subseteq G$, $u_E \in X_E \cap A^E$, $F \in \mathcal{F}(G)$, and any $k$-cover $\mathcal{K}$ of $F$, it holds that
$$
Z^{u_E}_F \leq \prod_{K \in \mathcal{K}} (Z^{u_E}_K)^{1/k}.
$$

In particular, $\varphi$ satisfies Shearer's inequality.
\end{lemma}

\begin{proof}
Given a $k$-cover $\mathcal{K}$ of $F$, notice that, since $\left.\phi\right\vert_{X \cap A^G}$ is non-negative,
$$
\phi_F(x) = \sum_{g \in F} \phi(g \cdot x) \leq \frac{1}{k}\sum_{K \in \mathcal{K}} \sum_{g \in K} \phi(g \cdot x) = \frac{1}{k}\sum_{K \in \mathcal{K}} \phi_K(x)
$$
for any $x \in X \cap A^G$. We proceed by induction on the size of $F \setminus E$. First, suppose that $|F \setminus E| = 0$. Then, $F \setminus E = \emptyset$ and
\begin{align*}
Z^{u_E}_F   &   =   \exp \left(\sup \phi_F([u_E])\right) \\
            &   \leq \exp\left( \sup \frac{1}{k}\sum_{K \in \mathcal{K}} \phi_K([u_E])\right) \\
            &   \leq \exp\left( \sum_{K \in \mathcal{K}}\frac{1}{k} \sup  \phi_K([u_E]) \right) \\
            &   =   \prod_{K \in \mathcal{K}} \left(\exp \sup \phi_K ([u_E])\right)^{1/k}    \\
            &   \leq \prod_{K \in \mathcal{K}} \left(\sum_{w_{K \setminus E}}\exp \sup \phi_K ([w_{K \setminus E}u_E])\right)^{1/k}    \\
            &   =   \prod_{K \in \mathcal{K}} \left(Z^{u_E}_K\right)^{1/k}.
\end{align*}

Now suppose that $Z^{u_E}_F \leq \prod_{K \in \mathcal{K}} (Z^{u_E}_K)^{1/k}$ for every $E\subseteq G, u_E \in X_E \cap A^E$, $F \in \finSet$ with $|F \setminus E| \leq n$, and every $k$-cover $\mathcal{K}$ of $F$. We will show that the same holds for $E, F$ with $|F \setminus E| = n+1$. Fix $g \in F \setminus E$ and notice that $|F \setminus (E \cup \{g\})| = n$. Then,
\begin{align*}
Z^{u_E}_F   &=   \sum_{a \in A} Z^{a^g u_E}_F\\
    &\leq   \sum_{a \in A}\prod_{K \in \mathcal{K}} \left(Z^{a^gu_E}_K \right)^{1/k}   \\
    &=   \sum_{a \in A}\prod_{K \in \mathcal{K}: g \notin K} \left(Z^{a^gu_E}_K \right)^{1/k}\cdot  \prod_{K \in \mathcal{K}: g \in K} \left(Z^{a^gu_E}_K \right)^{1/k}   \\
            &   \leq \prod_{K \in \mathcal{K}: g \notin K} \left(Z^{u_E}_K \right)^{1/k} \sum_{a \in A} \prod_{K \in \mathcal{K}: g \in K} \left(Z^{a^gu_E}_K \right)^{1/k}   \\
            &   \leq   \prod_{K \in \mathcal{K}: g \notin K} \left(Z^{u_E}_K \right)^{1/k} \cdot \prod_{K \in \mathcal{K}: g \in K} \left(\sum_{a \in A}  Z^{a^gu_E}_K \right)^{1/k}  \\
            &   =   \prod_{K \in \mathcal{K}: g \notin K} \left(Z^{u_E}_K \right)^{1/k}\cdot \prod_{K \in \mathcal{K}: g \in K} \left(Z^{u_E}_K \right)^{1/k}  \\
            &   =   \prod_{K \in \mathcal{K}} \left(Z^{u_E}_K \right)^{1/k}.
\end{align*}

Notice that the first inequality follows from the induction hypothesis and the third inequality follows from the generalized Hölder inequality. Indeed, consider $p \leq 1$ such that $\sum_{K \in \mathcal{K}: g \in K} \frac{1}{k} = \frac{1}{p}$ and the functions $f_K\colon A \to \R$ given by $f_K(a) = \left(Z^{a^gu_E}_K \right)^{1/k}$. By the generalized Hölder inequality,
$$
\left\|\prod_{K \in \mathcal{K}: g \in K} f_K\right\|_p \leq \prod_{K \in \mathcal{K}: g \in K}\|f_K\|_k,
$$
where
\begin{align*}
\prod_{K \in \mathcal{K}: g \in K}\|f_K\|_k &   =   \prod_{K \in \mathcal{K}: g \in K} \left(\sum_{a \in A} ((Z^{a^gu_E}_K)^{1/k})^k \right)^{1/k}  = \prod_{K \in \mathcal{K}: g \in K} \left(\sum_{a \in A} Z^{a^gu_E}_K \right)^{1/k} = \prod_{K \in \mathcal{K}: g \in K} \left(Z^{u_E}_K \right)^{1/k}
\end{align*}
and, since $\|\cdot\|_p$ is monotonically decreasing in $p$ for any fixed $|A|$-dimensional vector, 
$$
\left\|\prod_{K \in \mathcal{K}: g \in K} f_K\right\|_p \geq \left\|\prod_{K \in \mathcal{K}: g \in K} f_K\right\|_1 = \sum_{a \in A}\prod_{K \in \mathcal{K}: g \in K} \left(Z^{a^gu_E}_K \right)^{1/k}.
$$
Therefore, $Z^{u_E}_F \leq \prod_{K \in \mathcal{K}} \left(Z^{u_E}_K \right)^{1/k}$. In particular, if $E = \emptyset$, $Z_F \leq \prod_{K \in \mathcal{K}} \left(Z_K \right)^{1/k}$.
\end{proof}

\begin{proof}[Proof (of Theorem \ref{thm:continuous-potential-finite-alphabet-then-pressure-exists})]

As a consequence of Lemma \ref{lemma:shearers-inequality-is-satisfied}, we have that if $\left.\phi\right\vert_{X \cap A^G}$ is non-negative, then $\varphi$ satisfies Shearer's inequality. Thus, by the Ornstein-Weiss lemma, $p(A,\phi)$ exists and it satisfies the infimum rule, i.e.,
$$
p(A,\phi)  = \inf_{E \in \finSet} \frac{1}{|E|}\log Z_E(A,\phi).
$$

Finally, in order to deal with the general case, it suffices to apply the previous result to $\phi + \|\phi\|$ and then observe that $p(A,\phi + C) = p(A,\phi) + C$ for any constant $C$.
\end{proof}

\begin{remark}
Notice that the previous results (namely, Lemma \ref{lemma:shearers-inequality-is-satisfied} and Theorem \ref{thm:continuous-potential-finite-alphabet-then-pressure-exists}) also hold for \emph{$G$-subshifts}, this is to say, any closed and $G$-invariant subsets $X$ of $\N^G$.
\end{remark}

\subsection{Tilings}

Pressure is one of the most important notions in thermodynamic formalism. One key technique to properly define pressure is sub-additivity, which is based on our ability to partition a system in smaller and representative pieces. In the context of countable amenable groups, it appears to be necessary to generalize tools that are classically used in the $\Z^d$ case (e.g., \cite{ruelle-2004-thermo,muir2011gibbs}). In order to do this, we will begin by exploring the concept of (exact) tilings of amenable groups and the relatively recent techniques introduced in \cite{downarowicz2019tilings}.

\begin{definition}
Given
\begin{enumerate}
\item a finite collection $\shapes(\tiling)$ of finite subsets of $G$ containing the identity $1_G$, called {\bf the shapes}, and
\item a finite collection $\centers(\tiling) = \{C(S): S \in \shapes(\tiling)\}$ of disjoint subsets of $G$, called {\bf center sets} (for the shapes),
\end{enumerate}
the family $\tiling = \{(S, c): S \in \shapes(\tiling), c \in C(S)\}$ is called a {\bf tiling} if the map $(S,c) \mapsto S c$ is injective and $\{Sc\}_{S \in \shapes(\tiling), c \in C(S)}$ is a partition of $G$. In addition, by the {\bf tiles} of $\tiling$ (usually denoted by the letter $T$) we will mean either the sets $S c$ or the pairs $(S, c)$, 
depending on the context.
\end{definition}

We say that a sequence $\{\tiling_n\}_n$ of tilings is {\bf congruent} if, for each $n \geq 1$, every tile of $\tiling_{n+1}$ is equal to a (disjoint) union of tiles of $\tiling_n$. The following theorem is the main result in \cite{downarowicz2019tilings}, which gives sufficient conditions so that we can guarantee the existence of such sequence with extra invariance properties.

\begin{theorem}{(\cite[Theorem 5.2]{downarowicz2019tilings})}
\label{thm:tiling}
Let $\{\epsilon_n\}_n$ be a sequence of positive real numbers converging to zero and $\{K_n\}_n$ be a sequence of finite subsets of $G$. Then, there exists a congruent sequence $\{\tiling_n\}_n$ of tilings  of $G$ such that the shapes of $\tiling_n$ are $\left(K_n, \epsilon_n\right)$-invariant.
\end{theorem}

Given a tiling $\tiling$, we define $S_\tiling = \bigcup_{S \in \shapes(\tiling)} SS^{-1}$. Notice that $S_{\tiling}$ contains every shape $S \in \shapes(\tiling)$, $S_\tiling^{-1} = S_\tiling$, and $1_G \in S_\tiling$. Given a tiling, the next lemma provides a way to approximate any sufficiently invariant shape by a  union of tiles. 

\begin{lemma}\label{cor:invariant} Given $K \in \finSet$ and $\delta > 0$, consider a tiling $\tiling$ with $(K,\delta)$-invariant shapes. Then, for any $\epsilon > 0$ and any $(S_\tiling,\epsilon)$-invariant  set $F \in \finSet$, there exist center sets $C_F(S) \subseteq C(S)$ for $S \in \shapes(\tiling)$ such that
$$
F \supseteq \bigsqcup_{S \in \mathcal{S}(\mathcal{T})} SC_F(S) \quad \text{and} \quad \left|F \setminus \bigsqcup_{S \in \mathcal{S}(\mathcal{T})} SC_F(S)\right| \leq \epsilon|F|.
$$
\end{lemma}

\begin{proof}
Consider a tiling $\tiling$ made of $(K, \delta)$-invariant shapes and $\epsilon > 0$. Suppose that $F$ is $(S_\tiling,\epsilon)$-invariant. Consider the sets $C_F(S) = C(S) \cap \mathrm{Int}_S(F)$ and $\overline{C}_F(S) = C(S) \cap S^{-1}F$ for $S \in \mathcal{S}(\mathcal{T})$. Notice that, since $\tiling$ induces a partition, $|SC_F(S)| = |S||C_F(S)|$, $|S\overline{C}_F(S)| = |S||\overline{C}_F(S)|$, and
$$
\bigsqcup_{S \in \shapes(\tiling)} SC_F(S) \subseteq F \subseteq \bigsqcup_{S \in \shapes(\tiling)} S\overline{C}_F(S).
$$

Therefore,
\begin{align*}
F \setminus \bigsqcup_{S \in \shapes(\tiling)} SC_F(S)   &   \subseteq	\bigsqcup_{S \in \shapes(\tiling)} S\overline{C}_F(S) \setminus \bigsqcup_{S \in \shapes(\tiling)} SC_F(S)   \\
                            &   =		\bigsqcup_{S \in \shapes(\tiling)} S(\overline{C}_F(S) \setminus C_F(S)) \subseteq   \partial_{S_\tiling}(F).
\end{align*}

Indeed, to check the last inclusion, notice that if $g \in \bigsqcup_{S \in \shapes(\tiling)} S(\overline{C}_F(S) \setminus C_F(S))$, then $g = sc$, where $s \in S$ and $c \in \overline{C}_F(S) \setminus C_F(S)$ for some $S \in \shapes(\tiling)$. Therefore, since $c \in \overline{C}_F(S)$,
$$
S_\tiling g \cap F  \supseteq SS^{-1}sc \cap F \supseteq Sc \cap F \neq \emptyset.
$$
Similarly, since $c \notin C_F(S)$,
$$
S_\tiling g \cap F^c  \supseteq SS^{-1}sc \cap F^c \supseteq Sc \cap F^c \neq \emptyset.
$$
so that $g \in \partial_{S_\tiling}(F)$. Then,
$$
\left|F \setminus \bigsqcup_{S \in \mathcal{S}(\mathcal{T})} SC_F(S)\right| \leq \left|\partial_{S_\tiling}(F)\right| \leq \left|S_\tiling F \triangle F\right| \leq \epsilon\cdot|F|,
$$
where we have used that $|\partial_{K}(F)| \leq |(K \cup K^{-1} \cup \{1_G\})F \triangle F|$ for any $K \in \finSet$ and that $S_\tiling^{-1} = S_\tiling$ and $1_G \in S_\tiling$.
\end{proof}

\subsection{Infimum rule for countable alphabet pressure}

We say that $\phi\colon X \to \R$ is {\bf exp-summable} if $Z_{1_G}(\phi) <\infty$. Notice that $Z_F(\phi)$ is sub-multiplicative, that is, if $E, F \in \finSet$ are disjoint, then $Z_{E \cup F}(\Phi) \leq Z_E(\phi) \cdot Z_F(\phi)$. Also, notice that $Z_F(\phi)$ is $G$-invariant, namely, for any $g \in G$, $Z_F(\phi) = Z_{Fg}(\phi)$. Then, in particular, $Z_F(\phi) \leq Z_{1_G}(\phi)^{|F|}$, so $\phi$ is exp-summable if and only if $Z_F(\phi) < \infty$ for every $F \in \finSet$. Finally, observe that if $\phi$ is exp-summable, then it must be bounded from above.  

Before stating the main result of this section, we begin by the next lemma, that guarantees that given a finite shape $F$, one can approximate the partition function on $F$ using a finite alphabet.

\begin{lemma}
\label{lem:appoximate-by-finite-alphabet}
Let $\phi\colon X \to \R$ be an exp-summable and uniformly continuous potential. Then, for every $\epsilon > 0$ and every $F \in \finSet$ such that $|F| \geq -\frac{1}{\epsilon}\log(1-\epsilon)$, there exists $A_F \in \cF(\N)$ such that
\begin{align*}
Z_F(A_F, \phi) \geq (1-\epsilon) Z_F(\phi).
\end{align*}
\end{lemma}

\begin{proof}
Let $\epsilon > 0$ and $F \in \finSet$ be such that $|F| \geq -\frac{1}{\epsilon}\log(1-\epsilon)$. For every such $F$, there exists a finite set of words $W_F \Subset X_F$ such that
\begin{align*}
\sum_{w \in W_F} \exp \left(\sup \phi_{F}\right) \geq Z_{F}(\phi) \sqrt{1-\epsilon}.
\end{align*}
On the other hand, since $\phi$ is uniformly continuous, there must be an index $m \geq 1$ for which
\begin{align*}
\delta_{E_m}(\phi) \leq \frac{1}{3|F|} \log \left(\frac{1}{\sqrt{1-\epsilon}}\right).
\end{align*}

For each $w \in W_F$, pick a word $w' \in \N^{E_mF}$ such that $w'_F=w$ and
\begin{align*}
\sup \phi_{F} [w'] \geq \sup \phi_{F} [w] - \frac{1}{3} \log \left(\frac{1}{\sqrt{1-\epsilon}}\right).
\end{align*}

In addition, for each such $w'$, pick a configuration $x_w \in [w']$ such that
\begin{align*}
\phi_{F}(x_w) \geq \sup \phi_{F}[w'] - \frac{1}{3} \log \left(\frac{1}{\sqrt{1-\epsilon}}\right).
\end{align*}

Define $A_F$ to be $\bigcup_{w \in W_F} w'(E_mF)$, where $w'(E_mF) = \bigcup_{g \in E_mF}\{w'(g)\}$.  It is direct that $A_F$ is a finite subset of $\N$. Pick $y \in [w'] \cap A_F^G$ and notice that $(g \cdot x_w)_{E_m} = (g \cdot y)_{E_m}$ for all $g \in F$. Then, for every $w \in W_F$,
\begin{align*}
\sup \phi_{F}[[w'] \cap A_F^G]	&	\geq	\phi_{F}(y)\\
&	\geq	\phi_{F}(x_w) - \sum_{g \in F}\left|\phi(g \cdot x_w) - \phi(g \cdot y)\right|\\
&	\geq	\phi_{F}(x_w) - \left|F\right| \delta_{E_m}(\phi)\\
&	\geq \phi_{F}(x_w) - \frac{1}{3} \log \left(\frac{1}{\sqrt{1-\epsilon}}\right)\\
&\geq	\sup \phi_{F}[w'] - \frac{2}{3} \log \left(\frac{1}{\sqrt{1-\epsilon}}\right)\\
&	\geq	\sup \phi_{F}[w] - \log \left(\frac{1}{\sqrt{1-\epsilon}}\right).
\end{align*}
Hence,
\begin{align*}
Z_{F}(A_F, \phi)	&	=	\sum_{w \in A_{F}^{F}} \exp \left(\sup \phi_{F}[[w] \cap A_F^G]\right)		\\
				&	\geq	\sum_{w \in W_F} \exp \left(\sup \phi_{F}[[w] \cap A_F^G]\right)				\\
				&	\geq	\sum_{w \in W_F} \exp \left(\sup \phi_{F}[[w'] \cap A_F^G]\right) \\
				&	\geq	\sum_{w \in W_F} \exp \left(\sup \phi_{F}[w]-\log \left(\frac{1}{\sqrt{1-\epsilon}}\right)\right) \\
				&   =   	\sqrt{1-\epsilon}\sum_{w \in W_F} \exp \left(\sup \phi_{F}[w]\right) \\
				&	\geq	(1-\epsilon) Z_{F}(\phi).
\end{align*}
\end{proof}

The next proposition establishes a fundamental connection between the partition function for sufficiently invariant sets $F \in \finSet$ and the pressure for a sufficiently large finite alphabet $A$.

\begin{proposition}
\label{prop:inf-of-log-partition-function}
Let $\phi\colon X \to \R$ be an exp-summable and uniformly continuous potential with finite oscillation. Then, for every $\frac{1}{2} > \epsilon > 0$, there exist $A \in \cF(\N)$, $K \in \finSet$, and $\delta > 0$ such that for every $(K,\delta)$-invariant set $F \in \finSet$, it holds that
\begin{align}
\frac{1}{|F|} \log Z_F(\phi) \leq\inf_{E \in \finSet} \frac{1}{|E|} \log Z_E(A,\phi) + \epsilon.
\end{align}

\end{proposition}

\begin{proof}
Fix $1/2 > \epsilon > 0$ and an exhausting sequence $\{E_m\}_m$ for $G$. Since $\phi$ is uniformly continuous, we have that $\lim_{F \to G} \frac{\Delta_{F}(\phi)}{|F|} = 0$, by Lemma \ref{regularityDeltak}. Therefore, there exist $K' \in \finSet$ and $\delta' > 0$ such that $\Delta_{F}(\phi) < \epsilon|F|$ for every finite $(K',\delta')$-invariant set $F$.

By Theorem \ref{thm:tiling}, there exists a tiling $\tiling'$ such that its shapes are $(K',\delta')$-invariant. Without loss of generality, by possibly readjusting $K'$ and $\delta'$, assume that $|S'| \geq -\frac{1}{\epsilon}\log(1-\epsilon)$ for every $S' \in \shapes(\tiling')$. 
Therefore, by Lemma \ref{lem:appoximate-by-finite-alphabet}, for every $S' \in \shapes(\tiling')$ there exists $A_{S'} \Subset \N$ such that $Z_{S'}(A_{S'}, \phi) \geq (1-\epsilon) Z_{S'}(\phi)$. Define $A$ to be $\bigcup_{S' \in \shapes(\tiling')} A_{S'}$. Then, $A$ is a finite subset of $\N$. Moreover, since $A_{S'} \subseteq A$, for each $S' \in \shapes(\tiling')$, we have that
\begin{equation}\label{eq:partition-funcion-finite-alphabet-and-infinite}
    Z_{S'}(A, \phi) \geq (1-\epsilon) Z_{S'}(\phi),
\end{equation} 
for every $S' \in \shapes(\tiling')$.

Now, by Theorem \ref{thm:continuous-potential-finite-alphabet-then-pressure-exists}, $p(A,\phi) =\lim_{F \to G} \frac{1}{|F|}\log Z_F(A,\phi)$ exists, so we can pick $K \in \finSet$ and $\delta > 0$ such that $K \supseteq K'$, $\delta < \delta'$, and
\begin{equation}\label{eq:inequality-log-partition-function-and-pressure-finite-alphabet}
\log Z_F(A,\phi) \leq |F| (p(A,\phi) + \epsilon)
\end{equation}
for every $(K,\delta)$-invariant set $F \in \finSet$.

Next, by Theorem \ref{thm:tiling}, we can obtain a tiling $\tiling$ of $(K,\delta)$-invariant sets such that every tile in $\tiling$ is a union of tiles in $\tiling'$, i.e.,
$S = \bigsqcup_{S' \in \shapes(\tiling')}\bigsqcup_{c' \in C_{S}(S')} S'c'$. Furthermore, by Lemma \ref{cor:invariant}, for every $(S_{\tiling},\epsilon)$-invariant set $F \in \finSet$, there exist center sets $C_F(S) \subseteq C(S) \in \centers(\tiling)$ for $S \in \shapes(\tiling)$ such that
$$
F \supseteq T_F \quad \text{and} \quad |F \setminus T_F| \leq \epsilon|F|,
$$
where $T_F = \bigsqcup_{S \in \shapes(\tiling)} SC_F(S)$.

Furthermore, for every $S \in \shapes(\tiling)$, we have that
\begin{align*}
Z_S(A,\phi)	&=	\sum_{w_{S} \in A^{S}} \exp \left(\sup\phi_{S} ([w_{S}] \cap A^G)\right)\\
&\geq	\sum_{w_{S} \in A^{S}} \exp\left(\inf\phi_{S} \left([w_{S}] \cap A^G\right)\right)\\
&\geq	\prod_{S' \in \shapes(\tiling')} \prod_{c' \in C_S(S')} \sum_{w_{S'c'} \in A^{S'c'}} \exp \left(\inf\phi_{S'c'} \left([w_{S'c'}] \cap A^G\right)\right)	\\
&	\geq	\prod_{S' \in \shapes(\tiling')} \prod_{c' \in C_{S}(S')} \sum_{w_{S'c'} \in A^{S'c'}} \exp\left(\sup\phi_{S'c'} \left([w_{S'c'}] \cap A^G\right) -  \Delta_{S'c'}(\phi)\right)			\\
&	=	\prod_{S' \in \shapes(\tiling')} \exp\left(-\Delta_{S'}(\phi)|C_{S}(S')|\right) \prod_{c' \in C_{S}(S')} Z_{S'c'}(A,\phi)	\\
	&	=	\prod_{S' \in \shapes(\tiling')} \exp\left(-\Delta_{S'}(\phi)|C_{S}(S')|\right) Z_{S'}(A,\phi)^{|C_{S}(S')|},
\end{align*}
where we used that, for every $g \in G$, $Z_F(A, \phi) = Z_{Fg}(A, \phi)$ and that $\Delta_{F}(\phi) = \Delta_{Fg}(\phi)$. Thus,
\begin{align}\label{eq:inequality-product-of-partition-function}
\prod_{S' \in \shapes(\tiling')} Z_{S'}(A,\phi)^{|C_{S}(S')|} \leq  Z_{S}(A,\phi) \cdot \prod_{S' \in \shapes(\tiling')} \exp\left(|C_{S}(S')|\Delta_{S'}(\phi)\right).
\end{align}

Now, given a $(S_{\tiling},\epsilon)$-invariant set $F \in \finSet$, we have that
\begin{align*}
Z_{T_F}(\phi)	&	\leq		\prod_{S \in \shapes(\tiling)} \prod_{c \in C_F(S)} Z_{Sc}(\phi)	\\
&=		\prod_{S \in \shapes(\tiling)} Z_{S}(\phi)^{|C_F(S)|} \\
&\leq		\prod_{S \in \shapes(\tiling)} \left(\prod_{S' \in \shapes(\tiling')} \prod_{c' \in C_{S}(S')} Z_{S'c'}(\phi)\right)^{|C_F(S)|}	\\
&=		\prod_{S \in \shapes(\tiling)} \prod_{S' \in \shapes(\tiling')} (Z_{S'}(\phi)^{|C_{S}(S')|})^{|C_F(S)|}.
\end{align*}

Therefore, from equation \eqref{eq:partition-funcion-finite-alphabet-and-infinite}, we obtain that
\begin{align*}
	\prod_{S \in \shapes(\tiling)} &\prod_{S' \in \shapes(\tiling')} (Z_{S'}(\phi)^{|C_{S}(S')|})^{|C_F(S)|}\\ 
&	\leq	\prod_{S \in \shapes(\tiling)} \prod_{S' \in \shapes(\tiling')} \left(\frac{1}{1-\epsilon} Z_{S'}(A,\phi)\right)^{|C_{S}(S')||C_F(S)|} 	\\
&	\leq		\left(\frac{1}{1-\epsilon}\right)^{|T_F|} \prod_{S \in \shapes(\tiling)} \left(Z_{S}(A,\phi) \exp\left(\sum_{S' \in \shapes(\tiling')} |C_{S}(S')|\Delta_{S'}(\phi)\right)\right)^{|C_F(S)|} 	\\
&	\leq		\left(\frac{1}{1-\epsilon}\right)^{|F|} \prod_{S \in \shapes(\tiling)} \exp\left(|S|(p(A,\phi) + \epsilon) + \sum_{S' \in \shapes(\tiling')} |C_{S}(S')|\Delta_{S'}(\phi)\right)^{|C_F(S)|},
\end{align*}
where the second inequality follows from equation \eqref{eq:inequality-product-of-partition-function} and the third from equation \eqref{eq:inequality-log-partition-function-and-pressure-finite-alphabet}. Hence, if $0 < \epsilon < \frac{1}{2}$, we have that $\log\left(\frac{1}{1-\epsilon}\right) \leq 2\epsilon$, so
\begin{align*}
\frac{1}{|F|} &\log Z_{T_F}(\phi)\\
&	\leq	\log\left(\frac{1}{1-\epsilon}\right) + \frac{1}{|F|}\sum_{S \in \shapes(\tiling)}|C_F(S)| \left(|S|(p(A,\phi) + \epsilon) + \sum_{S' \in \shapes(\tiling')} |C_{S}(S')|\Delta_{S'}(\phi)\right) 	\\
&	=	\log\left(\frac{1}{1-\epsilon}\right) + \frac{|T_F|}{|F|}(p(A,\phi) + \epsilon) + \sum_{S \in \shapes(\tiling)} \sum_{S' \in \shapes(\tiling')} \frac{|C_F(S)||C_{S}(S')||S'|}{|F|} \frac{\Delta_{S'}(\phi)}{|S'|} \\
&	\leq	2\epsilon + (p(A,\phi) + \epsilon) + \sum_{S \in \shapes(\tiling)} \sum_{S' \in \shapes(\tiling')} \frac{|C_F(S)||C_{S}(S')||S'|}{|F|}  \epsilon	\\
&	=	p(A,\phi) + 3\epsilon + \frac{|T_F|}{|F|} \epsilon	\\
&	\leq	p(A,\phi) + 4\epsilon.
\end{align*}

In addition,
\begin{align*}
Z_F(\phi) \leq Z_{T_F}(\phi)Z_{F \setminus T_F}(\phi) \leq Z_{T_F}(\phi)Z_{1_G}(\phi)^{|F \setminus T_F|} \leq Z_{T_F}(\phi)Z_{1_G}(\phi)^{\epsilon|F|},
\end{align*}
so, considering that $p(A,\phi) = \inf_{E \in \finSet} \frac{1}{|E|} \log Z_E(A,\phi)$ by Theorem \ref{thm:continuous-potential-finite-alphabet-then-pressure-exists}, we have that
\begin{align*}
\frac{1}{|F|} \log Z_F(\phi)	\leq \inf_{E \in \finSet} \frac{1}{|E|} \log Z_E(A,\phi) + 4\epsilon + \epsilon\cdot\log Z_{1_G}(\phi).
\end{align*}

We conclude that, for every $0 < \epsilon < \frac{1}{2}$, there exist $A \in \cF(\N)$, $K \in \finSet$, and $\delta > 0$ such that for every $(K,\delta)$-invariant set $F \in \finSet$,
\begin{align*}
\frac{1}{|F|} \log Z_F(\phi) \leq	\inf_{E \in \finSet} \frac{1}{|E|} \log Z_E(A,\phi) + \epsilon \cdot C 
\end{align*}
where $C = 4 + \log Z_{1_G}(\phi)$. Since $\epsilon$ was arbitrary, we conclude the result.
\end{proof}

Now we can prove the following generalization of Theorem \ref{thm:continuous-potential-finite-alphabet-then-pressure-exists}.

\begin{theorem}\label{thm:existence-of-pressure}
Let $\phi\colon X \to \R$ be an exp-summable and uniformly continuous potential with finite oscillation. Then, $p(\phi)$ exists and $p(\phi) = \inf_{E \in \finSet} \frac{1}{|E|}\log Z_E(\phi)$. Moreover, $p(\phi) = \sup_{A \in \cF(\N)} p(A,\phi)$.
\end{theorem}

\begin{proof}
By Proposition \ref{prop:inf-of-log-partition-function}, for every $\frac{1}{2} > \epsilon > 0$, there exist $A \in \cF(\N)$, $K \in \finSet$ and $\delta > 0$ such that for every $(K,\delta)$-invariant set $F \in \finSet$,
\begin{align*}
\frac{1}{|F|} \log Z_F(\phi) \leq \inf_{E \in \finSet} \frac{1}{|E|} \log Z_E(A,\phi) + \epsilon.
\end{align*}

Therefore, for every such $F$,
\begin{align*}
\inf_{E \in \finSet} \frac{1}{|E|} \log Z_E(\phi)   &   \leq \frac{1}{|F|} \log Z_F(\phi)   \\
            &   \leq \inf_{E \in \finSet} \frac{1}{|E|} \log Z_E(A,\phi) + \epsilon  \\
&   \leq  \inf_{E \in \finSet} \frac{1}{|E|} \log Z_E(\phi) + \epsilon.
\end{align*}

Thus, $\lim_{F \to G} \frac{1}{|F|} \log Z_F(\phi) = \inf_{E \in \finSet} \frac{1}{|E|} \log Z_E(\phi)$, $p(\phi)$ exists, and there exists $A \in \cF(\N)$ such that
\begin{align*}
p(\phi) \leq p(A,\phi) + \epsilon \leq  p(\phi) + \epsilon,
\end{align*}
so $p(\phi) = \sup_{A \in \cF(\N)} p(A,\phi)$.
\end{proof}

\section{Permutations and specifications}
\label{section:permutations}

In order to define conformal and DLR measures it will be crucial to introduce coordinate-wise permutations and specifications. We begin by describing and exploring some properties of coordinate-wise permutations.

\subsection{Coordinate-wise permutations}

Let $S_{\N}$ be the set of all permutations of $\N$. Following \cite{keller1998equilibrium,muir2011gibbs}, we now introduce a class of local maps on $X$. Given an exhausting sequence $\{E_m\}_m$, this class will allow us to understand how $\phi_{E_m}(x)$ behaves if $x$ is changed at finitely many sites and it will be central when defining conformal measures in \S \ref{section:main-theorem}.

\begin{definition} Given $K \in \finSet$, denote by $\cE_K$ the set of all maps $\tau\colon X \to X$ such that 
\begin{equation*}
    \tau(x)_g =
\begin{cases}
\tau_g(x_g),  & \text{if } g \in K;     \\
x_g,          & \text{if } g \notin K;
\end{cases}
\end{equation*}
where $\tau_g \in S_{\N}$.
We usually denote $\tau$ by $\tau_K$ to emphasize the set $K$. 
\end{definition}

Let $\cE = \bigcup_{K \in \finSet} \cE_K$ and notice that there is a natural action of $G$ on $\cE$ given by
$$
(g \cdot \tau_K)(x) = g \cdot \tau_K(g^{-1} \cdot x),
$$
where $g \in G$, $x \in X$, $K \in \finSet$, $\tau_K \in \cE_K$, and $g \cdot \tau_K \in \cE_{Kg^{-1}}$. In order to avoid ambiguity, we will denote $g \cdot \tau_K$ by $\tau_{Kg^{-1}}$ and that will be enough for our purposes. 

We can also restrict ourselves to permutations over a finite alphabet. More explicitly, for $A \in \mathcal{F}(\N)$ and $K \in \finSet$, define
\begin{equation*}
    \cE_{K, A} = \{\tau \in \cE_K : \forall h\in K, \, \tau_h|_{A^c}=\textnormal{Id}_\N|_{A^c}\}.
\end{equation*}
Notice that $\cE$ is a group with the composition generated by single-site permutations $\tau_g$, where $\cE_K$ and $\cE_{K, A}$ are subgroups. Moreover, observe that if $g \neq h$, then $\tau_g \tau_h = \tau_h \tau_g$.

We will also consider a particular type of permutations, which are defined below.

\begin{definition}
\label{defn:coordinate-wise-permutation-two} Given $K \in \finSet$ and $w, w' \in X_K$, let $\tau_{w, w'} \colon X \to X$ be the map defined as
\begin{equation*}
    \tau_{w, w'}(x) =
\begin{cases}
w x_{K^c},    &   \text{if } x_K = w';     \\
w'x_{K^c},   &   \text{if } x_K = w;   \\
x,              &   \text{otherwise}.
\end{cases}
\end{equation*}
\end{definition}

It is clear that $\tau_{w, w'} \in \cE_K$, $\tau_{w, w'} = \tau_{w', w}$ and that $\tau_{w, w'}$ is an involution, that is, it is its own inverse. Moreover, there exists $A \in \cF(\N)$, namely $A = w(K) \cup w'(K)$, such that $\tau_{w, w'} \in \cE_{K, A}$. For $\tau \in \cE$ and $F \in \finSet$, define $\phi^{\tau}_F\colon X \to \R$ as
\begin{equation}\label{eq:defn-of-phi-tau-F}
\phi^{\tau}_F(x) = \phi_F \circ \tau^{-1}(x) - \phi_F(x).
\end{equation}

Notice that, for $\tau \in \cE_K$,
\begin{align*}
\phi^{\tau}_F(x) &  = \sum_{g \in F} \phi(g \cdot \tau_K^{-1}(x)) - \phi(g \cdot x) \\
                &   = \sum_{g \in F} \phi(g \cdot \tau_K^{-1}(g^{-1}\cdot(g \cdot x))) - \phi(g \cdot x) \\
                &   = \sum_{g \in F} \phi(\tau_{Kg^{-1}}^{-1}(g \cdot x))) - \phi(g \cdot x).
\end{align*}

\begin{lemma}\label{lemma:bound-of-phi-tau-k-sup-norm} Let $K \in \finSet$ and $\phi\colon X \to \R$ be a potential. Then, for every $\tau_K \in \cE_K$ and every $E, F \in \finSet$ with $F \subseteq E$,
\begin{equation*}
     \Vert \phi_{E}^{\tau_{K}} - \phi_{F}^{\tau_K}\Vert_{\infty} \leq \sum_{g \in G \setminus F} \left\Vert \phi \circ \tau_{Kg^{-1}}^{-1} - \phi\right\Vert_{\infty}.
\end{equation*}
\end{lemma}

\begin{proof}
Let $K, E, F \in \finSet$ and $\tau_K \in \cE_K$ be as in the statement of the Lemma. Then, it is easy to verify that, for any $x \in X$, $(\phi^{\tau_K}_{E} - \phi^{\tau_K}_{F})(x) = \sum_{g \in E \setminus F} \left[\phi\left(\tau_{Kg^{-1}}^{-1}(g \cdot x)\right) - \phi(g \cdot x)\right]$. Thus, 
\begin{align*}
    \Vert \phi_{E}^{\tau_{K}} - \phi_{F}^{\tau_K}\Vert_{\infty} &= \sup_{x \in X} \left|\phi_{E}^{\tau_K}(x) - \phi_{F}^{\tau_K}(x)\right|\\
    &= \sup_{x \in X} \left|\sum_{g \in E \setminus F} \left[\phi\left(\tau_{Kg^{-1}}^{-1}(g \cdot x)\right) - \phi(g \cdot x)\right]\right|\\
    &\leq  \sup_{x \in X} \sum_{g \in E \setminus F} \left| \phi\left(\tau_{Kg^{-1}}^{-1}(g \cdot x)\right) - \phi(g \cdot x)\right|\\
    &\leq \sum_{g \in E \setminus F} \sup_{x \in X} \left|\phi\left(\tau_{Kg^{-1}}^{-1}(g \cdot x)\right) - \phi(g \cdot x)\right|\\
    &=  \sum_{g \in E \setminus F} \left\Vert \phi \circ \tau_{Kg^{-1}}^{-1} - \phi\right\Vert_{\infty}\\
    &\leq \sum_{g \in G \setminus F} \left\Vert \phi \circ \tau_{Kg^{-1}}^{-1} - \phi\right\Vert_{\infty}.
\end{align*}

\end{proof}

Given a potential $\phi\colon X \to \R$ with summable variation according to an exhausting sequence $\{E_m\}_m$, the next theorem tells us that the asymptotic behaviour of $\phi_{E_m}(x)$ is essentially independent of the value of the configuration $x$ at finite sets $K \in \finSet$. The reader can compare the next result with \cite[Lemma 5.1.6]{keller1998equilibrium}.

\begin{theorem}
\label{thm:existence-of-uniform-limit-f-tau}
Let $\phi\colon X \to \R$ be a potential with summable variation according to an exhausting sequence $\{E_m\}_m$. Then, given any (possibly different) exhausting sequence $\{\tilde{E}_m\}_m$, for all $K \in \finSet$ and for all $\tau_K \in \cE_{K}$, the limit
\begin{equation*}
    \phi_*^{\tau_K} := \lim_{m \to \infty} \phi^{\tau_K}_{\tilde{E}_m}
\end{equation*}
exists uniformly on $X$ and on $\cE_{K}$. Moreover, such limit does not depend on the exhausting sequence.
\end{theorem}

\begin{proof}
First, suppose that $K$ is a singleton $\{h\}$ for some $h \in G$ and let $\epsilon > 0$. Since $\phi$ has summable variation according to  $\{E_m\}_m$, there exists $m_0 \in \N$ such that $\sum_{m \geq m_0} |E_{m+1}^{-1} \setminus E_m^{-1}| \delta_{E_m}(\phi) < \epsilon$. Now,  consider $\{\tilde{E}_m\}_m$ another (possibly different) exhausting sequence. Then, there exists $m_1 \geq m_0$ such that $E^{-1}_{m_0}h \subseteq \Tilde{E}_m$, for all $m \geq m_1$. On the other hand, since $\{E_m\}_m$ is an exhausting sequence, for every $m \geq m_1$, there exists $k_m \in \N$ such that for all $k \geq k_m$, $\Tilde{E}_m \subseteq E_k$. Therefore, by Lemma \ref{lemma:bound-of-phi-tau-k-sup-norm}, for every $m \geq m_1$ and every $k \geq k_m$,
\begin{align*}
    \left\Vert \phi^{\tau_h}_{E_k} - \phi^{\tau_h}_{\Tilde{E}_m} \right\Vert_{\infty}
    \leq \sum_{g \in G \setminus \Tilde{E}_m} \left\Vert \phi \circ \tau^{-1}_{hg^{-1}} - \phi \right\Vert_{\infty}.
\end{align*}

Moreover, since $E_{m_0}^{-1}h \subseteq \Tilde{E}_m$, we obtain that $G\setminus \Tilde{E}_m \subseteq G\setminus E_{m_0}^{-1}h$, so that
\begin{align*}
    \sum_{g \in G \setminus \Tilde{E}_m} \left\Vert \phi \circ \tau^{-1}_{hg^{-1}} - \phi \right\Vert_{\infty} &\leq \sum_{g \in G \setminus E_{m_0}^{-1}h} \left\Vert \phi \circ \tau^{-1}_{hg^{-1}} - \phi \right\Vert_{\infty} \\
    &= \sum_{g \in G \setminus E_{m_0}^{-1}} \left\Vert \phi \circ \tau^{-1}_{g^{-1}} - \phi \right\Vert_{\infty} \\
    &= \sum_{m \geq m_0}\sum_{g \in E_{m+1}^{-1} \setminus E_{m}^{-1}} \left\Vert \phi \circ \tau_{g^{-1}}^{-1} - \phi \right\Vert_{\infty}\\
    &\leq \sum_{m \geq m_0}\sum_{g \in E_{m+1}^{-1} \setminus E_{m}^{-1}}\delta_{E_m}(\phi)\\
    &= \sum_{m \geq m_0}\left|E_{m+1}^{-1} \setminus E_{m}^{-1}\right| \delta_{E_m}(\phi) < \epsilon.
\end{align*}

Therefore, for every $\epsilon > 0$, there exists $m_1 \geq m_0$ such that for every $m \geq m_1$, there exists $k_m$ such that for every $k \geq k_m$,
$$
\left\Vert \phi^{\tau_h}_{E_k} - \phi^{\tau_h}_{\Tilde{E}_m} \right\Vert_{\infty} < \epsilon.
$$

Notice that, in the particular case that $\{\tilde{E}_m\}_m$ is the same as $\{E_m\}_m$, one just need to take $k_m = m$ and the same inequality would follow. This proves that $\{\phi^{\tau_h}_{\Tilde{E}_m}\}_m$ is a Cauchy sequence for any $\tau_h \in \cE_{\{h\}}$, which implies that the uniform limit $\phi_*^{\tau_h} = \lim_{m \to \infty} \phi^{\tau_h}_{E_m}$ exists. On the other hand, if $\{\tilde{E}_m\}_m$ is another exhausting sequence, this proves that $\phi_*^{\tau_h} = \lim_{m \to \infty} \phi^{\tau_h}_{\tilde{E}_m}$, i.e., the limit is independent of the exhausting sequence provided $\phi$ has summable variation according to some exhausting sequence.

Now, let's consider a general $K \in \finSet$ and write $K = \{h_1,\dots,h_{|K|}\}$. Then, for each $m \in \N$,
\begin{align*}
 \phi_{E_m} \circ \tau^{-1}_K - \phi_{E_m}
    = \sum_{i=0}^{|K|-1} \left(\phi_{E_m} \circ \tau^{-1}_{\{h_1,\dots,h_{i+1}\}} - \phi_{E_m} \circ \tau^{-1}_{\{h_1,\dots,h_i\}}\right)
    = \sum_{i=0}^{|K|-1} \phi^{\tau_{h_{i+1}}}_{E_m} \circ \tau^{-1}_{\{h_1,\dots,h_i\}},
\end{align*}
where we regard $\tau_\emptyset$ as the identity, so the first equality follows from the fact that the considered sum is telescopic. Therefore, by considering the uniform convergence for singletons,
\begin{align*}
   \lim_{m \to \infty} \phi^{\tau_K}_{E_m} &=  \lim_{m \to \infty} \sum_{i=0}^{|K|-1} \phi^{\tau_{h_{i+1}}}_{E_m} \circ \tau^{-1}_{\{h_1,\dots,h_i\}}    \\
    &=  \sum_{i=0}^{|K|-1} \lim_{m \to \infty} \phi^{\tau_{h_{i+1}}}_{E_m} \circ \tau^{-1}_{\{h_1,\dots,h_i\}}    \\
    &   =  \sum_{i=0}^{|K|-1} \phi^{\tau_{h_{i+1}}}_* \circ \tau^{-1}_{\{h_1,\dots,h_i\}},   
\end{align*}
which concludes the result.
\end{proof}

\begin{corollary}
\label{cor:invariance-phi}
Let $\phi\colon X \to \R$ be a potential with summable variation according to an exhausting sequence $\{E_m\}_m$. Then, for all $K \in \finSet$ and for all $\tau_K \in \cE_{K}$,
\begin{equation*}
\phi_*^{\tau_K}(g \cdot x) = \phi_*^{\tau_{Kg}}(x),
\end{equation*}
for all $g \in G$ and $x \in X$.
\end{corollary}

\begin{proof} Notice that, given $g \in G$ and $x \in X$, we have that $\tau_K^{-1}(g \cdot x) = g \cdot \tau^{-1}_{Kg}(x)$,  so that
$$
\phi_*^{\tau_K}(g \cdot x)  =   \lim_{m \to \infty} \phi_{E_m}^{\tau_K}(g \cdot x) = \lim_{m \to \infty} \phi_{E_mg}^{\tau_{Kg}}(x) = \phi_*^{\tau_{Kg}}(x),
$$
since $\{E_mg\}_m$ is also an exhausting sequence.

\end{proof}

\begin{proposition}\label{prop:variation-of-permutations}
Let $\phi\colon X \to \R$ be a potential with summable variation according to an exhausting sequence $\{E_m\}_m$. Then, for every $F \in \finSet$ and $\tau$ in $\mathcal{E}_F$,
$$
\Vert \phi_*^{\tau} - \phi^{\tau}_F\Vert_{\infty} \leq V_F(\phi).
$$
\end{proposition}

\begin{proof}
Let $F \in\finSet$. From Lemma \ref{lemma:bound-of-phi-tau-k-sup-norm}, we know that
$$
\left\Vert\phi^{\tau_F}_{E_m} - \phi^{\tau_F}_F\right\Vert_\infty \leq \sum_{g \in G \setminus F} \left\Vert\phi \circ \tau_{Fg^{-1}}^{-1} - \phi\right\Vert_\infty,
$$
for every $m \in \N$ such that $F \subseteq E_m$. Therefore, by Theorem \ref{thm:existence-of-uniform-limit-f-tau},
$$
\left\Vert\phi_*^{\tau_F} - \phi^{\tau_F}_F\right\Vert_\infty = \lim_{m \to \infty} \left\Vert\phi^{\tau_F}_{E_m} - \phi^{\tau_F}_F\right\Vert_\infty \leq \sum_{g \in G \setminus F} \left\Vert\phi \circ \tau_{Fg^{-1}}^{-1} - \phi\right\Vert_\infty.
$$
Now, given $m \in \N$,  notice that
$g \in (E_m^{-1}F)^c \iff Fg^{-1} \cap E_m = \emptyset$, so that $\left\Vert\phi \circ \tau_{Fg^{-1}}^{-1} - \phi\right\Vert_\infty \leq \delta_{E_m}(\phi)$. Considering this, we have that
\begin{align*}
 \sum_{g \in G \setminus F} \left\Vert\phi \circ \tau_{Fg^{-1}}^{-1} - \phi\right\Vert_\infty 
 &=\sum_{m=1}^\infty \sum_{g \in E_{m+1}^{-1}F \setminus E_{m}^{-1}F} \left\Vert\phi \circ \tau_{Fg^{-1}}^{-1} - \phi\right\Vert_\infty \\
& \leq  \sum_{m=1}^\infty |E^{-1}_{m+1}F \setminus E^{-1}_{m}F| \cdot \delta_{E_m}(\phi)    \\
& =   V_F(\phi).
\end{align*}

\end{proof}

\subsection{Specifications}

This section tackles results about specifications, a concept related to DLR measures. More precisely, DLR measures can be defined using a special kind of specifications, but here we begin by presenting some more general results.  

Let $\sigmaAlg$ be the Borel $\sigma$-algebra, that is, the $\sigma$-algebra generated by the cylinder sets, and, for each $K \in \finSet$, let $\sigmaAlg_K$ be the $\sigma$-algebra generated by cylinder sets $[w]$, with $w \in X_K$. Now, a \textbf{specification} in our context, will mean a family $\gamma = (\gamma_{K})_{K \in \finSet}$ of maps $\gamma_{K}\colon \sigmaAlg \times X  \to [0,1]$ such that
\begin{enumerate}
\item[i)] for each $x \in X$, the map $B \mapsto \gamma_K(B,x)$ is a probability measure on $\mathcal{M}(X)$;
\item[ii)] for each $B \in \sigmaAlg$, the map $x \mapsto \gamma_K(B,x)$ is $\sigmaAlg_{K^c}$-measurable;
\item[iii)](proper) for every $B \in \sigmaAlg$ and $C \in \sigmaAlg_{K^c}$,
$\gamma_K\left(B \cap C,\cdot\right) = \gamma_K(B,\cdot)\mathbbm{1}_C$; and
\item[iv)] if $F \subseteq K$, then $\gamma_{K}\gamma_{F} = \gamma_{K}$, where
$\gamma_K\gamma_F(B, x) = \int \gamma_K(dy, x)\gamma_F(B,y),$ for $B \in \sigmaAlg$ and $x \in X$.
\end{enumerate}

In other words, $\gamma$ is a particular family of proper probability kernels that satisfies consistency condition (iv). An element $\gamma_K$ in the specification maps each $\mu \in \mathcal{M}(X)$ to $\mu\gamma_K \in \mathcal{M}(X)$, where
$$
\mu\gamma_K(B) = \int \gamma_K(B, x)d\mu(x),
$$
and each $\sigmaAlg$-measurable function $h: X \to \R$ to a $\sigmaAlg_{K^c}$-measurable function $\gamma_Kh: X \to \R$ given by
$$
\gamma_Kh(y) = \int h(y) \gamma_K(dy, x)d\mu(x).
$$

It can be checked that $(\mu\gamma_K)(h) = \mu(\gamma_Kh)$. The probability measures on the set 
\begin{align*}
    \mathscr{G}(\gamma) = \{\mu \in \mathcal{M}(X): \mu\left(B \,|\, \sigmaAlg_{K^c}\right) = \gamma_K\left(B, \cdot \right)\, \mu\text{-a.s., for all }
     B \in \sigmaAlg \text{ and } K \in \finSet\}
\end{align*}
are said to be {\bf admitted} by the specification $\gamma$.

\begin{lemma}\cite[Remark 1.24]{georgii2011gibbs}\label{lemma:characterization-of-admitted-by-specification} Let $\gamma$ be a specification and $\mu \in \mathcal{M}(X)$. Then, $\mu \in \mathscr{G}(\gamma)$ if and only if
$\mu\gamma_K = \mu$, for all $K \in \finSet$.
\end{lemma}

Now, we restrict ourselves to a particular kind of specification. Namely, given an exhausting sequence of finite sets $\{E_m\}_m$ and $\phi\colon X \to \R$ an exp-summable potential with summable variation according to  $\{E_m\}_m$, consider $\gamma = (\gamma_{K})_{K \in \finSet}$ the specification coming from $\phi$, where each $\gamma_{K}\colon \mathcal{B} \times X \to [0,1]$  is given by 
\begin{equation}\label{specification}
    \gamma_{K}(B, x):= \lim_{m \to \infty} \frac{\sum_{w \in X_{K}} \exp\left(\phi_{E_m}(w x_{K^c})\right)\mathbbm{1}_{\{w x_{K^c} \in B\}}}{\sum_{v \in X_{K}} \exp\left(\phi_{E_m}(v x_{K^c})\right)},
\end{equation}
for each $B \in \cB$ and $x \in X$. The collection $\gamma$ is a \emph{(Gibbsian) specification}. The expression in equation \eqref{specification} is well-defined due to the following proposition. 

\begin{proposition}
\label{existenceoflimit}
Let $\phi\colon X \to \R$ be an exp-summable potential with summable variation according to an exhausting sequence $\{E_m\}_{m}$. If $K \in \finSet$, the limit
\begin{equation*}
    \gamma_K([w],x) = \lim_{m \to \infty} \frac{\exp\left(\phi_{E_m}(w x_{K^c})\right)}{\sum_{v \in X_{K}} \exp\left(\phi_{E_m}(v x_{K^c})\right)}
\end{equation*}
exists for each $w \in X_{K}$, uniformly on $X$. Furthermore, for every $B \in \sigmaAlg$ and every $x \in X$, it holds that
\begin{equation}\label{sumspecification}
    \gamma_{K}(B, x) = \sum_{w \in X_{K}} \gamma_{K}\left([w], x\right)\mathbbm{1}_{\{w x_{K^c} \in B\}}.
\end{equation}
\end{proposition}

In order to prove Proposition \ref{existenceoflimit}, we require two lemmas, which we state and prove next.

\begin{lemma}
\label{lem:bounds}
Let $\phi\colon X \to \R$ be an exp-summable potential with summable variation according to an exhausting sequence $\{E_m\}_{m }$. Then, for any $K \in \finSet$ and for any $m \in \N$ such that $K \subseteq E_m$,
$$
 \left\vert\phi_{E_m}^{\tau_{w,v}}(wx_{K^c}) - (\sup\phi_{K}[v] - \sup\phi_{K}[w])\right\vert \leq \Delta_K(\phi)+ V_K(\phi)
$$
for every $v,w \in X_K$ and $x \in X$.
\end{lemma}

\begin{proof}
Let $K \in \finSet$ and $x,y \in X$ be such that $x_{G \setminus K} = y_{G \setminus K}$. Notice that for any $g \in G$, $(g \cdot x)_{G \setminus Kg^{-1}} = (g \cdot y)_{G \setminus Kg^{-1}}$. In addition, given $m \in \N$, we have that $g \in\left(E_{m}^{-1} K\right)^{c} \iff K g^{-1} \cap E_{m} = \emptyset$. In particular, if $g \in \left(E_{m}^{-1} K\right)^{c}$, we have that $|\phi(g \cdot x) - \phi(g \cdot y)| \leq \delta_{E_{m}}(\phi)$. Considering this, we obtain that
\begin{align*}
\sum_{g \in G \setminus K}|\phi(g \cdot x) - \phi(g \cdot y)|    &   =   \sum_{m=1}^\infty \sum_{g \in E_{m+1}^{-1}K \setminus E_{m}^{-1}K} |\phi(g \cdot x) - \phi(g \cdot y)| \\
                            &   \leq  \sum_{m=1}^\infty \sum_{g \in E_{m+1}^{-1}K \setminus E_{m}^{-1} K}\delta_{E_m}(\phi) \\
                            &   =   \sum_{m=1}^\infty |E_{m+1}^{-1}K \setminus E_{m}^{-1} K|\cdot\delta_{E_m}(\phi) \\
                            &   =   V_K(\phi).
\end{align*}

Now, let $m_0 \in \N$ be the smallest index such that $K \subseteq E_{m_0}$. Then, for every $m \geq m_0$, every $x \in X$, and every $v,w \in X_K$, we have that
\begin{align*}
\phi_{E_m}^{\tau_{w,v}}(wx_{K^c})
&= \phi_{E_m}(vx_{K^c}) - \phi_{E_m}(wx_{K^c})\\
&\leq \phi_K(vx_{K^c}) - \phi_{K}(wx_{K^c})+ \sum_{g \in G \setminus K}|\phi(g \cdot (vx_{K^c})) - \phi(g \cdot (wx_{K^c}))|\\
&\leq \phi_K(vx_{K^c}) - \phi_{K}(wx_{K^c}) + V_K(\phi)\\
&\leq \sup\phi_{K}[v] - \sup\phi_{K}[w] + \Delta_K(\phi)+ V_K(\phi),
\end{align*}
and, similarly,
\begin{align*}
\phi_{E_m}^{\tau_{w,v}}(wx_{K^c}))
&\geq \sup\phi_{K}[v]-\sup\phi_{K}[w]-\Delta_K(\phi)-V_K(\phi),
\end{align*}
so we conclude that
$$
|\phi_{E_m}^{\tau_{w,v}}(wx_{K^c}) - (\sup\phi_{K}[v] - \sup\phi_{K}[w])| \leq \Delta_K(\phi)+ V_K(\phi).
$$
\end{proof}

\begin{lemma}
\label{lem:dominatedConvergence}
Let $\phi\colon X \to \R$ be an exp-summable potential with summable variation according to an exhausting sequence $\{E_m\}_{m}$. Then, for any  $K \in \finSet$ and $w \in X_K$,
\begin{equation*}
    0 < \sum_{v \in X_{K}} \exp(\phi_*^{\tau_{w, v}}(w x_{K^c})) = \lim_{m \to \infty}\sum_{v \in X_K} \exp(\phi^{\tau_{w, v}}_{E_m}(w x_{K^c})),
\end{equation*}
uniformly on $X$.
\end{lemma}

\begin{proof}
Given $K \in \finSet$, $w \in X_K$, and $x \in X$, consider the sequence of functions $f_m\colon X_K \to \R$ given by $f_m(v) := \exp(\phi^{\tau_{w, v}}_{E_m}(w x_{K^c})).$ By Theorem \ref{thm:existence-of-uniform-limit-f-tau}, we have that $\{f_m\}_m$ converges pointwise (in $v$) to $\exp(\phi_*^{\tau_{w, v}}(w x_{K^c}))$, uniformly on $X$. In addition, by 
Lemma \ref{lem:bounds}, there exist $m_0 \in \N$ and a constant $C = \exp(\Delta_K(\phi)+V_K(\phi)) > 0$ such that for every $m \geq m_0$ and for every $v \in X_K$,
$$
C^{-1} \cdot h(v) \leq f_m(v) \leq C \cdot h(v),
$$
where $h(v) := \exp(-\sup\phi_{K}[w]) \cdot \exp(\sup\phi_{K}[v])$. Notice that
$$\sum_{v \in X_K}  h(v) = \exp(-\sup\phi_{K}[w]) \cdot Z_K(\phi),
$$
so $h$ (and therefore, $C \cdot h$) is integrable with respect to the counting measure in $X_K$. Therefore, by the Dominated Convergence Theorem, if follows that
\begin{align*}
\sum_{v \in X_K} 
\exp\left(\phi_*^{\tau_{w,v}}(wx_{K^c})\right)    &   =   \sum_{v \in X_K} \lim_m 
\exp\left(\phi_{E_m}^{\tau_{w,v}}(wx_{K^c})\right)  \\
                                                &   =   \lim_m \sum_{v \in X_K}
\exp\left(\phi_{E_m}^{\tau_{w,v}}(wx_{K^c})\right)  \\
                                                &   \geq    \lim_m \sum_{v \in X_K}
C^{-1} \cdot h(v)   \\
                                                &   =   C^{-1} \cdot \exp(-\sup\phi_{K}[w]) \cdot Z_K(\phi) > 0.
\end{align*}
\end{proof}

\begin{proof}[Proof (of Proposition \ref{existenceoflimit})]
First, note that for any given $K \in \finSet$,
\begin{equation*}
    \sum_{v \in X_{K}} \exp\left(\phi_{E_m}(v x_{K^c})\right) > 0.
\end{equation*}
for all $m \in \N$ and, due to Lemma \ref{lem:dominatedConvergence}, the left-hand side is bounded away from zero uniformly in $m$. Furthermore, for each $w \in X_K$,
\begin{align*}
    \frac{\exp\left(\phi_{E_m}(w x_{K^c})\right)}{\sum_{v \in X_{K}} \exp\left(\phi_{E_m}(v x_{K^c})\right)}    &   = \frac{1}{\sum_{v \in X_{K}} \exp\left(\phi_{E_m}(v x_{K^c}) - \phi_{E_m}(w x_{K^c}))\right)}  = \frac{1}{\sum_{v \in X_{K}} \exp(\phi^{\tau_{w, v}}_{E_m}(w x_{K^c}))}.
\end{align*}
  
Therefore, uniformly on $X$,
\begin{align*}
   \lim_{m \to \infty} \frac{\exp\left(\phi_{E_m}(w x_{K^c})\right)}{\sum_{v \in X_{K}} \exp\left(\phi_{E_m}(v x_{K^c})\right)}
   =   \frac{1}{\lim_{m \to \infty}\sum_{v \in X_{K}} \exp(\phi^{\tau_{w, v}}_{E_m}(w x_{K^c}))}
   &=   \frac{1}{\sum_{v \in X_{K}} \exp(\phi_*^{\tau_{w, v}}(w x_{K^c}))},
\end{align*}
again due to Lemma \ref{lem:dominatedConvergence}. Now, let  $B \in \cB$ and  $x \in X$. Then, uniformly on $X$,
\begin{align*}
\sum_{w \in X_K} &\gamma_{K}\left([w], x\right) \mathbbm{1}_{\{w x_{K^c} \in B\}}\\
 &= \sum_{w \in X_K}  \lim_{m \to \infty} \frac{\exp\left(\phi_{E_m}(w x_{K^c})\right)\mathbbm{1}_{\{w x_{K^c} \in [w]\}}\mathbbm{1}_{\{w x_{K^c} \in B\}}}{\sum_{v' \in X_K} \exp\left(\phi_{E_m}(v' x_{K^c})\right)}\\
 &=  \lim_{m \to \infty} \frac{\sum_{w \in X_K} \exp\left(\phi_{E_m}(w x_{K^c})\right)\mathbbm{1}_{\{w x_{K^c} \in [w]\}}\mathbbm{1}_{\{w x_{K^c} \in B\}}}{\sum_{v' \in X_K} \exp\left(\phi_{E_m}(v' x_{K^c})\right)}\\
  &=   \lim_{m \to \infty} \frac{\sum_{w \in X_K} \exp\left(\phi_{E_m}(w x_{K^c})\right)\mathbbm{1}_{\{w x_{K^c} \in B\}}}{\sum_{v' \in X_K} \exp\left(\phi_{E_m}(v' x_{K^c})\right)}\\
  &= \gamma_{K}(B, x),
 \end{align*}
 where the exchange of the limit and the sum follows from Lemma \ref{lem:dominatedConvergence}.
\end{proof}

\begin{proposition}
\label{lemma:characterization-DLR-state}
Let $\phi\colon X \to \R$ be an exp-summable potential with summable variation according to an exhausting sequence $\{E_m\}_m$. Then, for every $K \in \finSet$, $w \in X_{K}$, and $x \in X$, the equation
\begin{equation*}
    \gamma_K([w],x) = \frac{\exp\left(\phi_*^{\tau_{w, v}}(vx_{K^c})\right)}{\sum_{w^{\prime} \in X_{K}}\exp\left(\phi_*^{\tau_{w^{\prime}, v}}(vx_{K^c})\right)}
\end{equation*}
holds for every $v \in X_{K}$. 
\end{proposition}

\begin{proof}
Let $K \in \finSet$,  $w \in X_K$ and $x \in X$. Then, for any $v \in X_K$,
\begin{align*}
\lim_{m \to \infty} \frac{\exp\left(\phi_{E_m}(w x_{K^c})\right)}{\sum_{{w'}\in X_{K}}\exp\left(\phi_{E_m}({w'} x_{K^c})\right)}
&=\lim_{m \to \infty} \frac{\exp\left(\phi_{E_m}\circ\tau_{w,v}^{-1} - \phi_{E_m}\right)(vx_{K^c})}{\sum_{{w'}\in X_{K}} \exp\left(\phi_{E_m}\circ\tau_{w',v}^{-1} - \phi_{E_m}\right)(vx_{K^c})}\\
&=\frac{\lim_{m \to \infty}\exp\left(\phi_{E_m}\circ\tau_{w,v}^{-1} - \phi_{E_m}\right)(vx_{K^c})}{\sum_{{w'}\in X_{K}}\lim_{m \to \infty} \exp\left(\phi_{E_m}\circ\tau_{w',v}^{-1} - \phi_{E_m}\right)(vx_{K^c})}\\
&= \frac{\exp\left(\lim_{m \to \infty} \left(\phi_{E_m}\circ\tau_{w,v}^{-1} - \phi_{E_m}\right)(vx_{K^c})\right)}{\sum_{{w'}\in X_{K}}\exp \left(\lim_{m \to \infty}\left(\phi_{E_m}\circ\tau_{w',v}^{-1} - \phi_{E_m}\right)(vx_{K^c})\right)}\\
&= \frac{\exp\left(\phi_*^{\tau_{w, v}}(vx_{K^c})\right)}{\sum_{w^{\prime} \in X_{K}}\exp\left(\phi_*^{\tau_{w^{\prime}, v}}(vx_{K^c})\right)},
\end{align*}
where the last equality follows from Theorem \ref{thm:existence-of-uniform-limit-f-tau}. Also, if $m_0 \in \N$ is such that $K \subseteq E_{m_0}$, the exchange of limit and sum in the denominator from the first to the second line follows from Lemma \ref{lem:dominatedConvergence}.

\end{proof}

\begin{corollary}
\label{cor:invariance-spec}
Let $\phi\colon X \to \R$ be an exp-summable potential with summable variation according to an exhausting sequence $\{E_m\}_m$. Then, for every $K \in \finSet$, $\gamma_K$ is $G$-invariant, that is, for every $w \in X_{K}$, $x \in X$, and $g \in G$, it holds that
\begin{equation*}
\gamma_{Kg^{-1}}(g \cdot [w],g \cdot x) = \gamma_K([w],x).
\end{equation*}

\end{corollary}

\begin{proof} Let $K \in \finSet$. Given $v \in X_K$, let $y^v \in X$ be arbitrary and such that $y^v_K = v$. Then,
\begin{align*}
\gamma_{Kg^{-1}}(g \cdot [w],g \cdot x)   &   =   \gamma_{Kg^{-1}}([(g \cdot y^w)_{Kg^{-1}}],g \cdot x)  \\
&=   \frac{\exp\left(\phi_*^{\tau_{Kg^{-1}}}((g \cdot y^w)_{Kg^{-1}}(g \cdot x)_{(Kg^{-1})^c})\right)}{\sum_{w' \in X_{K}}\exp\left(\phi_*^{\tau_{Kg^{-1}}}((g \cdot y^{w'})_{Kg^{-1}}(g \cdot x)_{(Kg^{-1})^c})\right)} \\
&=   \frac{\exp\left(\phi_*^{\tau_{Kg^{-1}}}(g \cdot (y^w_K x_{K^c}))\right)}{\sum_{w' \in X_{K}}\exp\left(\phi_*^{\tau_{Kg^{-1}}}(g \cdot (y^{w'}_K x_{K^c}))\right)} \\
&=   \frac{\exp\left(\phi_*^{\tau_K}(y^w_K x_{K^c})\right)}{\sum_{w' \in X_{K}}\exp\left(\phi_*^{\tau_K}(y^{w'}_K x_{K^c})\right)} \\
& = \gamma_{K}([w],x),
\end{align*}
where we have used the property of $\phi_*^\tau$ from Corollary \ref{cor:invariance-phi}.
\end{proof}

\begin{definition} A potential $h\colon X \to \R$ is \textbf{local} if $h$ is $\cB_{K}$-measurable for some $K \in \finSet$. For each $K \in \finSet$, denote by $\mathcal{L}_{K}$ the linear space of all bounded $\cB_{K}$-measurable potentials and $\mathcal{L} = \bigcup_{K \in \finSet} \mathcal{L}_{K}$.

A potential $h\colon X \to \R$ is \textbf{quasilocal} if there exists a sequence $\{\phi_n\}_{n}$ of local potentials such that $\lim_{n \to \infty}\Vert h - h_n\Vert_{\infty} = 0$. Note that $\overline{\mathcal{L}}$  is the linear space of all bounded quasilocal potentials, where $\overline{\mathcal{L}}$ is the uniform closure of $\mathcal{L}$ on the linear space of bounded $\sigmaAlg$-measurable potentials. 
\end{definition}

\begin{remark}\cite[Remark 2.21]{georgii2011gibbs}\label{rmk:characterization-of-quasilocal} A potential $h\colon X \to \R$ is quasilocal if and only if for all exhausting sequences of finite subsets $\{E_m\}_{m}$ of $G$, $\lim_{m \to \infty}\sup_{\substack{x, y \in X\\ x_{E_m} = y_{E_m}}} |h(x) - h(y)| = 0 $. 
\end{remark}

\begin{definition} A specification $\gamma = \left(\gamma_K\right)_{K \in \finSet}$ is quasilocal if, for each $K \in \finSet$ and $h \in \overline{\mathcal{L}}$, it holds that $\gamma_K h \in \overline{\mathcal{L}}$, where 
\begin{equation*}
    \gamma_{K}h(x) = \sum_{w \in X_{K}}\gamma_{K}(w,x)h(wx_{K^c}).
\end{equation*}
\end{definition}

\begin{remark}\label{rmk:quasilocal-specifications} In order to verify that a specification is quasilocal it suffices to prove that $\gamma_K h  \in \overline{\mathcal{L}}$, for $K \in \finSet$ and $h \in \mathcal{L}$ (see \cite{georgii2011gibbs}, page 32).
\end{remark}

\begin{theorem}\label{thm:quasilocality} Let $\phi\colon X \to \R$ be an exp-summable potential with summable variation according to an exhausting sequence $\{E_m\}_m$. If $\gamma = \{\gamma_{K}\}_{K \in \finSet}$ is defined as in equation \eqref{specification}, then $\gamma$ is quasilocal.
\end{theorem}

\begin{proof} Let $h \in \mathcal{L}$ and let $\epsilon > 0$. Given any $K \in \finSet$, first notice that 
\begin{equation*}
|\gamma_{K}h(x)| \leq \sum_{w \in X_{K}}\gamma_{K}(w,x)|h(wx_{K^c})| \leq \|h\|_\infty\sum_{w \in X_{K}}\gamma_{K}(w,x) = \|h\|_\infty,    
\end{equation*}
so $\|\gamma_{K}h\|_\infty \leq \|h\|_\infty$. In addition, if $x, y \in X$ are such that $x_{E_n} = y_{E_n}$ for $n$ to be determined, we have that
\begin{align*}
|\gamma_{K}&h(x) - \gamma_{K}h(y)|\\   
&\leq    \sum_{w \in X_{K}}|\gamma_{K}(w,x)h(wx_{K^c}) - \gamma_{K}(w,y)h(wy_{K^c})|   \\
&=    \sum_{w \in X_{K}}\gamma_{K}(w,x) \left|h(wx_{K^c}) - \frac{\gamma_{K}(w,y)}{\gamma_{K}(w,x)}h(wy_{K^c})\right| \\
&\leq    \sum_{w \in X_{K}}\gamma_{K}(w,x) \left|h(wx_{K^c}) - e^{\pm 2\epsilon}h(wy_{K^c})\right| \\
&\leq    \sum_{w \in X_{K}}\gamma_{K}(w,x) \left|h(wx_{K^c}) - h(wy_{K^c})\right| + \sum_{w \in X_{K}}\gamma_{K}(w,x)(1-e^{\pm 2\epsilon}) \left|h(wy_{K^c})\right| \\
&\leq    \sum_{w \in X_{K}}\gamma_{K}(w,x) \left|h(wx_{K^c}) - h(wy_{K^c})\right| + \sum_{w \in X_{K}}\gamma_{K}(w,x)(1-e^{\pm 2\epsilon})\left\|h\right\|_\infty \\
&\leq    \sup_{x',y': x'_{E_n} = y'_{E_n}}|h(x') - h(y')| + (1-e^{\pm 2\epsilon}) \left\|h\right\|_\infty.   
\end{align*}
To justify the second inequality, first observe that,  for every $w,v \in X_K$, $\phi_*^{\tau_{v,w}}$ is uniformly continuous, since it is a uniform limit of uniformly continuous potentials, namely, $\phi_{E_m}$. Then, there exists $n_0 \in \N$ such that for every $n \geq n_0$, every $w,v \in X_K$, and every $x, y \in X$ with $x_{E_n} = y_{E_n}$,

$$
|\phi_*^{\tau_{v, w}}(wx_{K^c}) - \phi_*^{\tau_{v, w}}(wy_{K^c})| < \epsilon, 
$$
so
$$
\gamma_{K}(w,y) = \frac{\exp\left(\phi_*^{\tau_{w, w}}(wy_{K^c})\right)}{\sum_{v \in X_{K}}\exp\left(\phi_*^{\tau_{v, w}}(wy_{K^c})\right)} \leq \frac{\exp\left(\phi_*^{\tau_{w, w}}(wx_{K^c}) + \epsilon\right)}{\sum_{v \in X_{K}}\exp\left(\phi_*^{\tau_{v, w}}(wx_{K^c})-\epsilon\right)} = e^{2\epsilon}\gamma_{K}(w,x).
$$

Now, since $h$ is local, we have that $\lim_{n \to \infty} \sup_{\substack{x,y \in X\\ x_{E_n} = y_{E_n}}} |h(x) - h(y)| = 0$, so that there exists $n_1 \in \N$ such that for all $n \geq n_1$, $\sup_{\substack{x,y \in X\\ x_{E_n} = y_{E_n}}} |h(x) - h(y)| < \epsilon$. Taking $n=\max\{n_0,n_1\}$, we obtain that
\begin{equation*}
    |\gamma_{K}h(x) - \gamma_{K}h(y)| \leq \epsilon +  (1-e^{\pm 2\epsilon}) \left\|h\right\|_\infty,
\end{equation*}
and since $\epsilon$ was arbitrary, we conclude.
\end{proof}

\section{Equivalences of different notions of Gibbs measures}\label{section:main-theorem}

In this section, we introduce the four notions of Gibbs measures to be considered, namely, DLR, conformal, Bowen-Gibbs, and equilibrium measures, and prove the equivalence among them provided extra conditions. We mainly assume that $G$ is a countable amenable group, the configuration space is $X = \N^G$, and $\phi: X \to \R$ is an exp-summable potential with summable variation according to an exhausting sequence $\{E_m\}_{m}$.

We proceed to describe the content of each subsection: in \S \ref{sec:main-def}, we provide a rigorous definition of each kind of measure and results about entropy and pressure; in \S \ref{sec:DLR-iff-conformal}, we establish that the set of DLR measures and the set of conformal measures coincide; in \S \ref{sec:DLR-implies-Bowen-Gibbs}, we prove that every DLR measure is a Bowen-Gibbs measure; in \S\ref{sec:existence-of-DLR-states}, we show the existence of a conformal measure; in \S\ref{sec:Bowen-Gibbs-are-equilibrium}, we prove that  a $G$-invariant Bowen-Gibbs measure with finite entropy is an equilibrium measure; finally, in \S\ref{sec:equilibrium-states-are-DLR}, we prove that if a measure is an equilibrium measure, then it is also a DLR measure.

Below, we provide a diagram of the main results of this section, including extra assumptions needed.
$$
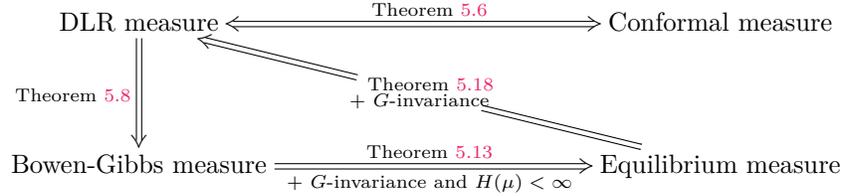
\xymatrix@C=12em@R=4em{
\text{DLR measure} \ar@{<=>}[r]^{\text{Theorem \ref{thm:conformal-iff-DLR}}} \ar@{=>}[d]_{\text{Theorem \ref{thm:bowgib}}} & \text{Conformal measure}\\
\text{Bowen-Gibbs measure} \ar@{=>}[r]_{\text{+ $G$-invariance  and $H(\mu) < \infty$}}^{\text{Theorem \ref{prop:bowen-gibbs-implies-equlibrium}}} & \text{Equilibrium measure} \ar@{=>}[lu]|-{\substack{\text{\hspace{0.2cm} Theorem \ref{thm:equilibrium-implies-DLR}}\\ \text{+ $G$-invariance}}}
} 
$$

\begin{remark}
We are not aware whether it is possible to prove that a Bowen-Gibbs measure is necessarily a DLR measure without the finite entropy assumption. In fact, we do not know if $G$-invariance is a necessary assumption for that implication.
\end{remark}

\subsection{Definitions of Gibbs measures}

\label{sec:main-def}

We start by giving the definitions of DLR, conformal, and Bowen-Gibbs measures.

\begin{definition}
Let $\phi\colon X \to \R$ be an exp-summable potential with summable variation according to an exhausting sequence $\{E_m\}_m$. A measure $\mu \in \mathcal{M}(X)$ is a {\bf DLR measure (for $\phi$)} if
\begin{equation*}
    \mu\left(B\,|\,\cB_{K^c}\right)(x) = \gamma_{K}(B,x) \qquad \mu(x)\text{-a.s.,}
\end{equation*}
 for every $K \in \finSet$, $B \in \cB$, and $x \in X$, where $\gamma_K$ is defined as in equation \eqref{specification}. We denote the set of DLR measures for $\phi$ by $\gibbs$.
\end{definition}

\begin{definition}
Let $\phi\colon X \to \R$ be an exp-summable potential with summable variation according to an exhausting sequence $\{E_m\}_m$. A measure $\mu \in \mathcal{M}(X)$ is a {\bf conformal measure (for $\phi$)} if
\begin{equation}
\label{defn:conformal-measure-countable-alphabet}
    \frac{d(\mu\circ\tau^{-1})}{d \mu}=\exp(\phi_*^{\tau}) \qquad \mu(x)\text{-a.s.}
\end{equation}
for every $A \in \mathcal{F}(\N)$, $K \in \finSet $, and $\tau\in\cE_{K,A}$.
\end{definition}

\begin{definition}\label{defn:bowen-gibbs-measure} Let $\phi\colon X \to \R$ be an exp-summable potential with summable variation according to an exhausting sequence $\{E_m\}_m$. A measure $\mu \in \mathcal{M}(X)$ is a {\bf Bowen-Gibbs measure (for $\phi$)} if there exists $p \in \R$ such that, for every $\epsilon > 0$, there exist $K \in \finSet$ and $\delta > 0$ such that, for every $(K,\delta)$-invariant set $F \in \finSet$ and $x \in X$,
\begin{equation}\label{eq:defn-of-Bowen-Gibbs}
       \exp\left(-\epsilon \cdot|F|\right) \leq  \frac{\mu([x_F])}{\exp\left(\phi_F(x) - p  |F|\right)} \leq \exp\left(\epsilon\cdot  |F|\right).
\end{equation}
\end{definition}

\begin{remark} Notice that, in Definition \ref{defn:bowen-gibbs-measure}, we can replace $\phi_F(x)$ by $\sup\phi_F([x_F])$ in an equivalent way, so that we have
\begin{equation*}
       \exp\left(-\epsilon \cdot|F|\right) \leq  \frac{\mu([x_F])}{\exp\left(\sup\phi_F([x_F]) - p  |F|\right)} \leq \exp\left(\epsilon \cdot |F|\right).
\end{equation*}

\end{remark}

\begin{proposition} Let $\phi\colon X \to \R$ be an exp-summable potential with summable variation according to an exhausting sequence $\{E_m\}_m$. Then, if $\mu$ is a Bowen-Gibbs measure for $\phi$, the constant $p$ is necessarily $p(\phi)$.
\end{proposition}\label{subsec:p=p(phi)}

\begin{proof} Indeed, given $\epsilon > 0$, there exist $K \in \finSet$ and $\delta > 0$ so that
$$
\exp\left(-\epsilon\cdot|F|\right) \exp\left(\phi_F(x)\right) \leq \mu([x_F])\exp\left(p |F|\right)  \leq \exp\left(\epsilon\cdot|F|\right) \exp\left(\phi_F(x)\right)
$$
for every $(K,\delta)$-invariant set $F \in \finSet$ and every $x \in X$. Since $x$ is arbitrary, we have that
$$
\exp\left(-\epsilon\cdot|F|\right) \exp\left(\sup \phi_F[x_F]\right) \leq \mu([x_F])\exp\left(p  |F|\right)  \leq \exp\left(\epsilon\cdot|F|\right) \exp\left(\sup \phi_F[x_F]\right),
$$
and, since $\mu$ is a probability measure, adding over all $x_F \in X_F$, we get
$$
\exp\left(-\epsilon\cdot|F|\right) Z_F(\phi) \leq \exp\left(p |F|\right)  \leq \exp\left(\epsilon\cdot|F|\right)  Z_F(\phi).
$$

Then, if we take logarithms and divide by $|F|$, we obtain that
$$
-\epsilon + \frac{\log Z_F(\phi)}{|F|} \leq p \leq \frac{\log Z_F(\phi)}{|F|} + \epsilon,
$$
so, taking the limit as $F$ becomes more and more invariant, we obtain that
$$
-\epsilon + p(\phi) \leq p \leq p(\phi) + \epsilon,
$$
and since $\epsilon$ was arbitrary, we conclude that $p = p(\phi)$.
\end{proof}

Consider the canonical partition of $X$ given by $\{[a]\}_{a \in \N}$. This is a countable partition that generates the Borel $\sigma$-algebra $\sigmaAlg$ under the shift dynamic. Given a measure  $\nu \in \mathcal{M}(X)$, the \textbf{Shannon entropy} of the canonical partition associated with $\nu$ is given by 
\begin{equation*}
    H(\nu) := -\sum_{a \in \N} \nu([a])\log\nu([a]).
\end{equation*}
Now, for each $F \in \finSet$, let $\{[w]\}_{w \in X_F}$ be the $F$-refinement of the canonical partition and consider its corresponding Shannon entropy, which is given by
$$
H_F(\nu) := -\sum_{w \in X_F} \nu([w])\log\nu([w]).
$$

We have the following proposition.

\begin{proposition}\label{prop:finite-Shannon-entropy} Let $\phi\colon X \to \R$ be an exp-summable and continuous potential with finite oscillation. If $\nu \in \mathcal{M}(X)$ is such that $\int \phi d\nu > -\infty$, then $H(\nu) < \infty$. Furthermore, if $\nu$ is $G$-invariant, then, for every $F \in \finSet$, $H_F(\nu) < \infty$.
\end{proposition} 

\begin{proof}
Let $\{A_n\}_n$ an exhausting sequence of finite alphabets and $F \in \finSet$. Consider $X^{F,n} = \{x \in X: x_F \in A_n^F\} \in \sigmaAlg_F$. Since $\phi$ is exp-summable, then it is bounded from above. Without loss of generality, suppose that it is bounded from above by $0$. Thus, so is $\phi_F$. Define
$$
\phi^{F,n}(x) = \begin{cases}
\phi_F(x), & x \in X^{F,n};\\
0, & \text{otherwise}.
\end{cases}
$$
Notice that, for every $x \in X$, $\phi_F(x) = \lim_{n\to\infty}\phi_{F,n}(x)$ and, for every $n \in \N$, $\phi_F(x) \leq \phi^{F, n+1}(x) \leq \phi^{F, n}(x)$. Therefore, by the Monotone Convergence Theorem, we can conclude that 
\begin{equation*}
    \int \phi_F d\nu = \lim_{n\to\infty} \int \phi^{F,n}(x) d\nu.
\end{equation*}

For each $n\in \N$, let $H_{F, n}(\nu) = -\sum_{w \in A_n^F} \nu([w])\log\nu([w])$. Then, $\lim_{n\to\infty} H_{F,n}(\nu) = H_F(\nu)$. Also, for each $n\in \N$ and $F \in \finSet$, notice that $\phi^{F,n } \leq \sum_{w \in A_n^{F}} 1_{[w]} \sup \phi_{F}([w])$. Therefore, for every $n\in \N$ and $F \in \finSet$,
\begin{align*}
   H_{F,n}(\nu) + \int \phi^{F,n} d\nu &= -\sum_{w \in A^{F}_n} \nu([w])\log \nu([w]) + \int \phi^{F,n} d\nu\\
    &\leq -\sum_{w \in A_n^{F}} \nu([w])\log \nu([w]) +  \sum_{w \in A_n^{F}} \nu([w])\sup \phi_{F}([w])\\
    &= \sum_{w \in A_n^{F}} \nu([w])\log\left(\frac{\exp\left(\sup\phi_F([w])\right)}{\nu([w])}\right)\\
    &\leq \log\left(\sum_{w \in A_n^{F}} \exp \sup \phi_F([w])\right)\\
    &= \log Z_{F}(A_n,\phi),
\end{align*}
where we assume that all the sums involved are over cylinder sets with positive measure. The second inequality follows from Jensen's inequality. In addition, notice that, in the case that $\nu$ is $G$-invariant, it follows that
\begin{align*}
     H_{F}(\nu) &= \lim_{n \to \infty} H_{F,n}(\nu) \\
                &\leq \lim_{n \to \infty} \left(\log Z_{F}(A_n,\phi) - \int \phi^{F,n} d\nu\right) \\
                &= \log Z_F(\phi) - \int \phi_F d\nu \\
                & \leq |F|\left(\log Z_{1_G}(\phi) - \int \phi d\nu\right),
\end{align*}
where we have used that $\log Z_{F}(\phi) \leq |F|\log Z_{1_G}(\phi)$ and $\int \phi_F d\nu = |F|\int \phi d\nu$. Therefore, $H_F(\nu) < \infty$ and, in particular, $H(\nu) = H_{\{1_G\}}(\nu) < \infty$.
\end{proof}

Through a standard argument (for example, for the case $G = \Z$, see \cite{downarowicz2011entropy}; the general case is analogous), it can be justified that if the canonical partition has finite Shannon entropy, the \textbf{Kolmogorov-Sinai entropy} of $\nu$ can be written as 
\begin{equation*}
    h(\nu) = \lim_{F\to G} \frac{1}{|F|}H_F(\nu).
\end{equation*}

The next proposition is based on \cite[Lemma 4.9]{muir2011gibbs} and gives us an upper bound in terms of the pressure for the \emph{specific Gibbs free energy} of a given measure with respect to some potential. Sometimes this fact is known as the \emph{Gibbs inequality}.

\begin{proposition}\label{prop:half-of-pv} Let $\phi\colon X \to \R$ be an exp-summable and uniformly continuous potential with finite oscillation. If $\mu \in \mathcal{M}(X)$ is $G$-invariant and $\int \phi d\mu > -\infty$, then $h(\mu) + \int \phi d\mu \leq p(\phi)$.
\end{proposition}

\begin{proof} Since  $\phi\colon X \to \R$ is an exp-summable and uniformly continuous potential with finite oscillation, due to Theorem \ref{thm:existence-of-pressure}, the pressure $p(\phi)$ exists. Then,
\begin{align*}
    h(\mu) + \int \phi d\mu = \lim_{F \to G} \frac{1}{|F|}H_F(\mu) + \int \phi d\mu
    \leq \lim_{F \to G} \frac{1}{|F|} \log Z_{F}(\phi)
    = p(\phi).
\end{align*}
\end{proof}

\begin{definition}\label{defn:equilibrium-state}
Let $\phi\colon X \to \R$ be an exp-summable potential with summable variation according to an exhausting sequence $\{E_m\}_m$. A measure $\mu \in \mathcal{M}_G(X)$ is an {\bf equilibrium measure (for $\phi$)} if $\int{\phi}d\mu > -\infty$ and
\begin{equation}\label{eq:equilibrium-state-eq-defn}
h(\mu) + \int{\phi}d\mu = \sup \left\{h(\nu) + \int{\phi} d\nu\colon \nu \in \mathcal{M}_G(X), \int{\phi}d\nu > -\infty\right\}.
\end{equation}
\end{definition}

Notice that it is not clear whether the supremum in equation \eqref{eq:equilibrium-state-eq-defn} is achieved. The answer to this problem is intimately related to the concept of Gibbs measures in its various forms and their equivalences, which we address throughout this section.

\begin{remark} Notice that, in light of Proposition \ref{prop:finite-Shannon-entropy},  any measure $\nu \in \mathcal{M}_G(X)$ such that the $\int \phi d\nu > -\infty $ has finite entropy, that is, $h(\nu) < \infty$, provided that $\phi$ is exp-summable and has finite oscillation. Thus, in the particular case that $\phi$ is an exp-summable potential with summable variation according to an exhausting sequence $\{E_m\}_m$, we obtain that $h(\nu) < \infty$.
\end{remark}

\subsection{Equivalence between DLR and conformal measures}\label{sec:DLR-iff-conformal}

This section is dedicated to proving that the notions of DLR measure and conformal measure coincide in the full shift with countable alphabet over a countable amenable group context. Nevertheless, before proving this major result, notice that for $B \in \cB$, $K \in \finSet$, and $x \in X$,
\begin{equation}\label{eq:conditional-measure-of-measurable-as-sum-of-cylinders}
\mu(B \,\vert \, \mathcal{B}_{K^c})(x) = \sum_{w \in X_K} \mu([w] \,\vert \,\mathcal{B}_{K^c})(x)\mathbbm{1}_{\{w x_{K^c} \in B\}}.    
\end{equation}

Indeed, it can be checked that $\mathbbm{1}_{\{w x_{K^c} \in B\}}(x)$ is $\mathcal{B}_{K^c}$-measurable, so $\mu(x)$-a.s.,
\begin{align*}
\sum_{w \in X_K} \mu([w] \, \vert \, \mathcal{B}_{K^c})(x)\mathbbm{1}_{\{w x_{K^c} \in B\}}(x)  &=   \mu\left(\sum_{w \in X_K} \mathbbm{1}_{[w]} \mathbbm{1}_{\{w x_{K^c} \in B\}}\middle\vert \mathcal{B}_{K^c}\right)(x) \\
    &   =   \mu\left(\sum_{w \in X_K} \mathbbm{1}_{[w]} \mathbbm{1}_B \middle\vert \mathcal{B}_{K^c}\right)(x) \\
    &   =   \mu\left(B \vert \mathcal{B}_{K^c}\right)(x).
\end{align*}

This observation will allow us to reduce our calculations from arbitrary Borel sets $B \in \sigmaAlg$ to cylinder sets of the form $[w]$. Next, we have the following result.

\begin{corollary}\label{charact-DLR-corol}
Let $\phi\colon X \to \R$ be an exp-summable potential with summable variation according to an exhausting sequence $\{E_m\}_m$. A measure $\mu \in \mathcal{M}(X)$ is a DLR measure for $\phi$ if, and only if, for every $K \in \finSet$, $w \in X_{K}$ and $x \in X$, it holds that
\begin{equation}
\label{DLR-permutation}
    \mu([{w}]\,\vert\, \mathcal{B}_{K^c})(x) = \frac{\exp\left(\phi_*^{\tau_{w, v}}(vx_{K^c})\right)}{\sum_{w^{\prime} \in X_{K}}\exp\left(\phi_*^{\tau_{w^{\prime}, v}}(vx_{K^c})\right)} \qquad \mu(x)\text{-a.s.},
\end{equation}
for every $v \in X_{K}$.
\end{corollary}

\begin{proof}
If $\mu$ is a DLR measure for $\phi$, then for every $K \in \finSet$, $B \in \cB$, and $x \in X$, 
\begin{equation*}
    \mu\left(B\,|\,\cB_{K^c}\right)(x) = \gamma_{K}(B,x) \quad \mu(x)\text{-a.s.}
\end{equation*}
Thus, in particular, if $w \in X_K$, it holds that
\begin{equation*}
    \mu([{w}] \,\vert \, \mathcal{B}_{K^c})(x) = \gamma_K([w],x) \qquad  \mu(x)\text{-a.s.},
\end{equation*}
and the result follows from Proposition \ref{lemma:characterization-DLR-state}.

On the other hand, if we assume that for every $K \in \finSet$, $w \in X_{K}$, and $x \in X$, equation (\ref{DLR-permutation}) holds $\mu(x)$-almost surely for every $v \in X_{K}$, then, from equation \eqref{eq:conditional-measure-of-measurable-as-sum-of-cylinders} and Proposition \ref{lemma:characterization-DLR-state},  $\mu(x)$-a.s it holds that
\begin{align*}
\mu(B \vert \mathcal{B}_{K^c})(x)   &   = \sum_{w \in X_K} \mu([w] \vert  \mathcal{B}_{K^c})(x)\mathbbm{1}_{\{w x_{K^c} \in B\}}   
= \sum_{w \in X_K} \gamma_K([w],x)\mathbbm{1}_{\{w x_{K^c} \in B\}} =  \gamma_K(B,x).
\end{align*}
\end{proof}

In order to relate the functions $\gamma_K$ that appear in the definition of DLR measures with the permutations involved in the definition of conformal measures, we have the following lemma.

\begin{lemma}
\label{lemma:specification-permutation} Let $\phi\colon X \to \R$ be an exp-summable potential with summable variation according to an exhausting sequence $\{E_m\}_m$. Then, for every $K \in \finSet$,  $v, w \in X_{K}$ and $\tau \in \cE_K$ such that $\tau^{-1}([v]) = [w]$, 
\begin{equation*}
\gamma_K([w],x) = \exp\left(\phi_*^{\tau}(v x_{K^c})\right) \gamma_K([v],x).
\end{equation*}
\end{lemma}

\begin{proof}
Indeed, by Proposition \ref{lemma:characterization-DLR-state}, for every $x \in X$,
\begin{align*}
    \gamma_K([w],x) &= \frac{\exp \left(\phi_*^{\tau}(v x_{K^c})\right)}{\sum_{w' \in X_{K}}\exp \left(\phi_*^{\tau_{w^{\prime}, v}}(v x_{K^c})\right)}\\
    &= \frac{\exp\left(\phi_*^{\tau_{v, v}}(v x_{K^c})\right)}{\sum_{w' \in X_{K}}\exp \left(\phi_*^{\tau_{w^{\prime}, v}}(v x_{K^c})\right)} \cdot\frac{\exp \left(\phi_*^{\tau}(v x_{K^c})\right)}{\exp\left(\phi_*^{\tau_{v, v}}(v x_{K^c})\right)}\\
    &= \gamma_K([v],x)\exp \left(\phi_*^{\tau}(v x_{K^c}) - \phi_*^{\tau_{v, v}}(v x_{K^c})\right).
\end{align*}

Now, notice that $\phi_*^{\tau}(v x_{K^c}) - \phi_*^{\tau_{v, v}}(v x_{K^c}) = \phi_*^{\tau}(v x_{K^c})$, and the result follows.
\end{proof}

Now we can prove the main result of this subsection. The proof is a slight adaptation of the proof of \cite[Theorem 3.3]{muir2011paper} and we include it here for completeness.

\begin{theorem}
\label{thm:conformal-iff-DLR}
Let  $\phi\colon X \to \R$ be an exp-summable potential with summable variation according to an exhausting sequence $\{E_m\}_m$. Then, a measure $\mu \in \mathcal{M}(X)$ is a DLR measure for $\phi$ if and only if $\mu$ is a conformal measure for $\phi$.
\end{theorem}

\begin{proof}
Suppose first that $\mu \in \mathcal{M}(X)$ is a conformal measure for $\phi$ and let $K \in \finSet$. Begin by noticing that if $B \in \cB_{K^c}$ and, for some $w \in X_{K}$ and $x \in X$, $w x_{K^c} \in B$, then $vx_{K^c} \in B$, for every $v \in X_{K}$. As a consequence, we have that, for all $\tau\in\cE_{K}$ and all $B \in \cB_{K^c}$, $B = \tau^{-1}(B)$. 

For $w,v \in X_{K}$, consider $\tau_{w,v} \in \cE_{K}$. Thus, $\tau_{w,v}^{-1}([v])=[w]$ and, for every $B \in \cB_{K^c}$, $\tau_{w,v}^{-1}([v]\cap B) = \tau_{w,v}^{-1}([v])\cap \tau_{w,v}^{-1}(B) =  [w]\cap B$. Furthermore,
\begin{align*}
\int_{B} \mathbbm{1}_{[w]}(x) \, d\mu(x) &   = \int_{B} \mathbbm{1}_{[v]}(x) \, d(\mu \circ \tau_{w,v}^{-1})(x)\\
&   =   \int_{B} \mathbbm{1}_{[v]}(x)\exp \phi_*^{\tau_{w,v}}(x)\,d\mu(x)\\
&   =   \int_{B} \mathbbm{1}_{[v]}(x)\exp \phi_*^{\tau_{w,v}}(v x_{K^c})\,d\mu(x)\\
&= \int_{B} \mu\left(\mathbbm{1}_{[v]}(x)\exp \phi_*^{\tau_{w,v}}(v x_{K^c})\,\middle\vert\, \cB_{K^c}\right)(x) \, d\mu(x)\\
&= \int_{B}  \mu\left(\mathbbm{1}_{[v]}\,\middle\vert\, \cB_{K^c}\right)(x) \exp \phi_*^{\tau_{w,v}}(v x_{K^c})\,d\mu(x).
\end{align*}

On the other hand, 
\begin{equation*}
\int_{B} \mathbbm{1}_{[w]}(x) \, d\mu(x) = \int_{B}  \mu\left(\mathbbm{1}_{[w]}\,\middle\vert\, \cB_{K^c}\right)(x)\,d\mu(x).
\end{equation*}

Therefore, for any $w, v \in X_K$, $\mu(x)$-almost surely it holds that
\begin{equation}\label{M5}
\mu\left(\mathbbm{1}_{[v]} \,\middle\vert \, \cB_{K^c}\right)(x)\exp \phi_*^{\tau_{w,v}}(v x_{K^c})= \mu\left(\mathbbm{1}_{[w]}\,|\,\cB_{K^c}\right)(x).
\end{equation}

Now, let $A \in \cF(\mathbb{N})$ be a finite alphabet and $v \in A^{K}$. For any $w' \in A^K$, we have that $\tau_{w',v}  \in \cE_{K, A}$. Summing equation \eqref{M5} over all $w' \in A^{K}$, we obtain that $\mu(x)$-almost surely it holds that 
\begin{align}\label{M5-1}
\mu\left(\mathbbm{1}_{A^{K} \times X_{K^c}} \middle\vert \cB_{K^c}\right)(x) &= \sum_{w' \in A^{K}} \mu\left(\mathbbm{1}_{[w']}\middle\vert \cB_{K^c}\right)(x)\\ &= \mu\left(\mathbbm{1}_{[v]}\middle\vert\cB_{K^c}\right)(x)\sum_{w'\in A^{K}}\exp \phi_*^{\tau_{w',v}}(v x_{K^c}).
\end{align}

If $\{A_n\}_n$ is an exhausting sequence of finite alphabets, then $\bigcap_{n \geq 1} \left(A_n^{K}\times X_{K^c}\right)^{c} = \emptyset$. Moreover, for each $n \in \N$, 
\begin{align*}
    \int \left(1 - \mathbbm{1}_{A_n^{K} \times X_{K^c}}\right)^2 d\mu &= \int \left\vert 1 - \mathbbm{1}_{A_n^{K} \times X_{K^c}}\right\vert d\mu = \int \mathbbm{1}_{\left(A_n^{K} \times X_{K^c}\right)^c} \, \,d\mu. 
\end{align*}
Therefore, $\int \left(1 - \mathbbm{1}_{A_n^{K} \times X_{K^c}}\right)^2 \, d\mu \longrightarrow 0 \text{ as } n \to \infty$. Since conditional expectation given $\cB_{K^c}$ is a continuous linear operator on $L^2(\mu)$, 
we have $\mu\left(\mathbbm{1}_{A_n^{K}\times X_{K^c}} \middle\vert \cB_{K^c}\right) \longrightarrow \mu\left(1 \middle\vert \cB_{K^c}\right)$, $\mu(x)$-almost surely
in $L^2(\mu)$  as  $n \to \infty$. Therefore, for any fixed $v \in X_K$, there exists $n_0 \in \N$ be such that $v \in A^K_{n_0}$ and, consequently, $v \in A_n^K$, for all $n \geq n_0$. Therefore, $\mu(x)$-almost surely it holds that 

\begin{align*}
1 &= \mu\left(1\, |\, \cB_{K^c}\right)(x)\\
  &=\lim_{n\to\infty} \mu \left(\mathbbm{1}_{A_n^{K}\times X_{K^c}}\,|\,\cB_{K^c}\right)(x)\\
  &= \lim_{n \to \infty} \mu\left(\mathbbm{1}_{[v]}\,|\,\cB_{K^c}\right)(x)\sum_{w \in A_n^{K}}\exp \phi_*^{\tau_{w',v}}(v \, x_{K^c})\\
  &= \mu\left(\mathbbm{1}_{[v]}\middle\vert\cB_{K^c}\right)(x) \lim_{n \to \infty} \sum_{w' \in A_n^{K}}\exp \phi_*^{\tau_{w',v}}(v x_{K^c})\\
  &= \mu\left(\mathbbm{1}_{[v]} \middle\vert \cB_{K^c}\right)(x) \sum_{w \in X_K} \exp \phi_*^{\tau_{w',v}} (v x_{K^c}).
\end{align*}
Moreover,  equation \eqref{M5} yields that, for any $w \in X_K$, $\mu(x)$-almost surely
\begin{align*}
1 =  \mu\left(\mathbbm{1}_{[v]} \middle\vert \cB_{K^c}\right)(x) \sum_{w' \in X_{K}}\exp \phi_*^{\tau_{w',v}} (v x_{K^c}) = \frac{\mu\left(\mathbbm{1}_{[w]}\middle\vert\cB_{K^c}\right)(x)}{\exp \phi_*^{\tau_{w,v}}(v x_{K^c})}\sum_{w' \in X_{K}}\exp \phi_*^{\tau_{w',v}} (v x_{K^c}),
\end{align*}
so that, for any $w \in X_K$, $\mu(x)$-almost surely it holds that
\begin{align*}
     \mu\left([w] \middle\vert \cB_{K^c}\right)(x) = \frac{\exp\left(\phi_*^{\tau_{w,v}}(v x_{K^c})\right)}{\sum_{w^{\prime} \in X_{K}}\exp \phi_*^{\tau_{w',v}} (v x_{K^c})}.
\end{align*}
Therefore, due to Corollary \ref{charact-DLR-corol}, $\mu$ is a DLR measure.

Conversely, suppose that $\mu \in \mathcal{M}(X)$ is a DLR measure for $\phi$ and let $A \in \mathcal{F}(\N)$, $K \in \finSet$, and $\tau \in \cE_{K, A}$. For any $v \in X_{K}$ and $w = \left[\tau^{-1}([v])\right]_{K}$, due to Lemma \ref{lemma:specification-permutation}, we obtain
\begin{align*}
\mu\circ\tau^{-1}([v]) = \mu([w]) 
&=\int \mu\left([w] \middle\vert \cB_{K^c}\right)(x)\,d\mu(x)\\
&=\int \exp \phi_*^{\tau}(v x_{K^c}) \mu\left([v] \middle\vert \cB_{K^c}\right)(x) \, d\mu(x)\\
&=\int \mu \left(\exp \phi_*^{\tau}(v x_{K^c}) \mathbbm{1}_{[v]} \middle\vert \cB_{K^c}\right)(x)\,d\mu(x)\\
&=\int \exp \phi_*^{\tau}(v x_{K^c}) \mathbbm{1}_{[v]}(x)\, d\mu(x)\\
&=\int_{[v]} \exp \phi_*^{\tau}(v x_{K^c}) \, d\mu(x),
\end{align*}
which concludes the result.
\end{proof}

\subsection{DLR measures are Bowen-Gibbs measures}\label{sec:DLR-implies-Bowen-Gibbs}

This subsection is dedicated to proving that, provided some conditions, any DLR measure for a potential $\phi$ is a Bowen-Gibbs measure for $\phi$. 

\begin{proposition}\label{BowenGibbs}
Let $\phi\colon X \to \R$ be an exp-summable potential with summable variation according to an exhausting sequence $\{E_m\}_{m}$. If $\mu \in\mathcal{M}(X)$ is a DLR measure for $\phi$, then, for every $F \in \finSet$, $w \in X_F$ and $y \in X$, it holds $\mu(x)$-almost surely that
 \begin{equation*}
        \exp\left(-2V_F(\phi) - 3\Delta_F(\phi)\right) \leq  \frac{\mu\left([w] \,|\, \cB_{F^c}\right)(x)}{\exp\left(\phi_{F}(wy_{F^c}) - \log Z_F(\phi)\right)} \leq \exp\left(2V_F(\phi) + 3\Delta_F(\phi)\right).
     \end{equation*}
\end{proposition}

\begin{proof}
Let $F \in \finSet$ and $\tau \in \cE_F$. From Proposition \ref{prop:variation-of-permutations}, we have that for every $x \in X$,
\begin{equation}\label{ineq1}
   |\phi_*^{\tau}(x) - \phi^{\tau}_F(x)| \leq V_F(\phi),
\end{equation}
which, in particular, yields that, for every $x \in X$,

\begin{equation}\label{ineq:exp-phi-star-and-phi-F}
   0< \exp(-V_F(\phi))  \exp\phi_*^{\tau}(x) \leq \exp\phi^{\tau}_F(x) \leq \exp(V_F(\phi))  \exp\phi_*^{\tau}(x).
\end{equation}

For a fixed $v \in X_F$  and for every $w' \in  X_F$, the map $\tau_{w',v}$ belongs to $\cE_F$. Thus,  inequality \eqref{ineq:exp-phi-star-and-phi-F} holds for any such $\tau_{w',v}$ and, summing over all those such maps, we obtain that, for every $x \in X$,
\begin{equation*}
    \exp(-V_F(\phi))  \sum_{w' \in X_F} \exp\phi_*^{\tau_{w',v}}(x) \leq \sum_{w' \in X_F} \exp\phi^{\tau_{w',v}}_F(x) \leq \exp(V_F(\phi))  \sum_{w' \in X_F} \exp\phi_*^{\tau_{w',v}}(x).
\end{equation*}
Therefore, for every $F \in \finSet$, $v \in X_F$ and $x \in X$, we have
\begin{equation}\label{ineqwithsums}
    \exp(-V_F(\phi)) \leq \frac{\sum_{w' \in X_F}\exp\phi^{\tau_{w',v}}_F(x)}{\sum_{w' \in X_F}\exp\phi_*^{\tau_{w',v}}(x)} \leq \exp(V_F(\phi)).
\end{equation}

On the other hand, inequality \eqref{ineq1} also yields that for every $F \in \finSet$, $w, v \in X_F$ and $x \in X$, 
\begin{equation}\label{eq:inequality-with-tau-w-v}
    \exp(-V_F(\phi)) \leq \frac{\exp\phi_*^{\tau_{w,v}}(x)}{\exp\phi^{\tau_{w,v}}_F(x)} \leq \exp(V_F(\phi)).
\end{equation}

Then, from inequalities  \eqref{ineqwithsums} and \eqref{eq:inequality-with-tau-w-v}, we obtain that, for every $F \in \finSet$, $w, v \in X_F$ and $x \in X$,
\begin{equation}\label{ineq:bowen-gibbs-middle}
    \exp(-2 V_F(\phi)) \leq \frac{\sum_{w' \in X_F} \exp\phi^{\tau_{w',v}}_F(x)}{\sum_{w' \in X_F} \exp \phi_*^{\tau_{w',v}}(x)} \cdot\frac{\exp\phi_*^{\tau_{w,v}}(x)}{\exp\phi^{\tau_{w,v}}_F(x)} \leq \exp(2 V_F(\phi)).
\end{equation}

So, if $x \in [v]$, inequality \eqref{ineq:bowen-gibbs-middle} can be rewritten as 
\begin{equation}\label{ineq:bowen-gibbs-with-x-in-[v]}
    \exp(-2 V_F(\phi)) \leq \frac{\sum_{w' \in X_F} \exp\phi^{\tau_{w',v}}_F(v x_{F^c})}{\sum_{w' \in X_{F}} \exp \phi_*^{\tau_{w',v}}(v x_{F^c})} \cdot\frac{\exp\phi_*^{\tau_{w,v}}(v x_{F^c})}{\exp\phi^{\tau_{w,v}}_F(vx_{F^c})} \leq \exp(2 V_F(\phi)).
\end{equation}
Since $\mu$ is a DLR measure for $\phi$, from Corollary \ref{charact-DLR-corol} we obtain that $\mu(x)$-almost surely it holds that
\begin{equation}\label{eq:inequatily-conditional-measure-and-VF}
    \exp(-2  V_F(\phi)) \leq \mu\left([w] \,|\, \cB_{F^c}\right)(x)  \frac{\sum_{w' \in X_F} \exp\phi^{\tau_{w', v}}_{F}(v x_{F^c})}{\exp\phi^{\tau_{w,v}}_{F}(v x_{F^c})} \leq \exp(2 V_{F}(\phi)).
\end{equation}

Furthermore, notice that 
\begin{align*}
    \frac{\sum_{w' \in X_{F}} \exp\phi^{\tau_{w',v}}_{F}(v x_{F^c})}{\exp\phi^{\tau_{w,v}}_{F}(vx_{F^c})}  
    &= \frac{\sum_{w' \in X_{F}} \exp \phi_{F}(w' x_{F^c}) }{\exp \phi_{F}(w x_{F^c})},
\end{align*}
so that inequality \eqref{eq:inequatily-conditional-measure-and-VF} can be rewritten as \begin{equation}\label{eq:inequatily-conditional-measure-and-VF-2}
    \exp(-2  V_F(\phi)) \leq \mu\left([w] \,|\, \cB_{F^c}\right)(x)  \frac{\sum_{w' \in X_{F}} \exp \phi_{F}(w' x_{F^c}) }{\exp \phi_{F}(w x_{F^c})} \leq \exp(2 V_{F}(\phi)).
\end{equation}

For $F \in \finSet$ and $x \in X$,  define the following auxiliary probability measure over $X_F$:
$$\pi_{F}^x(w):=\frac{\exp \phi_{F}(wx_{F^c})}{\sum_{w' \in X_F} \exp \phi_{F}(w'x_{F^c})},
 \quad \text{ for } w \in X_F.$$

Thus, inequality \eqref{eq:inequatily-conditional-measure-and-VF-2} yields that $\mu(x)$-almost surely it holds that
\begin{equation*}
    \exp(-2  V_{F}(\phi)) \pi_{F}^{x}(w) \leq \mu\left([w] \,|\, \cB_{F^c}\right)(x) \leq \exp(2  V_{F}(\phi))  \pi_{F}^{x}(w).
\end{equation*}

Now, given $y \in X$, notice that the tail configuration $x_{F^c}$ can be replaced by $y_{F^c}$ with a penalty of $2\Delta_{F}(\phi)$ as follows
\begin{equation*}
\pi_{F}^y(w)  \exp(-2\Delta_{F}(\phi)) \leq \pi_{F}^{x}(w)  \leq \pi_{F}^y(w)\exp(2\Delta_{F}(\phi)),
\end{equation*}
so that
\begin{equation}
\label{equivMuir10}
\exp\left(-2(V_{F}(\phi) + \Delta_{F}(\phi))\right) \leq \frac{\mu\left([w] \,|\, \cB_{F^c}\right)(x)}{\pi^y_{F}(w)}\leq \exp\left(2( V_{F}(\phi) + \Delta_{F}(\phi))\right).
\end{equation}
Moreover, it is easy to verify that
\begin{equation*}
    \exp\left(-\Delta_{F}(\phi)\right) \leq \frac{\pi_{F}^y(w)}{\exp\left(\phi_{F}(wy_{F^c}) - \log Z_{F}(\phi)\right)} \leq  \exp\left( \Delta_{F}(\phi)\right).
\end{equation*}

Therefore,  for every $w \in X_{F}$, $y \in X$, it holds $\mu(x)$-almost surely that 
\begin{align*}
     \mu\left([w] \,|\, \cB_{F^c}\right)(x) &\geq  \exp\left(-2\left(V_F(\phi) + \Delta_{F}(\phi)\right)\right)\exp\left(\phi_F(wy_{F^c}) - \log Z_F(\phi) - \Delta_F(\phi)\right) \\
     &= \exp\left(-2  V_F(\phi) - 3\Delta_F(\phi)\right) \exp\left(\phi_{F}(wy_{F^c}) - \log Z_F(\phi)\right)
     \end{align*}
   and that
    \begin{align*}
         \mu\left([w] \,|\, \cB_{F^c}\right)(x) &\leq \exp\left(2\left( V_F(\phi) + \Delta_F(\phi)\right)\right) \exp\left(\phi_F (wy_{F^c}) - \log Z_F(\phi)+ \Delta_F(\phi)\right)\\
         &= \exp\left(2V_F(\phi) + 3\Delta_F(\phi)\right)\exp\left(\phi_{F}(wy_{F^c}) - \log Z_F(\phi)\right) .
     \end{align*}
     
     Thus,
     \begin{equation*}
        \exp\left(-2V_F(\phi) - 3\Delta_F(\phi)\right) \leq  \frac{\mu\left([w] \,|\, \cB_{F^c}\right)(x)}{\exp\left(\phi_{F}(wy_{F^c}) - \log Z_F(\phi)\right)} \leq \exp\left(2V_F(\phi) + 3\Delta_F(\phi)\right),
     \end{equation*}
     concluding the proof.
     \end{proof}
    
We now state the main theorem of this subsection. 
     
\begin{theorem}
\label{thm:bowgib} Let $\phi\colon X \to \R$ be an exp-summable potential with summable variation according to an exhausting sequence $\{E_m\}_m$. If $\mu$ is a DLR measure for $\phi$, then, for every $\epsilon > 0$, there exist $K \in \finSet$ and $\delta > 0$ such that for every $(K,\delta)$-invariant set $F$ and $x \in X$, it holds $\mu(x)$-almost surely that
\begin{equation*}
       \exp\left(-\epsilon\cdot|F|\right) \leq  \frac{\mu\left([w] \,|\, \cB_{F^c}\right)(x)}{\exp\left(\phi_F(x) - p(\phi)\cdot|F|\right)} \leq \exp\left(\epsilon\cdot|F|\right).
\end{equation*}

In particular, $\mu$ is a Bowen-Gibbs measure for $\phi$.
\end{theorem}

\begin{proof} Indeed, for every $\epsilon > 0$, we obtain, from Proposition \ref{prop:limit-of-VF-over-the-size-of-F}, Lemma \ref{regularityDeltak}, and Theorem \ref{thm:existence-of-pressure}, that there exist $K \in \finSet$ and $\delta > 0$ such that, for every $(K, \delta)$-invariant set $F \in \finSet$, 
\begin{equation*}
    \Delta_F(\phi) \leq \epsilon\cdot|F|,~  V_F(\phi) \leq \epsilon\cdot|F|, \text{ and }
    \left|\log Z_F(\phi) - p(\phi)|F|\right| \leq \epsilon\cdot|F|,
\end{equation*}
respectively. Considering a sufficiently large $K$ and sufficiently small $\delta$ so that the three conditions are satisfied at the same time, we obtain from Proposition \ref{BowenGibbs} that
\begin{equation*}
       \exp\left(-\epsilon \cdot|F|\right) \leq  \frac{\mu\left([w] \,|\, \cB_{F^c}\right)(x)}{\exp\left(\phi_F(x) - p(\phi) \cdot|F|\right)} \leq \exp\left(\epsilon\cdot |F|\right).
\end{equation*}

Integrating this inequality with respect to $d\mu(x)$, it follows that $\mu$ is a Bowen-Gibbs measure for $\phi$. 
\end{proof}

\subsection{Existence of conformal measures}
\label{sec:existence-of-DLR-states}

In order to guarantee that the equivalences we prove here are non-trivial, we prove the existence of a conformal measure for an exp-summable potential with summable variation in the context of a countably infinite state space over an amenable group. The strategy is to apply a version of Prokhorov's Theorem.  
     
\begin{definition}
A sequence of probability measures $\{\mu_n\}_n$ in $\mathcal{M}(X)$ is {\bf tight} if for every $\epsilon > 0$ there exists a compact set $K_{\epsilon} \subseteq X$ such that
\begin{equation*}
\mu_n(K_{\epsilon}) > 1 -\epsilon   \qquad  \text{ for all $n\in \N$.}
\end{equation*}
\end{definition}
    
We now state a version of Prokhorov's Theorem (see \cite{billingsley2013convergence,munkres2018elements}).

\begin{theorem} Every tight sequence of probability measures in $\mathcal{M}(X)$ has a weak convergent subsequence. 
\end{theorem}

Let $\phi\colon X \to \R$ be an exp-summable potential with summable variation according to an exhausting sequence $\{E_m\}_m$.  Consider $A \subseteq \N$ a finite alphabet. Then $\phi\vert_{A^G}$ is also an exp-summable potential with summable variation according to $\{E_m\}_m$ and the specification defined by equation \eqref{specification} is quasilocal. Moreover, the set of Borel probability measures on $A^G$ is compact. Then, following \cite[Comment (4.18)]{georgii2011gibbs}, for all $x \in A^G$, any accumulation point of the sequence $\left\{\gamma_{E_m}(\cdot,x)\right\}_m$, will be a DLR measure $\mu$. Finally, if we want to obtain a $G$-invariant DLR measure, for each $g\in G$, let $g\mu$ be given by $g\mu(A) = \mu(g^{-1} \cdot A)$, for any $A \in \sigmaAlg$. Notice that, for every $g \in G$, the measure $g\mu$ is also a DLR measure for $\phi\vert_{A^G}$ due to the $G$-invariance of $\gamma$ (see Corollary \ref{cor:invariance-spec}). Then it suffices to consider any accumulation point of the sequence $\left\{\frac{1}{|F_n|}\sum_{g \in F_n} g\mu\right\}_n$, for a F\o lner sequence $\{F_n\}_n$. 

Now, let $\{A_n\}_n$ in $\mathcal{F}(\N)$ be a fixed exhaustion of $\N$ and, for each $n \in \N$, denote the set of DLR measures and $G$-invariant DLR measures for $\phi^n = \phi|_{A^G_n}$ by $\mathcal{G}_n(\phi)$ and $\mathcal{G}_n^I(\phi)$, respectively. For each $n \in \N$ and each $\mu_n \in \mathcal{G}_n^I(\phi)$, consider its extension $\tilde{\mu}_n \in \mathcal{M}(X)$ given by 
\begin{equation*}\label{eq:mu-extension-infinite-alphabet}
    \tilde{\mu}_n(\cdot) = \mu_n(\cdot \cap A_n^G).
\end{equation*} The next result establishes that $\{\tilde{\mu}_n\}_n$ is tight and the reader can compare this to \cite[Lemma 5.15]{muir2011gibbs}.

\begin{lemma}\label{lemma:tightness-of-seq-induced-measures} Let $\phi\colon X \to \mathbb{R}$ be an exp-summable potential with summable variation according to some exhausting sequence $\{E_m\}_m$. Then, for any sequence $\{\mu_n\}_n$ with $\mu_n \in \mathcal{G}_n^I(\phi)$, for all $n\in \N$,  the sequence of extensions $\{\tilde{\mu}_n\}_n$ is tight.
\end{lemma}

\begin{proof}
Fix some $n \in \N$. Then, for any $a \in \N$ and any $y \in A_n^{\{1_G\}^c}$, Proposition \ref{BowenGibbs} yields that
\begin{align*}
\exp\left(-C(\phi^n) \right) \leq  \frac{\mu_n([a])}{\exp\left(\phi^n(ay) - \log Z_{E_1}(\phi^n)\right)}
       \leq \exp\left(C(\phi^n) \right),
\end{align*}
where $\phi^n = \phi|_{A_n^G}$ and $C(\phi^n) =  2 V_{E_1}(\phi^n) + 3\delta(\phi^n)$. Furthermore, 
$0 < Z_{E_1}(\phi^n) \leq Z_{E_1}(\phi^{n+1}) < \infty$ and $\{Z_{E_1}(\phi^n)\}_n$ converges monotonically to $Z_{E_1}(\phi)$. In particular, there exists $c = -\log Z_{E_1}(\phi^1)$ such that $c \geq -\log Z_{E_1}(\phi^n)$, for all $n \in \N$.

If $a \notin A_n$, then $\tilde{\mu}_n([a]) = 0$. 
On the other hand, if $a \in A_n$, then for every $y \in A_n^{\{1_G\}^c}$,
\begin{align*}
    \tilde{\mu}_n([a]) = \mu_n([a]) &\leq \exp\left(C(\phi^n) \right)\exp\left(\phi^n(ay) - \log Z_{E_1}(\phi^n)\right) \\
    &\leq\label{estimate-without-constraints} \exp\left(C(\phi) + \phi(ay) + c\right),
\end{align*}
where $C(\phi) =  2 V_{E_1}(\phi) + 3\delta(\phi)$

Now, let $\epsilon > 0$. Since $\phi$ is exp-summable, for each $m \in \N$, there must exist a finite alphabet $A_{\epsilon, m} \in \mathcal{F}(\N)$ such that
\begin{equation}\label{estimate-for-tightness-of-union}
    \sum_{b \in \N \setminus A_{\epsilon, m}} \exp \left(\sup_{x \in [b]} \phi(x) \right) < \frac{\epsilon\cdot\exp\left(-C(\phi)- c\right)}{2^m|E_m\setminus E_{m-1}|}.
\end{equation}

Let \begin{equation*}
    K_{\epsilon} = A_{\epsilon, 1}^{E_1} \times A_{\epsilon, 2}^{E_2 \setminus E_1} \times A_{\epsilon, 3}^{E_3 \setminus E_2} \times \cdots.
\end{equation*}

By Tychonoff's Theorem (see \cite{munkres2018elements}), $K_{\epsilon}$ is compact. Moreover, notice that
\begin{equation*}
K_{\epsilon} = 
\bigcap_{m=1}^{\infty} \bigcap_{g \in E_m\setminus E_{m-1}} \bigcup_{a \in A_{\epsilon, m}} [a^g], 
\end{equation*}
where $[a^g] = \{x \in X : x(g) = a\}$. Therefore, for each $n \in \N$,
\begin{align*}
    \tilde{\mu}_n\left(X \setminus K_{\epsilon}\right) 
&=  \tilde{\mu}_n\left(\bigcup_{m=1}^{\infty} \bigcup_{g \in E_m\setminus E_{m-1}} \bigcap_{a \in A_{\epsilon, m}} [a^g]^c\right)\\
    &\leq \sum_{m=1}^{\infty} \sum_{g \in E_m \setminus E_{m-1}}  \tilde{\mu}_n\left(\bigcap_{a \in A_{\epsilon, m}} [a^g]^c\right)\\
    &= \sum_{m=1}^{\infty} \sum_{g \in E_m \setminus E_{m-1}}  \tilde{\mu}_n\left(\bigsqcup_{b \in \N \setminus A_{\epsilon, m}} [b^g]\right)\\
    &= \sum_{m=1}^{\infty} \sum_{g \in E_m \setminus E_{m-1}} \sum_{b \in \N \setminus A_{\epsilon, m}}  \tilde{\mu}_n \left(\left[b^g\right]\right).
\end{align*}
Since all the measures considered here are $G$-invariant, it follows that,  for any $y \in A_n^{\{1_G\}^c}$,
\begin{align*}
     \tilde{\mu}_n\left(X \setminus K_{\epsilon}\right)  &\leq  \sum_{m=1}^{\infty}\sum_{g \in E_m \setminus E_{m-1}} \sum_{b \in \N \setminus A_{\epsilon, m}} \tilde{\mu}_n\left([b]\right)\\
    &\leq \sum_{m=1}^{\infty} \sum_{g \in E_m \setminus E_{m-1}} \sum_{b \in \N \setminus A_{\epsilon, m}} \exp\left(C(\phi) + \phi(b y) + c\right)\\
    &= \sum_{m=1}^{\infty} \sum_{g \in E_m \setminus E_{m-1}} \exp\left(C(\phi) + c\right) \sum_{b \in \N \setminus A_{\epsilon, m}}  \exp(\phi(by))\\
     &< \sum_{m=1}^{\infty} \sum_{g \in E_m \setminus E_{m-1}} \exp\left(C(\phi) + c\right)\frac{\epsilon \cdot\exp\left(-C(\phi) - c\right)}{2^m|E_m\setminus E_{m-1}|}\\
     &= \sum_{m=1}^{\infty} \sum_{g \in E_m \setminus E_{m-1}} \frac{\epsilon}{2^m|E_m\setminus E_{m-1}|}\\
     %
     %
     &= \epsilon,
\end{align*}
where the fifth line follows from estimate \eqref{estimate-for-tightness-of-union}. Therefore, for all $n \in \N$, $\tilde{\mu}_n\left(X \setminus K_{\epsilon}\right) < \epsilon$, so that $\tilde{\mu}_n\left(K_{\epsilon}\right) = 1 - \tilde{\mu}_n\left(K_{\epsilon}^c\right) > 1 - \epsilon$, which proves the tightness of $\{\tilde{\mu}_n\}_n$.
\end{proof}
 
We have proven that for each sequence $\{\mu_n\}_{n}$ with $\mu_n \in \mathcal{G}^I_n(\phi)$,  the sequence $\{\tilde{\mu}_n\}_{n}$ of their extensions is tight. Then, the existence of at least one accumulation point is guaranteed by Prokhorov's Theorem. Let's see that an arbitrary accumulation point, which we will denote by $\tilde{\mu}$, is conformal for $\phi$ and, moreover, that it is $G$-invariant.

\begin{theorem}\label{thm:non-emptyness-of-invariant-Gibbs} Let $\phi\colon X \to \R$ be an exp-summable potential with summable variation according to an exhausting sequence $\{E_m\}_{m}$. Then, the set of $G$-invariant DLR measures for $\phi$ is non-empty.
\end{theorem}
    
\begin{proof} 
Let $\{\mu_n\}_{n}$ be such that, for each $n \in \N$, $\mu_n$ is a $G$-invariant conformal measure for $\phi^n:   A_n^G \to \R$ (or, equivalently, $\mu_n$ is a $G$-invariant DLR measure for $\phi^n$). Thus, for each $n \in \N$, any $K \in \finSet$, and any $\tau \in \cE_{K, A_n}$, 
\begin{equation}\label{eq:mu-n-conformal-existence}
    \exp\left((\phi^n)^{\tau_n}_*\right) = \frac{d (\mu_n \circ (\tau_n)^{-1})}{d\mu_n},
\end{equation}
where $\tau_n = \tau|_{A^G_n}$. This yields that
\begin{equation*}
    \exp(\phi_*^{\tau}) = \frac{d( \tilde{\mu}_n \circ \tau^{-1})}{d\tilde{\mu}_n}.
\end{equation*}
Indeed, let $\psi\colon X \to \R$ be a bounded continuous potential. Observe that, for $\tau \in \cE_{K, A_n}$, $(\phi^n)^{\tau_n}_* = (\phi_*^\tau)|_{A_n^G}$. Moreover, for every $B \in \sigmaAlg$, since $\tau_n^{-1}(A^G_n) = A^G_n$  and $\tilde{\mu}_n(X\setminus A_n^G) = 0$, we have that $\widetilde{{\mu_n\circ \tau_n^{-1}}}(B) = \tilde{\mu}_n(\tau^{-1}(B))$. Then, we obtain
\begin{align*}
\int \psi d(\tilde{\mu}_n \circ \tau^{-1}) &= \int \psi d(\widetilde{\mu_n \circ \tau_n^{-1}})\\
&= \int \psi^n d(\widetilde{\mu_n \circ \tau_n^{-1}})\\
&= \int \psi^n d(\mu_n \circ \tau_n^{-1})\\
&= \int \psi^n \exp\left((\phi^n)^{\tau_n}_*\right) d\mu_n\\
&= \int \psi^n \exp\left((\phi^n|_{A_n^G})^{\tau}_*\right) d\mu_n\\
&= \int \psi \exp(\phi_*^{\tau}) d\tilde{\mu}_n,
\end{align*}
where $\psi^n = \psi|_{A_n^G}$.

Furthermore, Lemma \ref{lemma:tightness-of-seq-induced-measures} guarantees that the sequence of induced measures $\{\tilde{\mu}_n\}_{n}$ is tight and we can apply Prokhorov's Theorem to guarantee the existence of a limit point for some subsequence $\{\mu_{n_k}\}_k$, which we denote by $\tilde{\mu}$. Now, we are going to prove that $\tilde{\mu}$ is a conformal measure for $\phi$. For that, consider a bounded continuous potential $\psi\colon X \to \R$, $A \in \mathcal{F}(\N)$, $K \in \finSet$, and  $\tau \in \cE_{K, A}$. 
 Then,  
\begin{align*}
\int \psi \, d(\tilde{\mu} \circ \tau^{-1}) &= \int \psi\circ \tau \, d\tilde{\mu}\\
    &= \lim_{k\to\infty} \int \psi \circ \tau \, d\tilde{\mu}_{n_k}\\
    &= \lim_{k\to\infty} \int \psi \, d(\tilde{\mu}_{n_k} \circ \tau^{-1})\\
    &= \lim_{k\to\infty} \int \psi \exp \phi_*^{\tau} \, d\tilde{\mu}_{n_k}\\
    &= \int \psi \exp \phi_*^{\tau} \, d\tilde{\mu},
\end{align*}
where the fourth equality follows from the fact that for $k$ large enough, $A \subseteq A_{n_k}$, and the last equality follows from weak convergence and the fact that $\psi \exp \phi_*^{\tau}$ is a continuous and bounded function. Indeed, first notice that $\phi_*^{\tau}$ is a uniform limit of continuous functions that are bounded from above, since $\phi$ is exp-summable. Therefore, the same holds for $\phi_*^{\tau}$, so that $\exp(\phi_*^{\tau})$ is continuous and bounded (from above and below). Since $A$, $K$, and $\tau$ are arbitrary, this proves that $\tilde{\mu}$ is conformal for $\phi$ and, therefore, DLR for $\phi$.

It remains to show that $\tilde{\mu}$ is $G$-invariant. For that, notice that, due to the weak convergence, for any $B \in \sigmaAlg$,
\begin{equation*}
\tilde{\mu}(g \cdot B) = \lim_{k \to \infty} \tilde{\mu}_{n_k}(g \cdot B) = \lim_{k \to \infty} \tilde{\mu}_{n_k}(B) = \mu(B), 
\end{equation*}
where we have used that, for each $k \in \N$, $\tilde{\mu}_{n_k}$ is $G$-invariant due to $G$-invariance of $A_n^G$ and to the fact that $\mu_{n_k}$ is $G$-invariant.
\end{proof}

\subsection{Finite entropy Bowen-Gibbs measures are equilibrium measures}\label{sec:Bowen-Gibbs-are-equilibrium}

Thus far, we have proven that if $\phi\colon X \to \R$ is an exp-summable potential with summable variation, then a measure $\mu \in \mathcal{M}(X)$ is a DLR measure if and only if it is a conformal measure. Also, if $\mu$ is a DLR measure, then $\mu$ is also a Bowen-Gibbs measure. For Bowen-Gibbs measures, we begin by exploring some equivalent hypothesis to having $H_F(\mu) < \infty$ for every $F \in \finSet$, or, equivalently, to have finite Shannon entropy at the identity element. This will allow us to assume, indistinctly,  that the energy of the potential is finite. The following lemma generalizes \cite[Lemma 3.4] {mauldin2001gibbs}.

\begin{proposition}\label{prop:equivalentes-of-finiteness-of-integral} Let $\phi\colon X \to \R$ be an exp-summable potential with summable variation according to an exhausting sequence $\{E_m\}_m$. Then, if $\mu \in \mathcal{M}(X)$ is a Bowen-Gibbs measure for $\phi$, the following conditions are equivalent:
\begin{itemize}
    \item[$i)$] $ \int \phi d\mu > -\infty$;
    \item[$ii)$] $\sum_{a\in \N} \sup \phi([a])\exp\left(\sup \phi([a])\right) > -\infty$; and
    \item[$iii)$] $H(\mu) < \infty$.
\end{itemize}
\end{proposition}

\begin{proof} Begin by noticing that, since $\mu$ is a Bowen-Gibbs measure for $\phi$, we have that, in particular, for $\epsilon = 1$, there exist $K \in \finSet$, $\delta > 0$, and a $(K,\delta)$-invariant set $F \in \finSet$ with $1_G \in F$ such that, for every $x \in X$, it holds that 
\begin{equation}\label{eq:bowen-gibbs-one-side-1}
 \exp\left(-|F|(1 + p(\phi)) + \sup\phi_F([x_F])\right) \leq \mu([x_F]) \leq \exp\left(-|F|(-1 + p(\phi)) + \sup\phi_F([x_F])\right).
\end{equation}
We now prove that  $i) \implies iii) \implies ii) \implies i)$.

[$i) \implies iii)$] Notice that, since $\phi$ has summable variation according to $\{E_m\}_m$, then, in particular, $\phi$ has finite oscillation. Therefore, the result follows directly from Proposition \ref{prop:finite-Shannon-entropy}, disregarding whether $\mu$ is a Bowen-Gibbs measure for $\phi$ or not.

[$iii) \implies ii)$] Begin by noticing that, due to standard properties of Shannon entropy,
$H(\mu) \leq H_F(\mu) \leq |F|H(\mu)$. Then,
\begin{align*}
-\infty &< -H_F(\mu)\\
&= \sum_{x_F \in X_F} \mu([x_F]) \log \mu([x_F])\\
&\leq \sum_{x_F \in X_F} \mu([x_F])\left(-|F|(-1 + p(\phi)) + \sup\phi_F([x_F])\right)    \\
&=-|F|(1 + p(\phi)) + \sum_{x_F \in X_F}\mu([x_F])\sup\phi_F([x_F]).   
\end{align*}

Thus,
\begin{align*}
     -\infty &< \sum_{x_F \in X_F}\mu([x_F])\sup\phi_F([x_F])\\
     &\leq \sum_{x_F \in X_F}\exp\left(-|F|(-1+p(\phi)) + \sup \phi_F([x_F])\right)\cdot \sup\phi_F([x_F])\\
     &= \exp\left(-|F|(-1+p(\phi))\right)\sum_{x_F \in X_F}\exp\left(\sup \phi_F([x_F])\right)\cdot \sup\phi_F([x_F]),
\end{align*}
so that 
\begin{equation*}
     -\infty < \sum_{x_F \in X_F}\exp\left(\sup \phi_F([x_F])\right) \sup\phi_F([x_F]).
\end{equation*}

Also, for each $x_F \in X_F$, 
\begin{equation*}
    \sup \phi_F([x_F]) \geq \inf \phi_F([x_F]) \geq \sum_{g \in F} \inf\left(\phi_{\{g\}}([x_F])\right) \geq \sum_{g \in F} \inf\left(\phi_{\{g\}}([x_g])\right).
\end{equation*}

Now, due to exp-summability, without loss of generality we can assume that $\phi(x) \leq 0$, for all $x \in X$, so $\sup\phi_F([x_F]) \leq \sup\phi_F([x_{1_G}]) \leq \sup\phi([x_{1_G}]) \leq 0$. Then, abbreviating $\phi_{\{g\}}$ by $\phi_{g}$, we obtain that
\begin{align*}
     -\infty &< \sum_{x_F \in X_F} \sup\phi_F([x_F])\exp\left(\sup \phi_F([x_F])\right)\\
     &\leq \sum_{x_F \in X_F} \sup\phi_F([x_F])\prod_{g\in F} \exp\left(\inf \phi_g([x_g])\right)\\
     &\leq \sum_{x_F \in X_F} \sup\phi_F([x_F])\prod_{g\in F} \exp\left(\sup \phi_g([x_g]) - \delta(\phi)\right)\\
     &= \exp\left(- \delta(\phi)|F|\right)\sum_{x_F \in X_F} \sup\phi_F([x_F])\prod_{g\in F} \exp\left(\sup \phi_g([x_g])\right)\\
     &= \exp\left(- \delta(\phi)|F|\right)\sum_{x_F \in X_F} \sup\phi_F([x_F])\exp\left(\sup \phi([x_{1_G}])\right)\prod_{g\in F\setminus \{1_G\}} \exp\left(\sup \phi_g([x_g])\right)\\
     &\leq \exp\left(- \delta(\phi)|F|\right)\sum_{x_{F} \in X_F} \sup\phi_F([x_{1_G}])\exp\left(\sup \phi([x_{1_G}])\right)\prod_{g\in F\setminus \{1_G\}} \exp\left(\sup \phi_g([x_g])\right)\\
     &= \exp\left(- \delta(\phi)|F|\right)\sum_{x_{1_G} \in \N} \sup\phi_F([x_{1_G}])\exp\left(\sup \phi([x_{1_G}])\right)\sum_{x_{F\setminus \{1_G\}}}\prod_{g\in F\setminus \{1_G\}} \exp\left(\sup \phi_g([x_g])\right)\\
     &\leq  \exp\left(- \delta(\phi)|F|\right)\sum_{x_{1_G} \in \N} \sup\phi([x_{1_G}])\exp\left(\sup \phi([x_{1_G}])\right)\sum_{x_{F\setminus \{1_G\}}}\prod_{g\in F\setminus \{1_G\}} \exp\left(\sup \phi_g([x_g])\right).
\end{align*}
Moreover, notice that if $m = |F|-1$ and $g_1,\cdots, g_m$ is an enumeration of $F\setminus \{1_G\}$, then
\begin{align*}
    \sum_{x_{F\setminus \{1_G\}}}\prod_{g\in F\setminus \{1_G\}} \exp\left(\sup \phi_g([x_g])\right) &  = \sum_{x_{g_1}} \cdots \sum_{x_{g_m}} \exp\left(\sup \phi_{g_1}([x_{g_1}])\right) \cdots \exp\left(\sup \phi_{g_m}([x_{g_m}])\right)  \\
    &  = \sum_{x_{g_1}} \exp\left(\sup \phi_{g_1}([x_{g_1}])\right) \cdots \sum_{x_{g_m}} \exp\left(\sup \phi_{g_m}([x_{g_m}])\right)  \\
    &   = \prod_{g\in F\setminus \{1_G\}}\sum_{x_{g} \in X_g} \exp\left(\sup \phi_g([x_g])\right),
\end{align*}
so that 
\begin{align*}
    -\infty &<  \sum_{x_{1_G} \in \N} \sup\phi([x_{1_G}])\exp\left(\sup \phi([x_{1_G}])\right)\prod_{g\in F\setminus \{1_G\}}\sum_{x_{g} \in X_g} \exp\left(\sup \phi_g([x_g])\right)\\
    &= \sum_{x_{1_G} \in \N} \sup\phi([x_{1_G}])\exp\left(\sup \phi([x_{1_G}])\right)\prod_{g\in F\setminus \{1_G\}}Z_{g}(\phi)\\
    &= \sum_{x_{1_G} \in \N} \sup\phi([x_{1_G}])\exp\left(\sup \phi([x_{1_G}])\right)\prod_{g\in F\setminus \{1_G\}}Z_{1_G}(\phi)\\
    &= \sum_{x_{1_G} \in \N} \sup\phi([x_{1_G}])\exp\left(\sup \phi([x_{1_G}])\right)Z_{1_G}(\phi)^{|F| - 1}.
\end{align*}
Therefore, 
\begin{equation*}
    \sum_{x_{1_G} \in \N} \sup\phi([x_{1_G}])\exp\left(\sup \phi([x_{1_G}])\right) > -\infty.
\end{equation*}

[$ii) \implies i)$] Indeed, 
\begin{align*}
    \int \phi d\mu &\geq    \sum_{a \in \N} \inf \phi([a])\mu([a]) \\
    &=  \sum_{a \in \N} \inf \phi([a])\sum_{\substack{x_F: x_F(1_G) = a}}\mu([x_F]) \\
    &   \geq    \sum_{a \in \N} \inf \phi([a])\sum_{\substack{x_F: x_F(1_G) = a}} \exp\left(- |F|(1 +  p(\phi)) + \sup\phi_F ([x_F])\right) \\
     &=   \exp\left(- |F|(1 +  p(\phi))\right)\sum_{a \in \N} \inf \phi([a])\sum_{\substack{x_F: x_F(1_G) = a}} \exp\left(\sup\phi_F ([x_F])\right) \\
         &   \geq    \exp(-|F|(1+ p(\phi)))\sum_{a \in \N} \inf \phi([a])\sum_{\substack{x_F: x_F(1_G) = a}}\exp\left(\sum_{g \in F}\sup\phi_g([x_F])\right) \\
      &\geq    \exp\left(-|F|(1 + p(\phi)\right) \sum_{a \in \N} \inf \phi([a])\sum_{\substack{x_F: x_F(1_G) = a}} \prod_{g \in F}\exp\left(\sup\phi[x(g)]\right)\\
    &=   \exp\left(-|F|(1 + p(\phi)\right) \sum_{a \in \N} \inf \phi([a]) \exp\left(\sup\phi[a]\right) \sum_{x_{F\setminus\{1_G\}}} \prod_{g \in F\setminus \{1_G\}}\exp\left(\sup\phi[x(g)]\right).
 \end{align*} 

Notice that, due to the same argument as in the proof of $[iii) \implies ii)]$, we have that
\begin{equation*}
    \sum_{x_{F\setminus\{1_G\}}} \prod_{g \in F\setminus \{1_G\}}\exp\left(\sup\phi[x(g)]\right) = Z_{1_G}(\phi)^{|F|-1}.
\end{equation*}
Therefore, since $\exp\left(-|F|(1 + p(\phi)\right) Z_{1_G}(\phi)^{|F|-1} > 0$, it suffices to prove that 
\begin{align*}
   \sum_{a \in \N} \inf \phi([a]) \exp\left(\sup\phi[a]\right) > -\infty,
\end{align*}
but this is true since
\begin{align*}
   \sum_{a \in \N} \inf \phi([a]) \exp\left(\sup\phi[a]\right) &    \geq \sum_{a \in \N} (\sup \phi([a]) - \delta(\phi) )\exp\left(\sup\phi[a]\right)  \\
   & = \sum_{a \in \N} \sup \phi([a]) \exp\left(\sup\phi[a]\right) - \delta(\phi) \sum_{a \in \N} \exp\left(\sup\phi[a]\right)  \\
   & = \sum_{a \in \N} \sup \phi([a]) \exp\left(\sup\phi[a]\right) - \delta(\phi) \cdot Z_{1_G}(\phi)\\
   & > -\infty.
\end{align*}

\end{proof}

We now proceed to prove that Bowen-Gibbs measures with finite Shannon entropy at the identity are equilibrium measures.

\begin{theorem}\label{prop:bowen-gibbs-implies-equlibrium}
Let $\phi\colon X \to \R$ be an exp-summable potential with summable variation according to an exhausting sequence $\{E_m\}_m$. If $\mu \in \mathcal{M}(X)$ is a $G$-invariant Bowen-Gibbs measure for $\phi$ and $H(\mu) < \infty$, then $\mu$ is an equilibrium measure for $\phi$.
\end{theorem}

\begin{proof} Since $\mu$ is a Bowen-Gibbs measure for $\phi$, for every $\epsilon > 0$, there exist $K \in \finSet$ and $\delta > 0$, such that for every $(K,\delta)$-invariant set $F \in \finSet$ and $x \in X$, 
\begin{equation}\label{eq:bowen-gibbs-for-proof-of-equilibrium}
       \exp\left(-\epsilon\cdot|F|\right) \leq  \frac{\mu([x_F])}{\exp\left(\phi_F(x) - p(\phi)\cdot|F|\right)} \leq \exp\left(\epsilon\cdot|F|\right).
\end{equation}

Moreover, notice that, for every $x \in X$ and $F \in \finSet$,
\begin{equation}\label{eq:supremum-then-integral}
    \sup\phi_F([x_F]) \leq \phi_F(x) + \Delta_F(\phi)  = \sum_{g \in F} \phi(g \cdot x) + \Delta_F(\phi).
\end{equation}
Therefore,
\begin{align*}
\lim_{F \to G} \frac{1}{|F|}\int{\sup\phi_{F}([x_F])}\, d\mu(x) & \leq \lim_{F \to G} \frac{1}{|F|}\int{\left(\sum_{g \in F} \phi(g \cdot x) + \Delta_F(\phi)\right)} \,d\mu(x)\\
&   = \lim_{F \to G} \frac{1}{|F|}\left(\sum_{g \in F} \int \phi(x) \, d\mu(x)\right) +  \lim_{F \to G} \frac{\Delta_F(\phi)}{|F|} \\
&   = \lim_{F \to G} \frac{1}{|F|}\left(|F| \int{\phi} \, d\mu\right)\\
&   =   \int{\phi}d\mu,
\end{align*}
where the second line follows from the $G$-invariance of $\mu$ and the third line follows from Lemma \ref{regularityDeltak}.

On the other hand, after taking logarithm in equation \eqref{eq:bowen-gibbs-for-proof-of-equilibrium} and dividing by $|F|$, we obtain
\begin{equation*}
    -\epsilon \leq \frac{\log \mu([x_F]) - \phi_F(x)}{|F|} + p(\phi)\leq \epsilon.
\end{equation*}

Thus, for every $x \in X$ and every $(K,\delta)$-invariant set $F \in \finSet$, 
\begin{equation*}
    p(\phi) \leq \frac{- \log \mu([x_F]) + \phi_F(x)}{|F|} + \epsilon \leq \frac{- \log \mu([x_F]) + \sup\phi_F([x_F])}{|F|} + \epsilon.
\end{equation*}

Integrating the last equation with respect to $\mu$, we get
\begin{align*}
p(\phi) &\leq    \frac{-1}{|F|} \sum_{x_F \in X_F}\mu([x_F])\log\mu([x_F]) + \frac{1}{|F|} \int{\sup\phi_F([x_F])}d\mu + \epsilon\\
&=  \frac{1}{|F|} H_F(\mu) + \frac{1}{|F|} \int{\sup\phi_F([x_F])}d\mu + \epsilon.
\end{align*}

Therefore, if we take limit as $F$ becomes more and more invariant, we have that
\begin{align*}
p(\phi) \leq    h(\mu) + \lim_{F \to G}\frac{1}{|F|} \int{\sup\phi_F([x_F])}d\mu + \epsilon   \leq    h(\mu) + \int{\phi}d\mu + \epsilon,
\end{align*}
where the last inequality follows from  inequality \eqref{eq:supremum-then-integral}. Since $\epsilon>0$ is arbitrary, we obtain that
\begin{equation*}
    p(\phi) \leq  h(\mu) + \int{\phi}d\mu.
\end{equation*}
The reverse inequality follows from Proposition \ref{prop:half-of-pv} and this concludes the proof. 
\end{proof}

\subsection{Equilibrium measures are DLR measures}\label{sec:equilibrium-states-are-DLR}

In \S\ref{sec:existence-of-DLR-states}, we proved that if $\phi\colon X \to \R$ is an exp-summable potential with summable variation according to an exhausting sequence $\{E_m\}_m$, then the set of $G$-invariant DLR measures for $\phi$ is non-empty. Throughout this section, fix a $G$-invariant $\nu \in \gibbs$.

Given $E \in \finSet$ and $\mu \in \mathcal{M}_G(X)$,  denote by $f_{\mu,E}$ the Radon-Nikodym derivative of $\mu\vert_E$ with respect to $\nu\vert_E$, where $\mu\vert_E$ and $\nu\vert_E$ denote the restrictions of $\mu$ and $\nu$ to $\sigmaAlg_E$, respectively. More precisely, for every $x \in X$, 
\begin{equation}\label{eq:radon-nikodym-derivative}
f_{\mu, E}(x) = \sum_{w \in X_E} \frac{\mu([w])}{\nu([w])} \mathbbm{1}_{[w]}(x)
.\end{equation}

Notice that $f_{\mu,E}$ is well-defined, because any DLR measure for $\phi$, in our context, is fully supported. Moreover, we can understand it as the pointwise limit of the simple functions $f_{\mu,E}^n = \sum_{w \in X_E \cap A_n^G} \frac{\mu([w])}{\nu([w])} \mathbbm{1}_{[w]}$, where $\{A_n\}_n$ is a fixed exhausting sequence of finite alphabets.

Consider the function $\psi\colon [0,\infty) \to [0,\infty)$ given by $\psi(x) = 1 - x + x\log x$, where $0\log(0) = 0$. Define, for each $n\in \N$ and $E \in \finSet$, the simple function $I^n_{\mu,E} := \sum_{w \in X_E \cap A_n^G} \psi\left(\frac{\mu([w])}{\nu([w])}\right) \mathbbm{1}_{[w]}$. Notice that $0 \leq I^n_{\mu,E}(x) \leq I^{n+1}_{\mu,E}(x)$, so we can define a measurable function $I_{\mu,E}$ by considering the pointwise limit $I_E(x) := \lim_{n \to \infty}I^n_E(x)$ in $[0,\infty]$.

When there is no ambiguity, we will omit the subscript $\mu$ from the previous notations. 

Observe that, by the Monotone Convergence Theorem,
\begin{equation*}
    \lim_{n\to\infty} \int I^n_E d\nu = \int \lim_{n\to\infty}  I^{n}_E d\nu = \int I_E d\nu \in [0,\infty].
\end{equation*}

In addition, 
\begin{align*}
    H^n_E(\mu \vert \nu) := \int I^n_E d\nu = \sum_{w \in X_E \cap A_n^G} \left(\nu([w]) - \mu([w]) + \mu([w])\log\left(\frac{\mu([w])}{\nu([w])}\right)\right),
\end{align*}
so that 
\begin{align*}
    \int I_E d\nu &= \lim_{n\to \infty}\int I^n_E d\nu\\
    &= \sum_{w \in X_E} \left(\nu([w]) - \mu([w]) + \mu([w])\log\left(\frac{\mu([w])}{\nu([w])}\right)\right)\\
    &=  \sum_{w \in X_E }  \mu([w])\log\left(\frac{\mu([w])}{\nu([w])}\right).
\end{align*}

We define the {\bf relative entropy} of a measure  $\mu \in \mathcal{M}_G(X)$ with respect to $\nu$ to be
\begin{equation*}
     H_E(\mu \vert \nu) := \int I_{\mu,E} d\nu,
\end{equation*}
when $E \in \finSet$, and $0$ if $E = \emptyset$.  Notice that, \textit{a priori}, $H_E(\mu \vert \nu) \in [0,\infty]$. Also, if $\mu \in \mathcal{M}_G(X)$, then $H_{Eg}(\mu \vert \nu) = H_E(\mu \vert \nu)$ for every $g \in G$.

\begin{lemma}\label{rmk:H-F-increasing} Let $E, F \in \finSet$ be such that $E \subseteq F$ and $\mu \in \mathcal{M}(X)$. Then, for every $n \in \N$, $H^n_E(\mu \vert \nu) \leq H^n_F(\mu \vert \nu)$. Moreover, $H_E(\mu \vert \nu) \leq H_F(\mu \vert \nu)$.
\end{lemma}

\begin{proof}
Fix $n \in \N$. First, observe that $f^n_{\mu,E} = \nu[f^n_{\mu,F} \,\vert\, \sigmaAlg_E]$. Indeed, it suffices to prove that for any $v \in X_E$,
\begin{equation}\label{eq:tower-property-lemma}
    \int_{[v] \cap A_n^G} f^n_E \, d\nu = \int_{[v] \cap A_n^G} f^n_F \, d\nu,
\end{equation}
since the supports of $f^n_E$ and $f^n_F$ are contained in $A_n^G$ and $\sigmaAlg_E$ is generated by cylinder sets of this form. If $v \notin A_n^E$, then both sides of equation \eqref{eq:tower-property-lemma} are $0$ and the result is proven. Otherwise, if $v \in A_n^E$, then
\begin{align*}
    \int_{[v]\cap A_n^G} f^n_F \, d\nu &=  \int_{[v]\cap A_n^G} \sum_{w \in X_F \cap A_n^G} \frac{\mu([w])}{\nu([w])} \mathbbm{1}_{[w]} d\nu\\
    &= \int_{A_n^G} \sum_{w \in X_{F\setminus E} \cap A_n^G} \frac{\mu([vw_{F\setminus E}])}{\nu([vw_{F\setminus E}])} \mathbbm{1}_{[vw_{F \setminus E}]} d\nu\\
    &=  \sum_{w \in X_{F\setminus E} \cap A_n^G} \frac{\mu([vw_{F\setminus E}])}{\nu([vw_{F\setminus E}])}\int_{A_n^G}\mathbbm{1}_{[vw_{F \setminus E}]} d\nu\\
    &=  \sum_{w \in X_{F\setminus E} \cap A_n^G} \mu([vw_{F\setminus E}])\\
    &= \mu([v]\cap A_n^G)   \\
     &= \int_{[v]\cap A_n^G}  d\mu\\
     &=  \int_{[v]\cap A_n^G} f^n_E \, d\nu.
\end{align*}

 Thus,
\begin{align*}
H^n_E(\mu \vert \nu) &= \int{f^n_{E} \log f^n_{E}} \,d\nu \\
&= \int{\nu(f^n_{F}\vert\sigmaAlg_E) \log \nu(f^n_{F} \vert \sigmaAlg_E)}d\nu \\
&\leq  \int{\nu(f^n_{F} \log f^n_{F} \vert \sigmaAlg_E)}d\nu\\
&=  H^n_F(\mu \vert \nu),
\end{align*}
where the inequality follows from Jensen's inequality for conditional expectations. Finally, observe that
$$
H_E(\mu \vert \nu) = \lim_{n\to\infty} H^n_E(\mu \vert \nu) \leq \lim_{n\to\infty} H^n_F(\mu \vert \nu) = H_F(\mu \vert \nu).
$$
\end{proof}

\begin{proposition}\label{entropy-0} Let $\phi\colon X \to \R$ be an exp-summable potential with summable variation according to an exhausting sequence $\{E_m\}_m$ and $\mu \in \mathcal{M}_G(X)$. Then, $H_E(\mu \vert \nu) < \infty$ for every $E \in \finSet$. Moreover,  if  $\int \phi \, d\mu > -\infty$,
$$
h(\mu \,|\, \nu) := \lim_{F \to G} \frac{1}{|F|} H_F(\mu \vert \nu) = p(\phi) - \left(h(\mu) + \int \phi \,d\mu\right).
$$
\end{proposition}

\begin{proof} Let $E \in \finSet$. Since $\nu$ is a DLR measure for $\phi$, by Theorem \ref{thm:bowgib}, $\nu$ is a Bowen-Gibbs measure for $\phi$. Then, for every $\epsilon > 0$, there exist $K \in \finSet$ and $\delta > 0$ such that for all $(K, \delta)$-invariant set $F \in \finSet$, the following conditions hold at the same time:
\begin{equation*}
 \left|h(\mu) - \frac{H_F(\mu)}{|F|}\right| \leq \epsilon
\end{equation*}
and
\begin{equation*}
       \exp\left(-\epsilon\cdot|F|\right) \leq  \frac{\nu([x_F])}{\exp\left(\sup\phi_F(x) - p(\phi)\cdot|F|\right)} \leq \exp\left(\epsilon\cdot|F|\right).
\end{equation*}

Observe that, by considering the lower bound of the equation above,
\begin{align*}
    -\sum_{x_F \in X_F \cap A_n^G} \mu([x_F])\log(\nu([x_F])) & \leq  - \sum_{x_F \in X_F \cap A_n^G} \mu([x_F])\left(\sup\phi_F([x_F]) - p(\phi) |F| - \epsilon|F|\right)\\
    &= (p(\phi) + \epsilon)|F| - \sum_{x_F \in X_F\cap A_n^G} \mu([x_F])\sup\phi_F([x_F])\\
    &\leq (p(\phi) + \epsilon)|F| - \int_{A_n^G} \phi_F d\mu \\
    &=  \left(p(\phi) + \epsilon - \int_{A_n^G} \phi d\mu\right)|F|,
\end{align*}
for any $(K, \delta)$-invariant set $F$. Then, we have that
\begin{align*}
     H_F(\mu \vert \nu) &  = \lim_{n \to \infty} \left(H^n_F(\mu \vert \nu) - H^n_F(\mu)\right) + \lim_{n \to \infty} H^n_F(\mu)   \\
    &= \lim_{n\to\infty} - \sum_{x_F \in X_F \cap A_n^G} \mu([x_F])\log(
    \nu([x_F])) + H_F(\mu)\\
    &\leq  \lim_{n\to\infty} \left(p(\phi) + \epsilon - \int_{A_n^G} \phi d\mu\right)|F| + (h(\mu) + \epsilon)|F|\\
     &=  \left(p(\phi) + h(\mu) - \int \phi d\mu + 2\epsilon\right)|F| + H_F(\mu) < \infty,
\end{align*}
where $H^n_F(\mu) := -\sum_{x_F \in X_F \cap A_n^G} \mu([x_F])\log(\mu([x_F]))$.

First, observe that for any $E$, we can find a $(K,\delta)$-invariant set $F$ such that $E \subseteq F$. Then, by Lemma \ref{rmk:H-F-increasing}, $H_E(\mu \vert \nu) \leq H_E(\mu \vert \nu) < \infty$. Second, for any $(K,\delta)$-invariant set $F$,
$$
\frac{H_F(\mu \vert \nu)}{|F|} \leq p(\phi) + \left(h(\mu) - \int \phi d\mu\right) + 2\epsilon.
$$

Finally, by considering the upper bound given by the definition of Bowen-Gibbs measure and using a similar argument, we obtain that
$$
\frac{H_F(\mu \vert \nu)}{|F|} \geq p(\phi) + \left(h(\mu) - \int \phi d\mu\right) - 2\epsilon.
$$

Since $\epsilon$ was arbitrary, we conclude that
$$
\lim_{F \to G} \frac{H_F(\mu \vert \nu)}{|F|} = p(\phi) - \left(h(\mu) + \int \phi d\mu\right).
$$
\end{proof}

In particular, given $\phi\colon X \to \R$ an exp-summable potential with summable variation according to an exhausting sequence $\{E_m\}_m$, a $G$-invariant measure $\mu$ is an equilibrium measure for $\phi$ if and only if $h(\mu\,|\,\nu) = 0$, for some (or every) DLR measure $\nu$. The next proposition is a generalization of Step $1$ in the proof of \cite[Theorem 15.37]{georgii2011gibbs}.

\begin{proposition}\label{prop:step1-Georgii}
Let $\phi\colon X \to \R$ be an exp-summable potential with summable variation according to an exhausting sequence $\{E_m\}_m$ and $\mu \in \mathcal{M}_G(X)$ be an equilibrium measure for $\phi$. Then, for every $\alpha > 0$ and $K\in \finSet$, there exists $E \in \finSet$ such that $K \subseteq E$ and 
\begin{equation*}
0 \leq H_E(\mu|\nu)-H_{E \setminus K}(\mu|\nu) \leq \alpha.    
\end{equation*}
\end{proposition}

\begin{proof}
Pick $\delta > 0$ small enough so that every $(K,\delta)$-invariant set $F \in \finSet$ satisfies $\mathrm{Int}_K(F) \neq \emptyset$.  Consider $0 < \epsilon < 1$ and a tiling $\tiling$ with $(K,\delta)$-invariant shapes, which we can do by Theorem \ref{thm:tiling}. Then, from Lemma \ref{cor:invariant}, for every $(S_\tiling,\epsilon)$-invariant set $F \in \finSet$, there exist center sets $C_F(S) \subseteq C(S) \in \centers(\tiling)$ for $S \in \shapes(\tiling)$ such that
$$
F \supseteq \bigsqcup_{S \in \mathcal{S}(\tiling)} SC_F(S) \quad \text{and} \quad \left|F \setminus \bigsqcup_{S \in \mathcal{S}(\tiling)} SC_F(S)\right| \leq \epsilon|F|.
$$

Since $\mu$ is an equilibrium measure, $h(\mu\,|\nu) = 0$. Recall that $S_\tiling = \bigcup_{S \in \shapes(\tiling)} SS^{-1}$. Then, considering Lemma \ref{entropy-0}, pick $K' \supseteq S_\tiling$ and $\delta' < \epsilon$ so that, for every $(K',\delta')$-invariant set $F \in \finSet$, we have
\begin{equation*}
\frac{1}{|F|}H_F(\mu|\nu)\leq \frac{\alpha(1-\varepsilon)}{\max_{S \in \shapes(\tiling)}|S|}.    
\end{equation*}
Fix a $(K',\delta')$-invariant set $F \in \finSet$ and an arbitrary enumeration of the tiles $\{Sc: S \in \shapes(\tiling), c \in C_F(S)\}$, say $T_1,\dots,T_M$, where $M := \sum_{S \in \shapes(\tiling)} |C_F(S)|$. Notice that $(1-\epsilon)|F| \leq \sum_{S \in \shapes(\tiling)} |S||C_F(S)| \leq M\max_{S \in \shapes(\tiling)}|S|$. Moreover, since each $T_i$ is a $(K,\delta)$-invariant set, for every $1 \leq i \leq M$, $\mathrm{Int}_K(T_i) \neq \emptyset$, i.e., there exists $g_i \in G$ such that $Kg_i \subseteq T_i$. Denote $W(i)=\bigsqcup_{j=1}^iT_j$ for $0\leq i\leq M$. Then,
\begin{align*}
    	0\leq \frac{1}{M}\sum_{i=1}^{M}\left(H_{W(i)}(\mu|\nu)-H_{W(i)\setminus Kg_i}(\mu|\nu)\right) &\leq \frac{1}{M}\sum_{i=1}^{M} \left(H_{W(i)}(\mu|\nu)-H_{W(i)\setminus T_i}(\mu|\nu)\right)\\
	&= \frac{1}{M}\, H_{W(M)}(\mu|\nu)\\
	&\leq \frac{|F|}{M} \frac{1}{|F|}H_F(\mu|\nu)\\
	&\leq \frac{M\max_{S \in \shapes(\tiling)}|S|}{M(1-\epsilon)}\frac{\alpha(1-\varepsilon)}{\max_{S \in \shapes(\tiling)}|S|} \\
	&  = \alpha,
\end{align*}
where the first and second inequality follow from Lemma \ref{rmk:H-F-increasing} and, the first equality, from the fact that the sum is telescopic. Consequently, there must exist an index $i' \in \{1, \dots, M\}$ such that
$$H_{W(i')}(\mu|\nu) - H_{W(i')\setminus Kg_{i'}}(\mu|\nu) \leq \alpha.$$
Therefore, taking $E =  W(i')g_{i'}^{-1}$, the result follows from the $G$-invariance of $\mu$ and $\nu$.
\end{proof}

The next Lemma is a version of Step 2 in \cite[Theorem 15.37]{georgii2011gibbs}. 

\begin{lemma}
\label{step2-Georgii-Thm15.37}
Let $\phi\colon X \to \mathbb{R}$ be an exp-summable potential with summable variation with respect to an exhausting sequence $\{E_m\}_m$ and $\mu \in \mathcal{M}(X)$ be an equilibrium measure for $\phi$. Then, for every $\epsilon>0$, there exists $\alpha > 0$ such that, if $E \supseteq K$ and $H_{E}(\mu|\nu) - H_{E\setminus K}(\mu|\nu) \leq \alpha$, then $\nu\left(\left|f_{E} - f_{E \setminus K}\right|\right) \leq \epsilon$.
\end{lemma}

\begin{proof} Notice that, for each $\epsilon>0$, there exists $r_\epsilon > 0$ such that
\begin{equation}\label{ineq-Geor}
|x-1| \leq r_\epsilon\psi(x)+\dfrac{\epsilon}{2},
\end{equation}
where $\psi(x) = 1 - x +x\log x$.

For a given $\epsilon > 0$, consider $\alpha=\frac{\epsilon}{2r_\epsilon}$, and let $E, K \in \finSet$ be such that $K \subseteq E$ and $H_{E}(\mu|\nu) - H_{E\setminus K}(\mu|\nu) \leq \alpha$, which we can do by Proposition \ref{prop:step1-Georgii}. Let $B = \{x \in X : f_{E \setminus K}(x) \neq 0\}$.  Notice that $B \in \sigmaAlg_{E\setminus K}$ 
$$
\int{\mathbbm{1}_{X \setminus B}f_E}d\nu = \int_{X \setminus B} f_E d\nu = \int_{X \setminus B} \nu(f_E \vert \sigmaAlg_{E\setminus K})d\nu = \int_{X \setminus B} f_{E \setminus K}d\nu = 0.
$$
Then, since $f_E(x) \geq 0$, we obtain that $f_E(x) = 0$ $\nu(x)$-almost surely on $X\setminus B$. Next, notice that
\begin{align*}
    \int_{B}f_{E}\log\left(\frac{f_{E}}{f_{E\setminus K}}\right)d\nu &   =    \int_{B}\log\left(\frac{f_{E}}{f_{E\setminus K}}\right)d\mu  \\
    &   = \int_{B} \log f_{E} \,d\mu - \int_{B} \log f_{E\setminus K} \, d\mu\\
    &= \int_{B} f_{E}\log f_{E} d\nu - \int_{B} f_{E \setminus K} \log f_{E\setminus K}  \, d\nu\\
    &= \int f_{E}\log f_{E} d\nu - \int f_{E \setminus K}\log f_{E\setminus K} \, d\nu\\
    &= H_{E}(\mu \, |\, \nu) - H_{E\setminus K}(\mu \, |\, \nu),
\end{align*}
where, making an abuse of notation, we just write $\mu$ and $\nu$, ignoring the restrictions.
Thus,
\begin{equation*}
    H_{E}(\mu \, |\, \nu) - H_{E\setminus K}(\mu \, |\, \nu) =\int_{B} f_{E}\log \left(\frac{f_{E}}{f_{E\setminus K}}\right)d\nu.
\end{equation*}
Furthermore, in $B$, observe that 
\begin{equation*}
    \psi\left(\frac{f_E}{f_{E \setminus K}}\right) = 1 - \frac{f_E}{f_{E \setminus K}} + \frac{f_{E}}{f_{E \setminus K}}\log\left(\frac{f_E}{f_{E \setminus K}}\right),
\end{equation*}
so that 
\begin{align*}
    f_{E \setminus K}\psi\left(\frac{f_{E}}{f_{E \setminus K}}\right)  =  f_{E \setminus K} -  f_{E}  + f_{E}\log\left(\frac{f_{E}}{f_{E \setminus K}}\right).
\end{align*}

Therefore, 
\begin{align*}
    \int_{B} f_{E \setminus K} \psi\left(\frac{f_{E}}{f_{E \setminus K}}\right)d\nu    &= \int_{B}  \left(f_{E \setminus K} -  f_{E}\right) \, d\nu + \int_{B} f_{E} \log\left(\frac{f_{E}}{f_{E \setminus K}}\right) \, d\nu.
\end{align*}

Since  
$f_{E \setminus K} = \nu(f_{E} \,|\, \cB_{E \setminus K})$, we have that $\int_{B}  \left(f_{E \setminus K} -  f_{E}\right) \, d\nu = 0$, so that we can rewrite
\begin{equation*}
H_E(\mu \, |\, \nu) - H_{E\setminus K}(\mu \, |\, \nu)  = \int_{B} f_{E \setminus K} \psi\left(\frac{f_{E}}{f_{E \setminus K}}\right)d\nu.
\end{equation*}

Therefore, from inequality \eqref{ineq-Geor}, it follows that
\begin{align*}
\nu(|f_E-f_{E\setminus K}|) &   =   \int_{B}|f_E-f_{E\setminus K}|d\nu+\int_{X \setminus B}|f_{E}-f_{E\setminus K}|d\nu\\
&   =   \int_{B}\left|f_{E}-f_{E\setminus K}\right|d\nu\\
&   =   \int_{B}\left|\frac{f_E}{f_{E\setminus K}}-1\right|f_{E \setminus K}d\nu\\
&   \leq   r_\epsilon\int_{B}f_{E\setminus K}\psi\left(\frac{f_{E}}{f_{E\setminus K}}\right)d\nu +\frac{\epsilon}{2}\int_{B}f_{E\setminus K}d\nu\\
&   =   r_\epsilon(H_E(\mu \, |\, \nu) - H_{E\setminus K}(\mu \, |\, \nu)) + \frac{\epsilon}{2}\int_{B}d\mu\\
&\leq r_\epsilon\alpha + \frac{\epsilon}{2} = \epsilon.
\end{align*}
\end{proof}

\begin{theorem}\label{thm:equilibrium-implies-DLR}
Let $\phi\colon X \to \R$ be an exp-summable potential with summable variation according to an exhausting sequence $\{E_m\}_m$. If $\mu \in \mathcal{M}_G(X)$ is an equilibrium measure for $\phi$, then $\mu$ is a DLR measure for $\phi$.
\end{theorem}

\begin{proof} Since $\mu$ is an equilibrium measure, then $h(\mu \vert \nu) = 0$. The strategy is to prove that, for every $K \in \finSet$, $\mu\gamma_{K} = \mu$, where $\gamma$ is the Gibbsian specification defined by equation \eqref{specification}. Then, by Lemma \ref{lemma:characterization-of-admitted-by-specification}, it will follow that $\mu$ is a DLR measure for $\phi$. 

Let $h\colon X \to \R$ be a bounded local 
function and $\epsilon > 0$. Since $\gamma$ is a quasilocal specification (see Theorem \ref{thm:quasilocality}), then $\gamma_Kh$ is a bounded quasilocal $\sigmaAlg_{K^c}$-measurable function. Thus, there exists a bounded local $\sigmaAlg_{K^c}$-measurable function $\tilde{h}\colon X \to \R$ such that $\left\Vert \gamma_K h - \tilde{h}\right\Vert_{\infty} < \epsilon$. Since $\tilde{h}$ is a local potential, there exists $B \in \finSet$, $B \supseteq K$, such that $\tilde{h}$ is a $\sigmaAlg_{B\setminus K}$-measurable. Also, since $h$ is local, we can assume, without loss of generality, that $h$ is $\sigmaAlg_{B}$-measurable. 

Consider $\alpha$ as in Lemma \ref{step2-Georgii-Thm15.37}, that is, whenever $E \supseteq B$ and $H_{E}(\mu|\nu) - H_{E\setminus B}(\mu|\nu) \leq \alpha$, then $\nu\left(\left|f_{E} - f_{E \setminus B}\right|\right) \leq \epsilon$. Now, using Proposition \ref{prop:step1-Georgii}, fix a set $E \in \finSet$ such that $E \supseteq B$ and $H_E(\mu|\nu)-H_{E \setminus B}(\mu|\nu) \leq \alpha$. Therefore, by the monotonicity of the relative entropy, we obtain that $H_E(\mu|\nu)-H_{E \setminus K}(\mu|\nu) \leq \alpha$, so that $\nu\left(\left|f_{E} - f_{E \setminus K}\right|\right) \leq \epsilon$. 

We now compute $\left\vert \mu\gamma_K(h) - \mu(h)\right\vert$. First observe that since $\tilde{h}$ is $\sigmaAlg_{B\setminus K}$-measurable and $B \subseteq E$, then $\tilde{h}$ is $\sigmaAlg_{E\setminus K}$-measurable. Therefore, recalling that $\mu\gamma_K(h) = \mu(\gamma_K h)$,
\begin{align*}
    \left\vert \mu\gamma_K(h) - \mu(h)\right\vert &\leq \left\vert \mu(\gamma_Kh) - \mu(\tilde{h})\right\vert + \left\vert \mu(\tilde{h}) - \nu(f_{E\setminus K}\tilde{h}) \right\vert + \left\vert \nu(f_{E\setminus K}\tilde{h}) - \nu(f_{E\setminus K}(\gamma_K h)) \right\vert\\\
    & \quad + \left\vert \nu(f_{E\setminus K}(\gamma_K h)) - \nu(f_{E\setminus K}h) \right\vert  + \left\vert \nu(f_{E\setminus K}h) - \nu(f_E h)\right\vert + \left\vert \nu(f_Eh) -\mu(h)\right\vert\\
    &\leq  \mu\left(\left\vert\gamma_K h - \tilde{h}\right\vert\right) + 0 + \nu\left(f_{E\setminus K}\left\vert\tilde{h} - \gamma_K h\right\vert\right) + 0 + \left\Vert h \right\Vert_{\infty}\nu\left(\left\vert f_{E\setminus K}  -  f_{E} \right\vert\right) + 0.
\end{align*}

We begin by justifying the terms that vanished from the first inequality to the second. Notice that $\left\vert \mu(\tilde{h}) - \nu(f_{E\setminus K}\tilde{h}) \right\vert = 0$ and $\left\vert \nu(f_E h) -\mu(h)\right\vert = 0$, because $\tilde{h}$ is $\sigmaAlg_{E\setminus K}$-measurable and because $h$ is $\sigmaAlg_E$-measurable. We also have that $ \left\vert \nu\left(f_{E\setminus K}(\gamma_K h) \right) -  \nu\left(f_{E\setminus K}h\right) \right\vert= 0$, because $f_{E\setminus K}$ is $\sigmaAlg_{K^c}$-measurable and $\gamma$ is proper, so $
\nu(f_{E\setminus K} (\gamma_K h)) = \nu(\gamma_K (f_{E\setminus K} h))$ and, in addition, since $\nu$ is a DLR measure, we have that
$$
\nu(\gamma_K (f_{E\setminus K} h)) = (\nu\gamma_K)(f_{E\setminus K} h) = \nu(f_{E\setminus K} h).
$$

We now have to deal with the three other terms. Notice that
$$
\mu\left(\left\vert\gamma_K h - \tilde{h}\right\vert\right) < \epsilon \quad \text{ and } \quad \nu\left(f_{E\setminus K}\left\vert\tilde{h} - \gamma_K h\right\vert\right) < \epsilon,
$$
because $\left\Vert \gamma_K h - \tilde{h}\right\Vert_{\infty} < \epsilon$. 
Lastly, since $\nu\left(\left|f_{E \setminus K} - f_{E}\right|\right) \leq \epsilon$, it follows that
\begin{equation*}
    \left\vert \mu\gamma_K(h) - \mu(h)\right\vert < 2\epsilon + \left\Vert h \right\Vert_{\infty}\epsilon.
\end{equation*}
Since $\epsilon > 0$ and $h\colon X \to \R$ are arbitrary, we obtain that, $\mu\gamma_{K} = \mu$,  which concludes the result. 
\end{proof}

\section{Final considerations}
\label{section:examples}

In this section we consider the case when the group is finitely generated, which includes the well-studied case $G = \Z^d$ and show that our approach generalizes previous ones. Next, we present a version of Dobrushin's Uniqueness Theorem adapted to our framework and we apply it to a concrete class of examples of potentials defined in the $G$-full shift for any countable amenable group $G$.

\subsection{The finitely generated case}
\label{subsec:fin-gen}

We now restrict ourselves to the case that $G$ is a finitely generated group. The main goal is to prove that our definition of a Bowen-Gibbs measure (Definition \ref{defn:bowen-gibbs-measure}) for a given exp-summable potential with summable variation according to an exhausting sequence is related to the standard --- but more restrictive --- way to define Bowen-Gibbs measures (e.g., \cite{muir2011gibbs,keller1998equilibrium}). For that, we will prove that the bounds in Definition \ref{defn:bowen-gibbs-measure} can be replaced by a bound which involves the size of the boundary of invariant sets.

Suppose that $G$ is finitely generated and let $S$ be a finite and symmetric generating set. Without loss of generality, suppose that $1_G \in S$. In this context, it is common to implicitly consider an exhausting sequence $E_{m+1} = S^m$. For example, if $G = \Z^d$ and $S$ is the set of all elements $s \in \Z^d$ with $\|s\|_\infty \leq 1$, the sequence $\{E_m\}_m$ recovers the notion of ``boxes'' with sides of length $2m+1$ centered at the origin, which is the most usual in the literature. In particular, one recovers the more standard definition of summable variation for a potential $\phi\colon X \to \R$, which is given by
$$
\sum_{m \geq 1} |E_{m+1}^{-1} \setminus E_{m}^{-1}|\cdot\delta_{E_m}(\phi) = 
\sum_{m \geq 0} |S^{m+1} \setminus S^m|\cdot \delta_{S^m}(\phi) = \sum_{m \geq 0} |\partial B(1_G,m)|\cdot \delta_{B(1_G,m)}(\phi),
$$
where $B(1_G,m) = S^m$ denotes the ball of radius $m$ (according to the word metric), $\partial F := SF \setminus F$ denotes the ``(exterior) boundary'' of a set $F$, and $|\partial B(1_G,m)|$ is proportional to $m^{d-1}$ in the $\Z^d$ case. Usually, potentials that have summable variation according to this particular exhausting sequence are called {\bf regular} (see, for example, \cite{keller1998equilibrium}).

Notice that when $\{E_m\}_m$ is an exhausting sequence of the form $S^m$, we have that
$$
|\partial (S^m F)| = |S(S^mF) \setminus S^m F| = |S^{m+1}F \setminus S^m F| \leq |S^{m+1} \setminus S^m| |\partial_{\rm int} F|,
$$
where $\partial_{\rm int} F = \partial F^c$ denotes the ``interior boundary'' of $F$. Indeed, if $g \in S^{m+1}F \setminus S^m F$, there must exist $h \in \partial_{\rm int} F$ such that $d_S(g,h) = m+1$, where $d_S$ denotes the word metric. In addition, we also have that $|\partial_{\rm int} F| \leq |S||\partial F|$, so
$$
|\partial (S^m F)| = |S^{m+1}F \setminus S^m F| \leq |S^{m+1} \setminus S^m| |S| |\partial F|.
$$

From this, it is direct that
\begin{align*}
V_F(\phi)  =       \sum_{m \geq 0} |S^{m+1}F \setminus S^{m}F| \cdot\delta_{S^m}(\phi)       \leq    \sum_{m \geq 0} |S^{m+1} \setminus S^m| |S| |\partial F|\cdot\delta_{S^m}(\phi)   =       V(\phi)|S| |\partial F|.
\end{align*}

On the other hand, if $x,y \in X$ are such that $x_F = y_F$, we have that
\begin{align*}
|\phi_F(x) - \phi_F(y)| &   \leq    \sum_{g \in F} |\phi(g \cdot x) - \phi(g \cdot y)|    \\
                        &   =       \sum_{m \geq 0} \sum_{g \in \mathrm{Int}_{S^m}(F) \setminus \mathrm{Int}_{S^{m+1}}(F)} |\phi(g \cdot x) - \phi(g \cdot y)|   \\
                        &   \leq   \sum_{m \geq 0} |\mathrm{Int}_{S^m}(F) \setminus \mathrm{Int}_{S^{m+1}}(F)| \cdot\delta_{S^m}(\phi).
\end{align*}

Notice that if $g \in \mathrm{Int}_{S^m}(F) \setminus \mathrm{Int}_{S^{m+1}}(F)$, then $d_S(g,\partial F) = m+1$, i.e., $g \in S^{m+1}\partial F \setminus S^m\partial F$, so
\begin{align*}
|\mathrm{Int}_{S^m}(F) \setminus \mathrm{Int}_{S^{m+1}}(F)| &\leq  |S^{m+1}\partial F \setminus S^m\partial F|\\ 
&\leq |S^{m+1} \setminus S^m| |\partial_{\rm int}(\partial F)|\\ 
&\leq |S^{m+1} \setminus S^m| |S||\partial(\partial F)|\\
&\leq  |S^{m+1} \setminus S^m| |S|^2|\partial F|    
\end{align*}
and
\begin{align*}
|\phi_F(x) - \phi_F(y)| \leq    \sum_{m \geq 0} |S^{m+1} \setminus S^m| |S| |\partial F| \cdot\delta_{S^m}(\phi)  =     V(\phi)  |S|^2 |\partial F|.
\end{align*}

Therefore, we conclude that $\Delta_F(\phi) \leq   V(\phi)|S|^2|\partial F|$.

We now provide an alternative way of proving Proposition \ref{prop:limit-of-VF-over-the-size-of-F} and Lemma \ref{regularityDeltak}. Begin by noticing that a finitely generated group is amenable if and only if $\lim_{F \to G}\frac{|\partial F|}{|F|} = 0$ (indeed, given $\epsilon > 0$, we have that $|\partial F| \leq |SF \triangle F|< \epsilon\cdot|F|$ for every $(S,\epsilon)$-invariant set $F$). Therefore, if $\phi$ has summable variation, it follows that
$$
0 \leq \lim_{F \to G} \frac{V_F(\phi)}{|F|} \leq  V(\phi)|S| \lim_{F \to G} \frac{|\partial F|}{|F|} = 0
$$
and, similarly,
$$
0 \leq \lim_{F \to G} \frac{\Delta_F(\phi)}{|F|} \leq V(\phi)|S|^2\lim_{F \to G} \frac{|\partial F|}{|F|} = 0.
$$

In particular, in this context, we could alternatively have defined a Bowen-Gibbs measure as follows: if $G$ is a finitely generated amenable group with generating set $S$ and $\phi\colon X \to \R$ is an exp-summable potential with summable variation according to $\{S^m\}_m$, a measure $\mu \in \mathcal{M}(X)$ is a Bowen-Gibbs measure for $\phi$ if for every $\epsilon > 0$, there exist $K \in \finSet$ and $\delta > 0$ such that for every $(K,\delta)$-invariant set $F \in \finSet$ and $x \in X$,
\begin{equation*}
       \exp\left(- C |\partial F|\right) \leq  \frac{\mu([x_F])}{\exp\left(\phi_F(x) - p(\phi)\cdot |F|\right)} \leq \exp\left(C |\partial F|\right),
\end{equation*}
where $C > 0$ is a constant that we can choose to be $$
C := 5 V(\phi)|S|^2 \geq 2 V(\phi)|S| + 3 \Delta(\phi)|S|^2.
$$ 

This recovers the more standard definition of Bowen-Gibbs measure in terms of boundaries. Furthermore, with this choice of $C$, it is not difficult to check that we could mimic the proofs of Proposition \ref{BowenGibbs}, Theorem \ref{thm:bowgib}, and Theorem \ref{prop:bowen-gibbs-implies-equlibrium}, thus providing all the implications involving Bowen-Gibbs measures. 

\subsection{Dobrushin's Uniqueness Theorem}\label{sec:example-dobrushins-uniqueness-theorem}

From \S \ref{sec:existence-of-DLR-states}, we know that if $\phi\colon X \to \R$ is an exp-summable potential with summable variation according to an exhausting sequence $\{E_m\}_m$, then the set of $G$-invariant DLR measures for $\phi$ is non-empty. One natural question that may arise is under which conditions we have uniqueness of the DLR measure. When a specification is a Gibbsian specification, the Dobrushin's Uniqueness Theorem (see \cite{georgii2011gibbs}) addresses this question. For a detailed proof of a version of this theorem adapted to our setting, see \cite{borsato2022thesis}.

Let $2^\N$ be the set of all subsets of $\N$, which is a $\sigma$-algebra, and $\mathcal{M}(\mathbb{N},2^\N)$  be the set of probability
measures on $(\mathbb{N},2^\N)$. For $A \in 2^\N$,  $w \in X$, and $g \in G$, denote 
\begin{equation*}
    \gamma^0_{\{g\}}(A, w)(\eta) =  \gamma_{\{g\}}\left(A \times \N^{G\setminus \{g\}}, x\right),
\end{equation*}
where $\gamma$ is a specification, notice that, for each $x \in X$, $\gamma_{g}^0(\cdot,x)\in\mathcal{M}(\mathbb{N},2^\N)$. Now, for each $h \in G$,  the $w_h$-dependence of $\gamma^0_{\{g\}}(\cdot, w)$ is estimated by the quantity
\begin{equation*}
\rho_{gh}(\gamma) =  \sup_{\substack{w, \eta \in X \\ w_{G \setminus \{h\}} = \eta_{G \setminus \{h\}}}} \left\Vert \gamma_{\{g\}}^0(\cdot, \eta) - \gamma_{\{g\}}^0(\cdot, w)\right\Vert,
\end{equation*}
where, for any given $\mu, \tilde{\mu} \in \mathcal{M}(\mathbb{N},2^\N)$, $\Vert \mu - \tilde{\mu}\Vert = \max_{A \in \mathcal{E}} \vert \mu(A) - \tilde{\mu}(A)\vert$ (see \cite[\S 8.1]{georgii2011gibbs}).

The infinite matrix $\rho(\gamma) = (\rho_{gh}(\gamma))_{g,h \in G}$ is called Dobrushin's interdependence matrix for $\gamma$. When there is no ambiguity, we will omit the parameter $\gamma$ from the notation.  

\begin{remark}\label{rmk:rho-gg-equal-0} Notice that $\rho_{gg} = 0$, for all $g \in G$.
\end{remark}

\begin{definition} Let $\gamma$ be a specification. We say that $\gamma$ satisfies the \textit{Dobrushin's condition} if $\gamma$ is quasilocal and
    \begin{equation*}
        c(\gamma):= \sup_{g \in G}\sum_{h\in G}\rho_{gh} < 1.
    \end{equation*}
\end{definition}

\begin{theorem}[Dobrushin's Uniqueness Theorem]\label{thm:Dobrushin-uniqueness} If $\gamma$ is a specification that satisfies the Dobrushin's condition, then there is at most one measure that is admitted by the specification $\gamma$.
\end{theorem}

We now present an example of a potential inspired by the Potts model \cite{friedli2017statistical,duminil2017lectures} such that, under some conditions to be presented, is exp-summable and has summable variation according to an exhausting sequence $\{E_m\}_m$. Moreover, this potential will also satisfy that, if $\mu$ is a Bowen-Gibbs measure, $\int \phi d\mu > -\infty$. Another important property of this potential is that it is non-trivial, in the sense that it depends on every coordinate of $G$. We will also explore conditions on $\beta > 0$ such that the potential $\beta \phi$ satisfies Dobrushin's condition.

\subsection{Main example}

Given a countable amenable group $G$, consider the potential $\phi\colon X \to \R$ given by
\begin{equation}\label{eq:potential-for-dobrushin-example}
\phi(x) := - \sum_{g \in G} c(g,x(1_G)) \mathbbm{1}_{\{x(1_G) = x(g)\}},  
\end{equation}
with $c\colon G \times \N \to [0,\infty)$ such that, given an exhausting sequence $\{E_m\}_m$ of $G$, it holds that
\begin{enumerate}
    \item\label{defn-of-c-item1} $\sum_{m \geq 1} |E_{m+1} \setminus E_m| \sum_{g \in G \setminus E_m} C(g)  < \infty$, with $C(g) := \sup_{n} c(g,n)$ for $g \neq 1_G$; and
    \item\label{defn-of-c-item2} for all $M>0$, there exists $n_0 \in \N$ such that for all $n \geq n_0$, $M \log(n) \leq c(1_G,n)$.
\end{enumerate}

\begin{lemma}\label{claim:potential}
If the potential $\phi\colon X\to\R$ given by $\phi(x) = - \sum_{g \in G} c(g,x(1_G)) \mathbbm{1}_{\left\{x(1_G) = x(g)\right\}}$ satisfies conditions (\ref{defn-of-c-item1}) and (\ref{defn-of-c-item2}), then, for every $\beta > 0$, the potential $\beta\phi$ is well-defined, has summable variation according to the exhausting sequence $\{E_m\}_m$, is exp-summable, and $\int\phi d\mu_\beta > -\infty$ for any Bowen-Gibbs measure $\mu_\beta \in \mathcal{M}(X)$ for $\beta\phi$. 
\end{lemma}

\begin{proof}
Notice that condition (\ref{defn-of-c-item1}) implies that
$$
0 \leq  \sum_{\substack{g \in G\\g \neq 1_G}} C(g) = \sum_{m \geq 1} \sum_{g \in E_{m+1} \setminus E_m} C(g) \leq \sum_{m \geq 1} |E_{m+1} \setminus E_m| \sum_{g \in G \setminus E_m} C(g)  < \infty.
$$

Now, for any $x \in X$,
\begin{align*}
|\phi(x)| &=  
c(1_G,x(1_G)) + \sum_{\substack{g\in G\\g \neq 1_G}} c(g,x(1_G)) \leq c(1_G,x(1_G)) + \sum_{\substack{g\in G\\g \neq 1_G}} C(g) < \infty,
\end{align*}
so $\phi(x)$ is well-defined and therefore $\beta\phi$ is well-defined, too.

Next, notice that, for every $m\in \N$ and $x,y \in X$ such that $x_{E_m} = y_{E_m}$, we have that
\begin{align*}
|\phi(x)-\phi(y)| &= \left|\sum_{g \in G} c(g,x(1_G))\left(\mathbbm{1}_{\left\{y(1_G) = y(g)\right\}}- \mathbbm{1}_{\left\{x(1_G) = x(g)\right\}}\right)\right|\\
&\leq \sum_{g \in G\setminus E_m} c(g,x(1_G))\left|\mathbbm{1}_{\left\{y(1_G) = y(g)\right\}}- \mathbbm{1}_{\left\{x(1_G) = x(g)\right\}}\right|\\
&\leq \sum_{g \in G \setminus E_m} C(g).
\end{align*}
Therefore, for any $m\in \N$, $\delta_{E_m}(\phi) \leq \sum_{g \in G \setminus E_m} C(g)$, so that
\begin{align*}
    V(\phi) = \sum_{m=1}^{\infty} |E_{m+1}^{-1} \setminus E_{m}^{-1}|\cdot\delta_{E_m} (\phi) \leq \sum_{m=1}^{\infty} |E_{m+1}^{-1} \setminus E_{m}^{-1}|\sum_{g \in G \setminus E_m} C(g).
\end{align*}
Thus, due to condition (\ref{defn-of-c-item1}), we have that $\phi$ has summable variation according to $\{E_m\}_m$. In addition, observe that $V(\beta\phi) = \beta V(\phi)$, so $\beta\phi$ has also summable variation.

Pick $0 < \alpha = \frac{\beta}{2} < \beta$. This determines $n_1$ such that
\begin{align*}
\sum_{n \geq 1} \beta c(1_G,n) \exp(-\beta c(1_G,n))    &   = \sum_{n < n_1} \beta c(1_G,n) \exp(-\beta c(1_G,n)) \\
            &   \quad   + \sum_{n \geq n_1} \beta c(1_G,n) \exp(-\beta c(1_G,n))  \\
            &   = C_0 + \beta\sum_{n \geq n_1}  \exp(\alpha c(1_G,n)) \exp(-\beta c(1_G,n)) \\
            &   = C_0 + \beta\sum_{n \geq n_1}  \exp\left(-\frac{\beta}{2}c(1_G,n)\right),
\end{align*}
where
$$
C_0 := \sum_{n < n_1} \beta c(1_G,n) \exp(-\beta c(1_G,n)) < \infty.
$$

It remains to bound $\sum_{n \geq n_1}  \exp\left(-\frac{\beta}{2}c(1_G,n)\right)$. Now, for any $\epsilon > 0$ and $M = \frac{2}{\beta}(1 + \epsilon)>0$, condition (\ref{defn-of-c-item2}) implies that there exists $n_0 \in \N$ (maybe $n_0 > n_1$) such that $\forall n \geq n_0$, $M\log(n) \leq c(1_G,n)$, so $-\frac{2}{\beta}(1 + \epsilon) \log(n) \geq -c(1_G,n)$. Therefore,
\begin{align*}
\sum_{n \geq n_1}  \exp\left(-\frac{\beta}{2}c(1_G,n)\right)   &   =   \sum_{n_1 \leq n < n_0}  \exp\left(-\frac{\beta}{2}c(1_G,n)\right) + \sum_{n \geq n_0}  \exp\left(-\frac{\beta}{2}c(1_G,n)\right)   \\
            &   \leq  C_1 + \sum_{n \geq n_0}  \exp\left((1 + \epsilon) \log(n)\right)\\
        &   =   C_1 + \sum_{n \geq n_0}  \frac{1}{n^{1+\epsilon}} < \infty,
\end{align*}
where $C_1 := \sum_{n_1 \leq n < n_0}  \exp(-\frac{\beta}{2}c(1_G,n)) < \infty$. So, from Proposition \ref{prop:equivalentes-of-finiteness-of-integral} we have that $\int \phi d\mu_{\beta}>-\infty$ for every $\beta>0$.

 Later, choosing a $n_3\in\mathbb{N}$, great enough
\begin{align*}
    Z_{1_G}(\beta\phi)&= \sum_{n \geq 1}\exp(-\beta c(1_G,n)) \\
                      &= \sum_{n < n_3}\exp(-\beta c(1_G,n))+\sum_{n \geq n_3}c(1_G,n)\exp(-\beta c(1_G,n))<\infty.
\end{align*}
Therefore, the potential $\beta\phi$ is exp-summable, for all $\beta>0$.
\end{proof}

\begin{remark}
The set of functions $c\colon G \times \N \to [0,\infty)$ satisfying conditions (\ref{defn-of-c-item1}) and (\ref{defn-of-c-item2}) is non-vacuous. For example, given an exhausting sequence $\{E_m\}_m$, consider $c\colon G \times \N \to [0,\infty)$ and some constant $L \geq 0$ such that
\begin{itemize}
    \item[(a)] for every $m \geq 1$, $0 \leq c(g,n) \leq \frac{L2^{-m-1}}{|E_{m+1}|^2}$ for every $g \in E_{m+1} \setminus E_m$; and
    \item[(b)] any $c(1_G,n)$ of polynomial order will satisfy condition (\ref{defn-of-c-item2}).
\end{itemize}

We now prove that any such $c$ satisfies the previous conditions (\ref{defn-of-c-item1}) and (\ref{defn-of-c-item2}). Due to (a) and the fact that $\{E_m\}_m$ is nested, for every $m \geq 1$,
\begin{equation*}
    \sum_{g \in G \setminus E_m} C(g) \leq \sum_{\ell = m}^\infty \sum_{g \in E_{\ell + 1} \setminus E_\ell} \frac{L2^{-\ell-1}}{|E_{\ell + 1}|^2} \leq \sum_{\ell = m}^\infty \frac{L2^{-\ell-1}}{|E_{\ell+1}|} \leq \frac{L2^{-m}}{|E_{m+1}|} < \infty.  
\end{equation*}

Then, 
\begin{equation*}
    \sum_{\substack{g \in G\\g \neq 1_G}} C(g) = \sum_{g \in G \setminus E_1} C(g) \leq \frac{L}{2|E_2|} <\infty
\end{equation*} and
$$
\sum_{m =1}^{\infty} |E_{m+1} \setminus E_m| \sum_{g \in G \setminus E_m} C(g) \leq \sum_{m =1}^{\infty} |E_{m+1} \setminus E_m| \frac{L2^{-m}}{|E_{m+1}|} \leq \sum_{m =1}^{\infty} L2^{-m} = L < \infty,
$$
so condition (\ref{defn-of-c-item1}) is satisfied. Now, due to (b), we have that condition (\ref{defn-of-c-item2}) is satisfied.

\end{remark}

Our next goal is to study under which conditions we have uniqueness of Gibbs measures for $\beta \phi$, where $\beta$ can be interpreted as the inverse of the temperature of the system. For that, we use the Dobrushin's Uniqueness Theorem (Theorem \ref{thm:Dobrushin-uniqueness}). In order to obtain explicit conditions on $\beta$, we divide the rational into claims.

\begin{claim}\label{claim:dif-f} If $x,y \in X$ are such that $x_{G \setminus \{g\}} = y_{G \setminus \{g\}}$, for some $g \in G$, then
 \begin{equation*}
     \sum_{h \in G} \left(\phi(h \cdot x) - \phi(h \cdot y) \right)\end{equation*}
converges absolutely. Moreover,
 \begin{align*}
     \sum_{h \in G}\left(\phi(h \cdot x) - \phi(h \cdot y) \right) =& -c(1_G,x(g)) + c(1_G,y(g))\\
     &\quad + \sum_{\substack{h\in G\\h \neq 1_G}}\left(c(h,x(hg)) + c(h^{-1},x(hg))\right)\left(-\mathbbm{1}_{\{x_{hg} = x_{g}\}} + \mathbbm{1}_{\{y_{hg} = y_g\}}\right).
    \end{align*}
\end{claim}

\begin{proof}[Proof of Claim \ref{claim:dif-f}.] Since we are summing over all $h$-translations of $x$ and $y$, for $h \in G$, we can assume, without loss of generality, that $g = 1_G$, that is, $x_{G \setminus \{1_G\}} = y_{G \setminus \{1_G\}}$. Then,
\begin{align*}
\sum_{h \in G} \left|\phi(h \cdot x) - \phi(h \cdot y) \right| &= \sum_{h \in G}\left|-\sum_{g \in G} c(g,x(h)) \mathbbm{1}_{\{x_h = x_{gh}\}} + \sum_{g \in G} c(g,y(h)) \mathbbm{1}_{\{y_h = y_{gh}\}}\right| \\
    &= \left|\sum_{g\in G}-c(g,x(1_G)) \mathbbm{1}_{\{x_{1_G} = x_{g}\}} + c(g,y(1_G)) \mathbbm{1}_{\{y_{1_G} = y_{g}\}}\right|\\
    &\quad + \sum_{\substack{h \in G\\h\neq 1_G}}\left| \sum_{g \in G} \left(-c(g,x(h)) \mathbbm{1}_{\{x(h) = x(gh)\}} + c(g,y(h)) \mathbbm{1}_{\{y_h = y_{gh}\}}\right)\right|\\
    &= \left|\sum_{g\in G}-c(g,x(1_G)) \mathbbm{1}_{\{x_{1_G} = x_{g}\}} + c(g,y(1_G)) \mathbbm{1}_{\{y_{1_G} = y_{g}\}}\right|\\
    &\quad + \sum_{\substack{h \in G\\h\neq 1_G}} \left|\sum_{g \in G} -c(g,x(h)) \left(\mathbbm{1}_{\{x_{h} = x_{gh}\}} - \mathbbm{1}_{\{y_{h} = y_{gh}\}}\right)\right|\\
    &\leq \left|\sum_{g\in G}-c(g,x(1_G)) \mathbbm{1}_{\{x_{1_G} = x_{g}\}} + c(g,y(1_G)) \mathbbm{1}_{\{y_{1_G} = y_{g}\}}\right|\\
    &\quad + \sum_{\substack{h \in G\\h\neq 1_G}}\left| \sum_{\substack{g \in G\\g\neq h^{-1}}} -c(g,x(h)) \left(\mathbbm{1}_{\{x_{h} = x_{gh}\}} - \mathbbm{1}_{\{y_{h} = y_{gh}\}}\right)\right|\\
    &\quad + \sum_{\substack{h \in G\\h\neq 1_G}} \left|-c(h^{-1},x(h)) \left(\mathbbm{1}_{\{x_{h} = x_{1_G}\}} - \mathbbm{1}_{\{y_{h} = y_{1_G}\}}\right)\right|.
\end{align*}

Note that if $h \neq 1_G$ and $g \neq h^{-1}$, we have that $x(h) = y(h)$ and $x(gh) = y(gh)$, so that $x(gh) = y(gh) \iff y(h) = y(gh)$. Then, $\mathbbm{1}_{\{x(h) = x(gh)\}} - \mathbbm{1}_{\{y(h) = y(gh)\}} = 0$ and
\begin{align*}
     \sum_{h \in G} \left|\phi(h \cdot x) - \phi(h \cdot y) \right| &\leq \left|\sum_{g\in G}\left(-c(g,x(1_G)) \mathbbm{1}_{\{x(1_G) = x(g)\}} + c(g,y(1_G)) \mathbbm{1}_{\{y(1_G) = y(g)\}}\right)\right|\\
    &\quad + \sum_{\substack{h \in G\\h\neq 1_G}} \left|-c(h^{-1},x(h)) \left(\mathbbm{1}_{\{x(h) = x(1_G)\}} - \mathbbm{1}_{\{y(h) = y_{1_G}\}}\right)\right|\\
    &\leq c(1_G,x(1_G)) + c(1_G,y(1_G))\\
    &\quad + \sum_{\substack{g\in G\\g \neq 1_G}} \left|-c(g,x(1_G)) \mathbbm{1}_{\{x(1_G) = x(g)\}} + c(g,y(1_G)) \mathbbm{1}_{\{y(1_G) = y(g)\}}\right|\\
    &\quad + \sum_{\substack{h \in G\\h\neq 1_G}}\left| -c(h^{-1},x(h)) \left(\mathbbm{1}_{\{x(h) = x(1_G)\}} - \mathbbm{1}_{\{y(h) = y(1_G)\}}\right)\right|\\
    &\leq c(1_G,x(1_G)) + c(1_G,y(1_G)) \\
    &   \quad + \sum_{\substack{g\in G\\g \neq 1_G}} \left(c(g,x(1_G)) + c(g,y(1_G)) + c(g^{-1},x(g))\right).
\end{align*}

Therefore,
\begin{equation}\label{eq:bound-for-f-G-in-dobrushin-example}
   \sum_{h \in G} \left|\phi(h \cdot x) - \phi(h \cdot y) \right| \leq 3\sum_{\substack{h\in G\\h \neq 1_G}} C(h) + c(1_G,x(1_G)) + c(1_G,y(1_G)), 
\end{equation}
  from where it follows that, due to condition (\ref{defn-of-c-item1}),
 \begin{equation*}
     \sum_{h \in G} \left|\phi(h \cdot x) - \phi(h \cdot y) \right|  < \infty.
\end{equation*}
Moreover, notice that, applying the same rational with no absolute values, we get that 
\begin{align*}
   &\sum_{h \in G}\left(\phi(h \cdot x) - \phi(h \cdot y) \right)\label{ecu-claim2}\\
 &=  -c(1_G,x(1_G)) + c(1_G,y(1_G))\\
    &\quad + \sum_{\substack{h\in G\\h \neq 1_G}} \left[\mathbbm{1}_{\{x_{h} = x_{1_G} \}}\left(-c(h,x(1_G)) -c(h^{-1},x(h))\right) + \mathbbm{1}_{\{y_{h} = y_{1_G}\}} \left(c(h,y(1_G))  + c(h^{-1},x(h))\right)\right]\\
    &= -c(1_G,x(1_G)) + c(1_G,y(1_G)) \\
    &   \quad + \sum_{\substack{h\in G\\h \neq 1_G}}\left(c(h,x(h)) + c(h^{-1},x(h))\right)\left(-\mathbbm{1}_{\{x_h = x_{1_G}\}} + \mathbbm{1}_{\{y_h = y_{1_G}\}}\right).
\end{align*}
\end{proof}

Now, for a fixed $b \in \N$, define, for each $g \in G$ and $z \in X$ the potential $\varphi^g_z\colon \N \to \R$ given by
\begin{equation}\label{eq:defn-of-varphi-involving-f-star}
    \varphi^g_z(a) = \phi_{*}^{\tau_{a,b}}(bz_{G \setminus \{g\}}).
\end{equation}

Notice that, from Claim \ref{claim:dif-f},
\begin{align}\label{eq:calculation-of-varphi-g-z}
\varphi^g_z(a) &= \phi_{*}^{\tau_{a,b}}(bz_{G \setminus \{g\}})\nonumber\\
&= \sum_{h \in G} \left[\phi(h \cdot (a z_{G \setminus \{g\}})) - \phi(h \cdot(b z_{G \setminus \{g\}})) \right]\nonumber\\
&= c(1_G,b) -c(1_G,a) + \sum_{\substack{h \in G\\ h \neq 1_G}} \left[(c(h,z_{hg}) + c(h^{-1},z_{hg}))(\mathbbm{1}_{\{z_{hg} = b\}} -\mathbbm{1}_{\{z_{hg} = a\}})\right].
\end{align}
Now, pick $h_0 \neq g$ and $z,z' \in X$ such that $z_{G \setminus \{h_0\}} = z'_{G \setminus \{h_0\}}$ and define the function $ \varphi^g_{z,z'}\colon \N \times [0,1] \to \R$ given by
\begin{equation*}
    \varphi^g_{z,z'}(a,t) := t\varphi^g_{z'}(a) + (1-t)\varphi^g_{z}(a) = \varphi^g_{z}(a) +   t\Delta^g_{z,z'}(a),
\end{equation*}
with $\Delta^g_{z,z'}(a) = \varphi^g_{z'}(a) - \varphi^g_{z}(a)$. Notice that $\varphi^g_{z,z'}(a,0) = \varphi^g_{z}(a)$ and $\varphi^g_{z,z'}(a,1) = \varphi^g_{z'}(a)$.

\begin{claim}\label{claim-Delta} Let $g \in G$. Then, for every $h_0 \neq g$ and $z,z' \in X$ such that $z_{G \setminus \{h_0\}} = z'_{G \setminus \{h_0\}}$, it holds that
$$
\|\Delta^g_{z,z'}\|_\infty \leq 2(C(h_0g^{-1}) + C(gh_0^{-1})).
$$
\end{claim}

\begin{proof}[Proof of Claim \ref{claim-Delta}] Let $g, h_0\in G$, with $h_0 \neq g$, and $z, z' \in X$ be such that $z_{G \setminus \{h_0\}} = z'_{G \setminus \{h_0\}}$. Then, due to the equation \eqref{eq:calculation-of-varphi-g-z}, we have that, for all $a \in \N$,

\begin{align*}
\Delta^g_{z,z'}(a) 
&= \sum_{\substack{h \in G\\ h \neq 1_G}} \left[(c(h,z'(hg)) + c(h^{-1},z'(hg)))(-\mathbbm{1}_{\{z'_{hg} = a\}} + \mathbbm{1}_{\{z'_{hg} = b\}})\right]\\
&\quad - \sum_{\substack{h \in G\\ h \neq 1_G}} \left[(c(h,z(hg)) + c(h^{-1},z(hg)))(-\mathbbm{1}_{\{z_{hg} = a\}} + \mathbbm{1}_{\{z_{hg} = b\}})\right].
\end{align*}

Notice that, if $hg \neq h_0$, then $z_{hg} = z'_{hg}$, so that
\begin{align*}
   & (c(h,z'(hg)) + c(h^{-1},z'(hg)))(-\mathbbm{1}_{\{z'_{hg} = a\}} + \mathbbm{1}_{\{z'_{hg} = b\}}) \\ & \quad =  (c(h,z(hg)) + c(h^{-1},z(hg)))(-\mathbbm{1}_{\{z_{hg} = a\}} + \mathbbm{1}_{\{z_{hg} = b\}}).
\end{align*}
This means that the only terms that remain in the sums above are the ones such that $hg = h_0$, that is, $h = h_0g^{-1}$. Thus,
\begin{align*}
    \Delta^g_{z,z'}(a)&= \left[(c(h_0g^{-1}, z'_{h_0}) + c((h_0g^{-1})^{-1},z'_{h_0}))(-\mathbbm{1}_{\{z'_{h_0} = a\}} + \mathbbm{1}_{\{z'_{h_0} = b\}})\right]\\
    &\quad - \left[(c(h_0g^{-1},z_{h_0}) + c((h_0g^{-1})^{-1},z(h_0)))(-\mathbbm{1}_{\{z(h_0) = a\}} + \mathbbm{1}_{\{z(h_0) = b\}})\right].
\end{align*}

Therefore,
\begin{align*}
    \left|\Delta^g_{z,z'}(a)\right| &\leq c(h_0g^{-1}, z'(h_0)) + c(gh_0^{-1},z'(h_0)) + c(h_0g^{-1},z(h_0)) + c(gh_0^{-1},z(h_0))\\
    &\leq 2(C(h_0g^{-1}) + C(gh_0^{-1})),
\end{align*}
where the last inequality follows from the definition of $C(g)$. 
\end{proof}

\begin{claim}\label{claim:specification-in-terms-of-varphi} Let $g,h_0 \in G$ and $z,z' \in X$ be such that $z_{G \setminus \{h_0\}} = z'_{G \setminus \{h_0\}}$. Then, for every $A \in \mathcal{E}$,
\begin{equation}\label{eq:specification-for-A-times-N-z}
    \gamma^0_g(A,z) = \gamma_g(A\times \N^{G\setminus \{g\}},z) = \frac{\sum_{a \in A} \exp\left(\varphi^g_{z,z'}(a,0)\right)}{\sum_{n \in \N} \exp\left(\varphi^g_{z,z'}(n,0)\right)}
\end{equation}
and
\begin{equation}\label{eq:specification-for-A-times-N-z-prime}
     \gamma^0_g(A,z') =  \gamma_g(A\times \N^{G\setminus \{g\}},z') = \frac{\sum_{a \in A} \exp\left(\varphi^g_{z,z'}(a,1)\right)}{\sum_{n \in \N} \exp\left(\varphi^g_{z,z'}(n,1)\right)},
\end{equation}
where $\gamma$ is the Gibbsian specification given by equation \eqref{specification}.
\end{claim}

\begin{proof}[Proof of Claim \ref{claim:specification-in-terms-of-varphi}]
Begin by noticing that, from equation \eqref{sumspecification}, for any $z \in X$, we can write
\begin{align*}
     \gamma_g(A\times \N^{G\setminus \{g\}},z)  
     = \sum_{a \in \N} \gamma_{g}\left([a^g], z\right)\mathbbm{1}_{\{a  \in A\}} = \sum_{a \in A} \gamma_{g}\left([a^g], z\right).
\end{align*}
Moreover, from Proposition \ref{lemma:characterization-DLR-state}, we have that
\begin{align*}
    \gamma_g([a^g],z) = \frac{\exp\left(\phi_*^{\tau_{a, b}}(bz_{\{g\}^c})\right)}{\sum_{n \in \N}\exp\left(\phi_*^{\tau_{n, b}}(bz_{\{g\}^c})\right)} = \frac{\exp\left(\varphi^g_z(a)\right)}{\sum_{n \in \N}\exp\left(\varphi^g_z(n)\right)} = \frac{\exp\left(\varphi^g_{z,z'}(a,0)\right)}{\sum_{n \in \N}\exp\left(\varphi^g_{z,z'}(n,0)\right)}.
\end{align*}

This concludes the proof of equation \eqref{eq:specification-for-A-times-N-z}. The proof of equation \eqref{eq:specification-for-A-times-N-z-prime} follows from an analogous argument, by just replacing $z$ by $z'$.
\end{proof}

Now, let $m$ be the counting measure on $\N$ and, for each $t \in [0,1]$, $g \in G$, and $a \in \N$, consider the measure
\begin{equation*}
    \nu_t = \chi_{g}(\cdot,t)dm, \text{ with } \chi_{g}(a,t)= \frac{\exp\left(\varphi^g_{z,z'}(a,t)\right)}{\sum_{n \in \N} \exp\left(\varphi^g_{z,z'}(n,t)\right)}.
\end{equation*}

For each $A \subseteq \N$, $g,h_0 \in G$, and $z,z' \in X$ such that $z_{G \setminus \{h_0\}} = z'_{G \setminus \{h_0\}}$, from Claim \ref{claim:specification-in-terms-of-varphi}, we obtain that
\begin{align*}
\nu_0(A)   = \frac{\sum_{a \in A}\exp\left(\varphi^g_{z,z'}(a,0)\right)}{\sum_{n \in \N} \exp\left(\varphi^g_{z,z'}(n,0)\right)}
                 = \gamma^0_g(A,z)
\end{align*}
and
\begin{align*}
 \nu_1(A) = \frac{\sum_{a \in A}\exp\left(\varphi^g_{z,z'}(a,1)\right)}{\sum_{n \in \N} \exp\left(\varphi^g_{z,z'}(n,1)\right)}  = \gamma^0_g(A,z').
 \end{align*}

In order to study conditions under which Theorem  \ref{thm:Dobrushin-uniqueness} holds, we need some estimates, which we calculate now. First, notice that $    \left\Vert \nu_{1} - \nu_{0} \right\Vert_{TV}= \frac{1}{2} \int \left|\chi_g(a,1) - \chi_g(a,0)\right|dm.$

\begin{claim}\label{claim-diferential-t} For each $a \in \N$ and $g \in G$, the map $t \mapsto \chi_{g}(a,t)$ is differentiable and
$$\frac{\partial}{\partial t}\chi_{g}(a,t) =\chi_{g}(a,t)\left(\Delta^g_{z,z'}(a) - \int \Delta^g_{z,z'}(b) \, d\nu_t(b)\right).$$
\end{claim}

\begin{proof} Notice that, for each $t\in [0,1]$, the function $\exp\left( \varphi^g_{z,z'}(\cdot,t)\right)$ is integrable with respect to $m$. Furthermore, for each $a \in \N$, $\frac{\partial}{\partial t}\exp\left( \varphi^g_{z,z'}(a,t)\right)$ exists and, more precisely, 
$$\frac{\partial}{\partial t}\exp\left( \varphi^g_{z,z'}(a,t)\right)=\Delta^g_{z,z'}(a)\exp\left( \varphi^g_{z,z'}(a,t)\right).$$

Moreover, notice that, from Claim \ref{claim-Delta}, the expression $\Delta^g_{z,z'}(a)\exp\left( \varphi^g_{z,z'}(a,t)\right)$ is integrable with respect to $m$. Therefore, we conclude that 
\begin{equation}\label{deriv-t}
    \frac{\partial}{\partial t}\int \exp\left( \varphi^g_{z,z'}(a,t)\right) \, dm(a) = \int \Delta^g_{z,z'}(a)\exp\left( \varphi^g_{z,z'}(a,t)\right) \, dm(a).
\end{equation}

Now, fix $g$ and $h_0 \in G$ and let $z,z' \in X$ be such that $z_{G \setminus \{h_0\}} = z'_{G \setminus \{h_0\}}$. Then, the map $t \mapsto \sum_{n\in\N}\exp\left(\varphi^g_{z,z'}(n,t)\right)$ is differentiable with respect to $t$. This implies that the map $t \mapsto \chi_{g}(x,t)$ is also differentiable.

From equation \eqref{deriv-t} and the fact that $t \mapsto \chi_{g}(x,t)$ is differentiable, we have that
\begin{align*}
     \frac{\partial}{\partial t}\chi_{g}(a,t) &=  \frac{\partial}{\partial t}\left( \frac{\exp( \varphi^g_{z,z'}(a,t))}{\int \exp\left( \varphi^g_{z,z'}(b,t)\right)\, dm(b)}\right)\\
    &= \Delta^g_{z,z'}(a)\chi_{g}(a,t)  - \frac{\left(\int \Delta^g_{z,z'}(b)\chi_g(b,t) \, dm(b)\right)\exp\left(\varphi^g_{z,z'}(a,t)\right)}{\int \exp\left(\varphi^g_{z,z'}(b,t)\right) \, dm(b)}\\
     &= \chi_{g}(a,t)\left(\Delta^g_{z,z'}(a) - \int \Delta^g_{z,z'}(b) \, d\nu_t(b)\right).
\end{align*}
\end{proof}

Considering  Claim \ref{claim-diferential-t}, we have that
\begin{align*}
    \int \left|\chi_g(a,1) - \chi_g(a,0)\right|dm(a) &= \int \left|\int_0^1 \left(\frac{\partial}{\partial t}\chi_{g}(a,t)\right)\,dt\right| \, dm(a)\\
   &\leq \int_{0}^1 \int\left|\frac{\partial}{\partial t}\chi_{g}(a,t)\right|dm(a)\, dt\\
   &= \int_{0}^1 \int\left| \chi_{g}(a,t)\left(\Delta^g_{z,z'}(a) - \int \Delta^g_{z,z'}(b) \, d\nu_{t}(b)\right)\right|dm(a)\,dt\\
&= \int_{0}^1 \int\left|\Delta^g_{z,z'}(a) - \int \Delta^g_{z,z'}(b) \, d\nu_{t}(b)\right|d\nu_{t}(a)\,dt\\
&\leq \int_{0}^1 \int\left(\|\Delta^g_{z,z'}\|_\infty + \int \|\Delta^g_{z,z'}\|_\infty \, d\nu_{t}(b)\right)d\nu_{t}(a)\,dt\\
&= 2\|\Delta^g_{z,z'}\|_\infty.
\end{align*}
Thus, by Claim \ref{claim-Delta} we have
\begin{align*}
\rho_{gh_0} \leq\frac{1}{2} \sup_{\substack{z, z' \in X\\z_{G\setminus\{h_0\}} = z'_{G\setminus\{h_0\}}}} 2\|\Delta^g_{z,z'}\|_\infty 
\leq 2(C(h_0g^{-1}) + C(gh_0^{-1})).
\end{align*}

Therefore, considering that $\rho_{gg} = 0$,
\begin{align*}
\sum_{h \in G} \rho_{gh} &\leq 2\sum_{\substack{h\in G\\h \neq g}} \left[C(hg^{-1}) + C(gh^{-1})\right]
    =2\sum_{\substack{h\in G\\h \neq 1_G}} \left[C(h)+C(h^{-1})\right]
    =4\sum_{\substack{h\in G\\h \neq 1_G}} C(h),
\end{align*}
so
$$
c(\gamma)=\sup_{g \in G} \sum_{h \in G} \rho_{gh}(\gamma) \leq 4\sum_{\substack{h\in G\\h \neq 1_G}} C(h).
$$ 

Finally, if we consider the potential $\beta \phi$ for $\beta > 0$, then by linearity, we have 
\begin{equation*}
c(\gamma^{\beta \phi})=\sup_{g \in G} \sum_{h \in G} \rho_{gh}(\gamma^{\beta \phi}) \leq 4\beta\sum_{\substack{h\in G\\h \neq 1_G}} C(h),    
\end{equation*}
where $\gamma^{\beta \phi}$ is the specification given by equation \eqref{specification} for the potential $\beta \phi$. Thus, if 
$$
\beta < \left(4\sum_{\substack{h\in G\\h \neq 1_G}} C(h)\right)^{-1},
$$
Dobrushin's condition is satisfied and, by Theorem \ref{thm:Dobrushin-uniqueness}, we have at most one DLR measure for the potential $\beta \phi$. Furthermore, if $\beta > 0$, then the set of $G$-invariant DLR measures for $\beta \phi$ is non-empty, so that we can guarantee that if $\beta \in \left(0, \frac{1}{4\sum_{h \neq 1_G} C(h)}\right)$, there exists exactly one DLR measure for $\beta \phi$.

\section*{Acknowledgements}

Elmer R. Beltrán would like to thank to the fellow program Fondo Postdoctorado Universidad Católica del Norte No 0001, 2020. Rodrigo Bissacot is supported by CNPq grants 312294/2018-2 and 408851/2018-0, by FAPESP grant 16/25053-8, and by the University Center of Excellence ``Dynamics, Mathematical Analysis and Artificial Intelligence'', at the Nicolaus Copernicus University. Luísa Borsato is supported by grants
2018/21067-0 and 2019/08349-9, São Paulo Research Foundation (FAPESP). Raimundo Briceño would like to acknowledge the support of ANID/FONDECYT de Iniciación en Investigación 11200892.

\bibliographystyle{abbrv}
\bibliography{references}

\begin{thebibliography}{10}

\bibitem{aaronson2007exchangeable}
J.~Aaronson and H.~Nakada.
\newblock Exchangeable, {G}ibbs and equilibrium measures for {M}arkov
  subshifts.
\newblock {\em Ergod. Theory Dyn. Syst.}, 27(2):321--339, 2007.

\bibitem{alpeev2016entropy}
A.~Alpeev.
\newblock The entropy of {G}ibbs measures on sofic groups.
\newblock {\em J. Math. Sci.}, 215(6):649--659, 2016.

\bibitem{baladi1991gibbs}
V.~Baladi.
\newblock Gibbs states and equilibrium states for finitely presented dynamical
  systems.
\newblock {\em J. Stat. Phys.}, 62(1-2):239--256, 1991.

\bibitem{barbieri2020equivalence}
S.~Barbieri, R.~G{\'o}mez, B.~Marcus, and S.~Taati.
\newblock Equivalence of relative {G}ibbs and relative equilibrium measures for
  actions of countable amenable groups.
\newblock {\em Nonlinearity}, 33(5):2409, 2020.

\bibitem{barbieri2023lanford}
S.~Barbieri and T.~Meyerovitch.
\newblock The {L}anford--{R}uelle theorem for actions of sofic groups.
\newblock {\em Trans. Am. Math. Soc.}, 376(02):1299--1342, 2023.

\bibitem{Beltran_2021}
E.~R. Beltr\'an, R.~Bissacot, and E.~O. Endo.
\newblock Infinite {DLR} measures and volume-type phase transitions on
  countable {M}arkov shifts.
\newblock {\em Nonlinearity}, 34(7):4819, 2021.

\bibitem{berghout2019relation}
S.~Berghout, R.~Fern{\'a}ndez, and E.~Verbitskiy.
\newblock On the relation between {Gibbs} and $g$-measures.
\newblock {\em Ergod. Theory Dyn. Syst.}, 39(12):3224--3249, 2019.

\bibitem{billingsley2013convergence}
P.~Billingsley.
\newblock {\em Convergence of probability measures}.
\newblock John Wiley \& Sons, 2013.

\bibitem{Bissavanenter}
R.~Bissacot, E.~O. Endo, A.~C.~D. van Enter, and A.~Le~Ny.
\newblock {Entropic repulsion and lack of the $g$-measure property for Dyson
  models}.
\newblock {\em Commun. Math. Phys.}, 363(3):767--788, 2018.

\bibitem{boltzmann1902leccons}
L.~Boltzmann and M.~Brillouin.
\newblock {\em Le{\c{c}}ons sur la th{\'e}orie des gaz}, volume~2.
\newblock Gauthier-Villars, 1902.

\bibitem{borsato2022thesis}
L.~Borsato.
\newblock {\em Strong Law of Large Numbers for Bernoulli sequences and Gibbs
  measures on subshifts for finite and infinite alphabets}.
\newblock PhD thesis, University of S{\~{a}}o Paulo, 2022.

\bibitem{borsato2020dobrushin}
L.~Borsato and S.~MacDonald.
\newblock A {D}obrushin-{L}anford-{R}uelle theorem for irreducible sofic
  shifts, 2020.
\newblock Preprint, arXiv:2007.05862 [math.DS].

\bibitem{borsato2021conformal}
L.~Borsato and S.~MacDonald.
\newblock Conformal measures and the {D}obrushin-{L}anford-{R}uelle equations.
\newblock {\em Proc. Am. Math. Soc.}, 149(10):4355--4369, 2021.

\bibitem{bowen2008equilibrium}
R.~E. Bowen.
\newblock {\em Equilibrium states and the ergodic theory of {A}nosov
  diffeomorphisms}, volume 470.
\newblock Springer Science \& Business Media, 2008.

\bibitem{briceno2021kieffer}
R.~Briceño.
\newblock Kieffer-pinsker type formulas for {G}ibbs measures on sofic groups,
  2021.
\newblock Preprint, arXiv:2108.06053 [math.DS].

\bibitem{bufetov2011}
A.~I. Bufetov.
\newblock Pressure and equilibrium measures for actions of amenable groups on
  the space of configurations.
\newblock {\em Sb. Math.}, 202(3):341, 2011.

\bibitem{capocaccia1976definition}
D.~Capocaccia.
\newblock A definition of {G}ibbs state for a compact set with $\mathbb{Z}^v$
  action.
\newblock {\em Commun. Math. Phys.}, 48(1):85--88, 1976.

\bibitem{chazottes2004entropy}
J.~R. Chazottes and E.~Ugalde.
\newblock Entropy estimation and fluctuations of hitting and recurrence times
  for {G}ibbsian sources.
\newblock {\em Discrete Contin. Dyn. Syst. - B}, 5(3):565--586, 2005.

\bibitem{cioletti2020ruelle}
L.~Cioletti, A.~O. Lopes, and M.~Stadlbauer.
\newblock Ruelle operator for continuous potentials and {DLR-G}ibbs measures.
\newblock {\em Discrete Contin. Dyn. Syst.}, 40(8):4625, 2020.

\bibitem{denker-urbanski-1991-conformal}
M.~Denker and M.~Urba{\'n}ski.
\newblock On the existence of conformal measures.
\newblock {\em Trans. Am. Math. Soc.}, 328(2), 1991.

\bibitem{dobruschin1968description}
P.~L. Dobruschin.
\newblock The description of a random field by means of conditional
  probabilities and conditions of its regularity.
\newblock {\em Theory Probab. its Appl.}, 13(2):197--224, 1968.

\bibitem{dobrushin1968problem}
R.~L. Dobrushin.
\newblock The problem of uniqueness of a gibbsian random field and the problem
  of phase transitions.
\newblock {\em Funct. Anal. its Appl.}, 2(4):302--312, 1968.

\bibitem{dobrushin1970conditional}
R.~L. Dobrushin.
\newblock {Gibbsian random fields for lattice systems with pairwise
  interactions}.
\newblock {\em Funct. Anal. its Appl.}, 2(4):292--301, 1969.

\bibitem{downarowicz2011entropy}
T.~Downarowicz.
\newblock {\em Entropy in dynamical systems}, volume~18.
\newblock Cambridge University Press, 2011.

\bibitem{downarowicz2016shearer}
T.~Downarowicz, B.~Frej, and P.-P. Romagnoli.
\newblock Shearer’s inequality and infimum rule for {S}hannon entropy and
  topological entropy.
\newblock {\em Dynamics and Numbers}, pages 63--75, 2016.

\bibitem{downarowicz2019tilings}
T.~Downarowicz, D.~Huczek, and G.~Zhang.
\newblock Tilings of amenable groups.
\newblock {\em J. für die Reine und Angew. Math.}, 2019(747):277--298, 2019.

\bibitem{duminil2017lectures}
H.~Duminil-Copin.
\newblock Lectures on the {I}sing and {P}otts models on the hypercubic lattice.
\newblock In {\em PIMS-CRM Summer School in Probability}, pages 35--161.
  Springer, 2017.

\bibitem{10.1214/ECP.v16-1681}
R.~Fern\'andez, S.~Gallo, and G.~Maillard.
\newblock {Regular $g$-measures are not always Gibbsian}.
\newblock {\em Electron. Commun. Probab.}, 16(none):732--740, 2011.

\bibitem{friedli2017statistical}
S.~Friedli and Y.~Velenik.
\newblock {\em Statistical mechanics of lattice systems: a concrete
  mathematical introduction}.
\newblock Cambridge University Press, 2017.

\bibitem{georgii2011gibbs}
H.-O. Georgii.
\newblock {\em Gibbs measures and phase transitions}, volume~9.
\newblock Walter de Gruyter, 2011.

\bibitem{georgii2001random}
H.-O. Georgii, O.~H{\"a}ggstr{\"o}m, and C.~Maes.
\newblock The random geometry of equilibrium phases.
\newblock In {\em Phase transitions and critical phenomena}, volume~18, pages
  1--142. Elsevier, 2001.

\bibitem{gibbs1902elementary}
J.~W. Gibbs.
\newblock {\em Elementary principles in {S}tatistical {M}echanics: developed
  with especial reference to the rational foundations of thermodynamics}.
\newblock C. Scribner's sons, 1902.

\bibitem{gurevich1970shift}
B.~M. Gurevich.
\newblock Shift entropy and {M}arkov measures in the space of paths of a
  countable graph.
\newblock {\em Dokl. Akad. Nauk}, 192(5):963--965, 1970.

\bibitem{gurevich1998thermodynamic}
B.~M. Gurevich and S.~V. Savchenko.
\newblock Thermodynamic formalism for countable symbolic {M}arkov chains.
\newblock {\em Russ. Math. Surv.}, 53(2):245, 1998.

\bibitem{gurevich2007breiman}
B.~M. Gurevich and A.~A. Tempelman.
\newblock {A Breiman type theorem for {G}ibbs measures}.
\newblock {\em J. Dyn. Control Syst.}, 13(3):363--371, 2007.

\bibitem{haydn1992equivalence}
N.~Haydn and D.~Ruelle.
\newblock Equivalence of {G}ibbs and equilibrium states for homeomorphisms
  satisfying expansiveness and specification.
\newblock {\em Commun. Math. Phys.}, 148(1):155--167, 1992.

\bibitem{keane1972strongly}
M.~Keane.
\newblock Strongly mixing $g$-measures.
\newblock {\em Invent. Math.}, 16(4):309--324, 1972.

\bibitem{keller1998equilibrium}
G.~Keller.
\newblock {\em Equilibrium {S}tates in {E}rgodic {T}heory}, volume~42.
\newblock Cambridge University Press, 1998.

\bibitem{kerr2016ergodic}
D.~Kerr and H.~Li.
\newblock Ergodic theory.
\newblock {\em Springer Monographs in Mathematics. Springer, Cham}, 2016.

\bibitem{kimura-2015-thesis}
B.~Kimura.
\newblock Gibbs measures on subshifts.
\newblock Master's thesis, University of S{\~{a}}o Paulo, 2015.

\bibitem{lanford1969observables}
O.~Lanford and D.~Ruelle.
\newblock Observables at infinity and states with short range correlations in
  statistical mechanics.
\newblock {\em Commun. Math. Phys.}, 13(3):194--215, 1969.

\bibitem{loh2017geometric}
C.~L{\"o}h.
\newblock {\em Geometric {G}roup {T}heory}.
\newblock Springer, 2017.

\bibitem{mauldin2001gibbs}
R.~D. Mauldin and M.~Urba{\'n}ski.
\newblock Gibbs states on the symbolic space over an infinite alphabet.
\newblock {\em Isr. J. Math.}, 125:93--130, 2001.

\bibitem{meyerovitch-2013-gibbs-eqm}
T.~Meyerovitch.
\newblock Gibbs and equilibrium measures for some families of subshifts.
\newblock {\em Ergod. Theory Dyn. Syst.}, 33:934--953, 2013.

\bibitem{muir2011gibbs}
S.~Muir.
\newblock {\em Gibbs/equilibrium measures for functions of multidimensional
  shifts with countable alphabets}.
\newblock PhD thesis, University of North Texas, 2011.

\bibitem{muir2011paper}
S.~Muir.
\newblock A new characterization of {G}ibbs measures on
  $\mathbb{N}^{\mathbbm{z}^d}$.
\newblock {\em Nonlinearity}, 24(10):2933, 2011.

\bibitem{munkres2018elements}
J.~R. Munkres.
\newblock {\em Elements of algebraic topology}.
\newblock CRC Press, 2018.

\bibitem{ornstein1987}
D.~S. Ornstein and B.~Weiss.
\newblock Entropy and isomorphism theorems for actions of amenable groups.
\newblock {\em J. Anal. Math.}, 48(1):1--141, 1987.

\bibitem{petersen1997symmetric}
K.~Petersen and K.~Schmidt.
\newblock {Symmetric {G}ibbs measures}.
\newblock {\em Trans. Am. Math. Soc.}, 349(7):2775--2811, 1997.

\bibitem{pfister}
C.-E. Pfister.
\newblock Gibbs measures on compact ultra metric spaces, 2022.
\newblock Preprint, arXiv:2202.06802 [math.DS].

\bibitem{rassoul2015course}
F.~Rassoul-Agha and T.~Sepp{\"a}l{\"a}inen.
\newblock {\em A course on large deviations with an introduction to Gibbs
  measures}, volume 162.
\newblock American Mathematical Soc., 2015.

\bibitem{ruelle-2004-thermo}
D.~Ruelle.
\newblock {\em Thermodynamic formalism: the mathematical structures of
  equilibrium statistical mechanics}.
\newblock Cambridge, 2nd edition, 2004.

\bibitem{sarig1999thermodynamic}
O.~M. Sarig.
\newblock Thermodynamic formalism for countable {M}arkov shifts.
\newblock {\em Ergod. Theory Dyn. Syst.}, 19(6):1565--1593, 1999.

\bibitem{sarig2009notes}
O.~M. Sarig.
\newblock Lecture notes on thermodynamic formalism for topological {M}arkov
  shifts.
\newblock {\em Penn State}, 2009.

\bibitem{shriver2020free}
C.~Shriver.
\newblock Free energy, {G}ibbs measures, and {G}lauber dynamics for
  nearest-neighbor interactions on trees, 2020.
\newblock Preprint, arXiv:2011.00653 [math.PR].

\bibitem{sinai1972gibbs}
Y.~G. Sinai.
\newblock Gibbs measures in ergodic theory.
\newblock {\em Russ. Math. Surv.}, 27(4):21, 1972.

\bibitem{tempelman2013ergodic}
A.~A. Tempelman.
\newblock {\em Ergodic theorems for group actions: Informational and
  Thermodynamical Aspects}, volume~78.
\newblock Springer Science \& Business Media, 2013.

\bibitem{varandas2017weak}
P.~Varandas and Y.~Zhao.
\newblock Weak {G}ibbs measures: speed of convergence to entropy, topological
  and geometrical aspects.
\newblock {\em Ergod. Theory Dyn.}, 37(7):2313--2336, 2017.

\bibitem{walters1975}
P.~Walters.
\newblock Ruelle’s operator theorem and $g$-measures.
\newblock {\em Trans. Am. Math. Soc.}, 214:375--387, 1975.

\bibitem{yuri1998zeta}
M.~Yuri.
\newblock Zeta functions for certain non-hyperbolic systems and topological
  {M}arkov approximations.
\newblock {\em Ergod. Theory Dyn. Syst.}, 18(6):1589--1612, 1998.

\end{thebibliography}

\end{document}